\documentclass[a4paper,english,11pt]{article}
\usepackage{pdfpages}
\usepackage[T1]{fontenc}
\usepackage[utf8]{inputenc}
\usepackage{babel}
\usepackage{amsmath}
\usepackage{amssymb}
\usepackage{amsthm}
\usepackage{textcomp}
\usepackage{enumitem}
\usepackage{dsfont}
\usepackage{mathrsfs}
\usepackage{cite}
\usepackage{color}
\usepackage{breqn}
\usepackage[colorlinks=true,urlcolor=blue,linkcolor=blue]{hyperref}

\newcommand{\R}{\mathbb{R}}

\newcommand{\N}{\mathbb{N}}
\newcommand{\C}{\mathbb{C}}
\renewcommand{\S}{\mathbb{S}}
\newcommand{\T}{\mathbb{T}}
\newcommand{\F}{\mathscr{F}}
\newcommand{\holo}{\textnormal{Hol}}

\newcommand{\heis}{\mathbb{H}^1}
\newcommand{\cayley}{\mathcal{C}}
\newcommand{\hlapl}{\Delta_{\heis}}
\newcommand{\Dlapl}{\mathcal{D}}
\newcommand{\Jc}{J_\mathcal{C}}
\newcommand{\harm}{\textnormal{Ha}}
\newcommand{\id}{\textnormal{Id}}
\newcommand{\sh}{\textnormal{sinh}}
\newcommand{\ch}{\textnormal{cosh}}
\newcommand{\classeC}{\mathcal{C}}
\renewcommand{\L}{\mathcal{L}}
\renewcommand{\Im}{\textnormal{Im}}
\renewcommand{\Re}{\textnormal{Re}}
\newcommand{\un}{\mathds{1}}
\newcommand{\hardy}{\mathcal{H}}
\newcommand{\scw}{\mathscr{S}}
\newcommand{\half}{\frac{1}{2}}
\newcommand{\V}{\mathcal{V}}
\newcommand{\Qbeta}{\mathcal{Q}_{\beta}}
\newcommand{\Qplus}{\mathcal{Q}_{+}}
\newcommand{\Res}{\textnormal{Res}}
\newcommand{\Vect}{\textnormal{Vect}}

\renewcommand\d{\,{\mathrm d}}
\newcommand\e{\,{\mathrm e}}
\newcommand{\gdO}{{\mathcal O}}

\newcommand{\longrightarroww}[2] {\mathop{\longrightarrow}\limits_{#1}^{#2}}

\newtheorem{mydef}{Definition}[section]
\newtheorem{thm}[mydef]{Theorem}
\newtheorem{lem}[mydef]{Lemma}
\newtheorem{prop}[mydef]{Proposition}
\newtheorem{cor}[mydef]{Corollary}

\usepackage{authblk}

\title{On the radially symmetric traveling waves for the Schrödinger equation on the Heisenberg group.}
\author{Louise Gassot}
\date{}

\newcommand{\Addresses}{{
  \bigskip
  \footnotesize
\noindent
  \textsc{Louise Gassot}, \textit{Laboratoire de Mathématiques d'Orsay, Bâtiment 307, Université Paris-Sud (XI), 91405 Orsay Cedex, France}\par\nopagebreak
  \noindent
  \textit{E-mail address :} \texttt{louise.gassot@math.u-psud.fr}
}}

\begin{document}
\maketitle
\abstract{
We consider radial solutions to the cubic Schrödinger equation on the Heisenberg group
$$
i\partial_t u-\hlapl u=|u|^2u, \quad
	\hlapl=\frac{1}{4}(\partial_x^2+\partial_y^2)+(x^2+y^2)\partial_s^2, \quad
	(t,x,y,s)\in \R\times\heis.
$$
This equation is a model for totally non-dispersive evolution equations. We show existence of ground state traveling waves with speed $\beta\in (-1,1)$. When the speed $\beta$ is sufficiently close to $1$, we prove their uniqueness up to symmetries and their smoothness along the parameter $\beta$. The main ingredient is the emergence of a limiting system as $\beta$ tends to the limit $1$, for which we establish linear stability of the ground state traveling wave.
}

\tableofcontents

\section{Introduction}

\subsection{Dispersion for non-linear Schrödinger equations}

In this paper, we consider the cubic focusing Schrödinger equation on the Heisenberg group
\begin{equation}\label{eq:H}
i\partial_t u-\hlapl u=|u|^2u, \quad (t,x,y,s)\in\R\times\heis,
\end{equation}
where $\hlapl$ denotes the sub-Laplacian on the Heisenberg group. When the solution is radial, in the sense that it only depends on $t$, $|x+iy|$ and $s$, the sub-Laplacian writes
$$
\hlapl =\frac{1}{4}(\partial_x^2+\partial_y^2)+(x^2+y^2)\partial_s^2.
$$
The Heisenberg group is a typical case of geometry where dispersive properties of the non-linear Schrödinger equation disappear. Let us recall the motivation for this setting.

Fix a Riemannian manifold $M$, and denote by $\Delta$ the Laplace operator associated to the metric $g$ on $M$. As observed by Burq, Gérard and Tzvetkov \cite{Burq2005}, qualitative properties of the solutions to the non-linear Schrödinger equation
$$
i\partial_t u-\Delta u=|u|^2u, \quad (t,x)\in\R\times M
$$
are strongly influenced by the underlying geometry of the manifold $M$. When some loss of dispersion occurs, for example in the spherical geometry, a condition for well-posedness of the Cauchy problem in $H^s(M)$ is that $s$ must be larger than a critical parameter.

To take it further, on sub-Riemannian manifolds, Bahouri, Gérard and Xu \cite{BahouriGerardXu2000} noticed that the dispersion properties totally disappear for the sub-Laplacian on the Heisenberg group, leaving the existence and uniqueness of smooth global in time solutions as an open problem. In \cite{DelHierro2005}, Del Hierro analyzed the dispersion properties on H-type groups, proving sharp decay estimates for the Schrödinger equation depending on the dimension of the center of the group. More generally, Bahouri, Fermanian and Gallagher \cite{BahouriFermanianGallagher2016} proved optimal dispersive estimates on stratified Lie groups of step $2$ under some property of the canonical skew-symmetric form. In contrast, they also give a class of groups without this property displaying total lack of dispersion, which includes the Heisenberg group.

In this spirit, Gérard and Grellier introduced the cubic Szeg\H{o} equation on the torus \cite{GerardGrellier2008,GerardGrellier2010} as a simpler model of non-dispersive Hamiltonian equation in order to better understand the situation on the Heisenberg group. The cubic Szeg\H{o} equation was then studied on the real line by Pocovnicu \cite{Pocovnicu2011}, where it writes
$$
i\partial_t u=\Pi(|u|^2u), \quad (t,x)\in\R\times\R,
$$
$\Pi:L^2(\R)\to L^2_+(\R)$ being the Szeg\H{o} projector onto the space $L^2_+(\R)$ of fonctions in $L^2(\R)$ with non-negative frequencies. The cubic Szeg\H{o} equation displays a strong link with the mass-critical half-wave equation on the torus \cite{GerardGrellier2012} resp. on the real line \cite{KriegerLenzmannRaphael2013}. On the real line, the cubic focusing half-wave equation writes
$$
i\partial_t u+|D|u=|u|^2u, \quad (t,x)\in\R\times \R,
$$
where $D=-i\partial_x$, $\widehat{|D|f}(\xi)=|\xi|\widehat{f}(\xi)$. Some of the interactions between the Szeg\H{o} equation and the half-wave equation will be detailed below, because they can be transferred  to the setting of the Heisenberg group.

\subsection{Traveling waves and limiting profiles}

Constructing traveling wave solutions which are weak global solutions in the energy space can be obtained by a classical variational argument. For example, this technique was used to study the famous focusing mass-critical NLS problem
$$
i\partial_t u-\Delta u=|u|^{\frac{4}{n}}u, \quad (t,x)\in\R\times\R^n.
$$
From Weinstein's work \cite{Weinstein1983}, the existence of a ground state positive solution $Q\in H^1(\R^n)$ to
$$
\Delta Q-Q+Q^{1+\frac{4}{n}}=0,
$$
leads to a criterion for global existence of solutions in $H^1(\R^n)$. The uniqueness of this ground state (up to symmetries) holds \cite{GidasNiNirenberg1979,Kwong1989}.

Concerning the half-wave equation, the Cauchy problem is locally well-posed in the energy space $H^{\half}(\R)$\cite{GerardGrellier2012,KriegerLenzmannRaphael2013}. Moreover, one also gets a global existence criterion, derived from the existence of a 
unique \cite{FrankLenzmann2013} ground state positive solution $Q\in H^{\half}(\R)$ to
$$
|D|Q+Q-Q^3=0.
$$
Contrary to the mass-critical Schrödinger equation on $\R^n$, the half-wave equation admits mass-subcritical traveling waves with speed $\beta\in(-1,1)$ (see Krieger, Lenzmann and Raphaël \cite{KriegerLenzmannRaphael2013})
$$
u(t,x)=Q_\beta\Big(\frac{x+\beta t}{1-\beta}\Big)\e^{-it}.
$$
The profile $Q_\beta$ is a solution to
$$
\frac{|D|-\beta D}{1-\beta}Q_\beta+Q_\beta=|Q_\beta|^2Q_\beta.
$$
Moreover, it satisfies
$$
\lim_{\beta\to0}\|Q_\beta-Q\|_{H^\half(\R)}=0 \quad \textnormal{and}\quad \|Q_\beta\|_{L^2(\R)}<\|Q\|_{L^2(\R)}.
$$
While the existence of the profiles $Q_\beta$ follows from a standard variational argument, their uniqueness is more delicate to prove. This can be done through the study of the  photonic limit $\beta\to1$ as follows. It has been shown \cite{GerardLenzmannPocovnicuRaphael2018} that the traveling waves converge as $\beta$ tends to $1$ to a solution of the cubic Szeg\H{o} equation. More precisely, $(Q_\beta)_\beta$ converges in $H^\half(\R)$ to a profile $Q_+$, which is a ground state solution to
$$
DQ_++Q_+=\Pi(|Q_+|^2Q_+), \quad D=-i\partial_x.
$$
From $Q_+$, we recover a traveling wave solution to the cubic Szeg\H{o} equation by setting
$$
u(t,x)=Q_+(x-t)\e^{-it}.
$$
But Pocovnicu showed \cite{Pocovnicu2011} that the traveling waves $u$ are unique up to symmetries, and that $Q_+$ must have the form
$$
Q_+(x)=\frac{2}{2x+i}.
$$
Moreover, the linearized operator around $Q_+$ is coercive \cite{Pocovnicu2012}, and in particular, the Szeg\H{o} profile is orbitally stable. Gérard, Lenzmann, Pocovnicu and Raphaël \cite{GerardLenzmannPocovnicuRaphael2018} deduced the invertibility of the linearized operator for the half-wave equation around the profiles $Q_\beta$ when $\beta$ is close enough to $1$, which leads to their uniqueness up to symmetries. This allowed them to define a smooth map of solutions $\beta\mapsto Q_\beta$ on a neighbourhood of $1$. 

On the Heisenberg group, one can also construct a family of traveling waves with speed $\beta\in(-1,1)$ under the form
$$
u(t,x,y,s)=Q_\beta\Big(\frac{x}{\sqrt{1-\beta}},\frac{y}{\sqrt{1-\beta}},\frac{s+\beta t}{1-\beta}\Big).
$$
The profile $Q_\beta$ satisfies the following stationary hypoelliptic equation
\begin{equation}\label{eq:Hbeta}
-\frac{\hlapl +\beta D_s}{1-\beta}Q_\beta=|Q_\beta|^2Q_\beta.
\end{equation}
There exist ground state solutions, constructed as optimizers for some Gagliardo-Nirenberg inequalities derived from the Folland-Stein embedding $\dot{H}^1(\heis)\hookrightarrow L^4(\heis)$ \cite{Folland1974}. The proof of existence relies on a concentration-compactness argument, which first appeared in the work of Cazenave and Lions \cite{Cazenave1982} and was refined into a profile decomposition theorem on $\R^n$ by Gérard \cite{Gerard1998}. The profile decomposition theorem was then adapted to the Heisenberg group by Benameur~\cite{Benameur2008}.

Our purpose is to show the  uniqueness of the profiles $Q_\beta$ when their speed $\beta$ is close to $1$ up to some symmetries. Following the strategy deployed on the half-wave equation, we derive a limiting system in the photonic limit $\beta\to 1$. We then determine all ground states solutions to the limiting system and prove their linear stability. From the linear stability of the limiting ground states, we recover the uniqueness of the profiles $Q_\beta$ up to symmetries when their speed $\beta$ is close to $1$.

\subsection{Main results}

The Schrödinger equation on the Heisenberg group \eqref{eq:H} enjoys the following symmetries~: if $u$ is a solution, then
\begin{itemize}
\item for all $s_0\in\R$, $(t,x,y,s)\mapsto u(t,x,y,s+s_0)$ is a solution (translation in $s$);
\item for all $\theta\in\T$, $(t,x,y,s)\mapsto \e^{i\theta}u(t,x,y,s)$ is a solution (phase multiplication);
\item for all $\lambda\in\R$, $(t,x,y,s)\mapsto \lambda u(\lambda^2t,\lambda x,\lambda y,\lambda^2s)$ is a solution (scaling).
\end{itemize}

Our main result is the uniqueness of the ground states $Q_\beta$ when $\beta$ is close to $1$.

\begin{thm}\label{thm:main_intro}
There exists $\beta_*\in(0,1)$ such that the following holds. For all $\beta\in(\beta_*,1)$, there is a unique ground state up to symmetries  to \eqref{eq:Hbeta}
$$
-\frac{\hlapl +\beta D_s}{1-\beta}Q_\beta=|Q_\beta|^2Q_\beta.
$$
Denote by $Q_\beta$ this ground state, then the set of all ground state solutions of the above equation can be described as
$$
\{T_{s_0,\theta,\alpha}Q_\beta:(x,y,s)\mapsto \e^{i\theta}\alpha Q_\beta(\alpha x,\alpha y,\alpha^2 (s+s_0)); (s_0,\theta,\alpha)\in \R\times\T\times\R_+^* \}.
$$
For $\beta\in(\beta_*,1)$, $Q_\beta$ can be chosen such that it tends as $\beta$ tends to $1$ to the profile
$$
Q_+:(x,y,s)\in\heis\mapsto \frac{\sqrt{2}i}{s+i(x^2+y^2)+i},
$$
and so that the map $\beta\in(\beta_*,1)\mapsto Q_\beta\in\dot{H}^1(\heis)$ is smooth. Moreover, for all $\gamma\in(0,\frac{1}{4})$ and all $k\in[1,+\infty)$, $Q_\beta$ lies in $\dot{H}^k(\heis)$, and as $\beta$ tends to $1$,
$$
\|Q_\beta-Q_+\|_{\dot{H}^k(\heis)}=\gdO((1-\beta)^\gamma).
$$
\end{thm}

We refer to Theorem \ref{thm:description_Qbeta} for a more precise statement.

We now briefly present the emergence of the profile $Q_+$ as a ground state solution to a limiting system, and the key ingredient for the proof of Theorem \ref{thm:main_intro} which relies on the study of the limiting geometry.

We are interested in solutions with values in the homogeneous energy space $\dot{H}^1(\heis)$, which is a Hilbert space endowed with the real scalar product
$$
(u,v)_{\dot{H}^1(\heis)}=\Re\left(\int_{\heis}-\hlapl u(x,y,s)\overline{v(x,y,s)}\d x\d y\d s\right).
$$
For $u\in\dot{H}^{-1}(\heis)$ and $v\in \dot{H}^1(\heis)$, we will also make use of the duality product
$$
(u,v)_{\dot{H}^{-1}(\heis)\times\dot{H}^1(\heis)}
	=\Re\left(\int_{\heis}u(x,y,s)\overline{v(x,y,s)}\d x\d y\d s\right).
$$

Up to the three symmetries (translation, phase multiplication, scaling), one can show the convergence as $\beta$ tends to $1$ of the profiles $Q_\beta$  to some profile $Q_+$ in $\dot{H}^1(\heis)$. Then, $Q_+$ is a ground state solution to
\begin{equation}\label{eq:Hplus}
D_s Q_+=\Pi_0^+(|Q_+|^2Q_+), \quad D_s=-i\partial_s.
\end{equation}
The operator $\Pi_0^+$ is an orthogonal projector from $L^2(\heis)$ onto a subspace $L^2(\heis)\cap V_0^+$, which will be defined in part \ref{subsection:hermite}. In order to study this projector and the space $L^2(\heis)\cap V_0^+$, we introduce a link between the space $L^2(\heis)\cap V_0^+$ and the Bergman space $L^2(\C_+)\cap\holo(\C_+)$ on the complex upper half-plane \cite{bergman_projectors}. The orthogonal projection $\Pi_0^+$ from $L^2(\heis)$ onto $L^2(\heis)\cap V_0^+$ then matches with a Bergman projector. This projection is a simplification of the usual Cauchy-Szeg\H{o} projector for the Heisenberg group in the radial case.

A salutary fact is that the profile $Q_+$ can be determined explicitly, and is unique up to symmetry~:
$$
Q_+(x,y,s)=\frac{\sqrt{2}i}{s+i(x^2+y^2)+i}.
$$

Our key result is the coercivity of the linearized operator $\L$ around $Q_+$ on the orthogonal of a finite-dimensional manifold in some subspace $\dot{H}^1(\heis)\cap V_0^+$ of $\dot{H}^1(\heis)$ (cf. part \ref{subsection:hermite}). On $\dot{H}^1(\heis)\cap V_0^+$, the linearized operator $\L$ around $Q_+$ is defined by
$$
\L h=D_s h-2\Pi_0^+(|Q_+|^2h)-\Pi_0^+(Q_+^2\overline{h}).
$$
\begin{thm}\label{thm:coercivityL_intro}
For some constant $c>0$, the following holds. Let $h\in \dot{H}^1(\heis)\cap V_0^+$, and suppose $h$ orthogonal to the directions $Q_+,iQ_+,\partial_sQ_+$ and $i\partial_sQ_+$ in the Hilbert space $\dot{H}^1(\heis)$. Then
$$
(\L h,h)_{\dot{H}^{-1}(\heis)\times\dot{H}^1(\heis)}
	\geq c\|h\|_{\dot{H}^1(\heis)}^2.
$$
\end{thm}
In particular, the linearized operator $\L$ is non degenerate, in the sense that its kernel is composed only of three directions coming from the three symmetries of the equation~:
$$
\textnormal{Ker}(\L)=\Vect_{\R}(\partial_sQ_+,iQ_+,Q_++2i\partial_sQ_+).
$$

Following the approach employed in the study of the half-wave equation \cite{GerardLenzmannPocovnicuRaphael2018}, one can then prove the invertibility of the linearized operators $\L_{Q_\beta}$ for the Schrödinger equation around the profiles $Q_\beta$ for $\beta$ close enough to $1$. In order to do so, we need to combine the above coercivity result with some regularity estimates and decay properties for $Q_\beta$. This enables us to achieve our goal, which is the uniqueness of these profiles up to symmetries for $\beta$ close to $1$.

\subsection{Stereographic projection and Cayley transform}

Conclusive information on the linearized operator $\L$ around $Q_+$ is not easy to obtain directly. Indeed, the operator $\L$ is self-adjoint acting on $L^2(\heis)$, but the space we consider is the Hilbert space $\dot{H}^1(\heis)$. In order to get a coercivity estimate, we rely on a conformal invariance between the Heisenberg group $\heis$ and the CR sphere $\S^3$ in $\C^2$ called the Cayley transform
$$
\begin{matrix}
\cayley: &\heis & \to & \S^3\setminus (0,-1)\\
~ & (w,s) & \mapsto & (\frac{2w}{1+|w|^2+is},\frac{1-|w|^2-is}{1+|w|^2+is})
\end{matrix},
$$
where $\heis$ is here parametrized by the complex number $w=x+iy$ and by $s$.

This transformation links estimates for the linearized operator $\L$ to the spectrum of the sub-Laplacian on the CR sphere, which is explicit \cite{Stanton1989}. Potential negative eigenvalues are discarded by the orthogonality conditions from Theorem \ref{thm:coercivityL_intro}. This latter step follows from technical but direct calculations.

For the $n$-dimensional Heisenberg group $\mathbb{H}^n$, the Cayley transform gives an equivalence between $\mathbb{H}^n$ and the CR sphere $\S^{2n+1}$ in $\C^n$. This transform is the counterpart of the stereographic projection, which links the space $\R^n$ with the euclidean sphere $\S^n$ in $\R^{n+1}$. Both transformations have been a useful tool in the study of fractional Folland-Stein inequalities on $\mathbb{H}^n$, resp. fractional Sobolev inequalities in $\R^n$, as we will now recall.

On the space $\R^n$, Lieb \cite{Lieb1983} characterized all optimizers for the fractional Sobolev embeddings $\dot{H}^k(\R^n)\hookrightarrow L^p(\R^n)$, $0<k<\frac{n}{2}$, $p=\frac{2n}{n-2k}$, as the set of functions which, up to translation, dilation and multiplication by a non-zero constant, coincide with
$$
U(x)=\frac{1}{(1+|x|^2)^{\frac{n-2k}{2}}},\quad U\in \dot{H}^k(\R^n).
$$
The stereographic projection appears in Lieb's paper in order to show that these functions are actually optimizers. The formula for $U$ was first established with different methods for $k=2$ and $n=3$ by Rosen \cite{Rosen1971}, and then for $k=1$ and arbitrary $n$ by Aubin \cite{Aubin1976} and Talenti \cite{Talenti1976}.

Chen, Frank and Weth \cite{ChenFrankWeth2013} showed a quadratic estimate for the remainder terms for the equivalent fractional Hardy-Littlewood-Sobolev inequalities. In their proof, the stereographic projection enables them to transfer the second order term in the Taylor expansion to the unit sphere $\S^n$, and give a simpler form to the eigenvalue problem.

On the Heisenberg group $\mathbb{H}^n$, Frank and Lieb \cite{FrankLieb2010} determined the optimizers for the fractional Folland-Stein embeddings $\dot{H}^k(\mathbb{H}^n) \hookrightarrow L^p(\mathbb{H}^n)$, $0<k<\frac{Q}{2}$, $p=\frac{2Q}{Q-2k}$, $Q=2n+2$. These optimizers are the functions equal, up to translations, dilations and multiplication by a constant, to
$$
H(u)=\frac{1}{((1+\|w\|^2)^2+\|s\|^2)^{\frac{Q-2k}{4}}},\quad H\in\dot{H}^k(\mathbb{H}^n).
$$
Here, the notation $u=(w,s)$ uses the identification of $\mathbb{H}^n$ with $\C^n\times\R^n$. Using the Cayley transform, both problems of characterizing the optimizers \cite{FrankLieb2010} and studying the remainder term (see Liu and Zhang \cite{LiuZhang2015}) are carried to the complex sphere $\S^{2n+1}$. When $k=1$, the optimizers were first determined by Jerison and Lee \cite{JerisonLee1988}, who already made use of the Cayley transform. One can notice that fixing $n=k=1$, $u=(x,y,s)\in\heis$, we get
$$
H(u)=\frac{1}{((1+x^2+y^2)^2+s^2)^{\half}}.
$$
Therefore, up to multiplication by a constant, $H$ coincides  with $|Q_+|$, where $Q_+$ is the ground state we are interested in. In fact, $Q_+$ is an optimizer for the Folland-Stein inequality $\dot{H}^1(\heis)\hookrightarrow L^4(\heis)$ restricted to the subspace $\dot{H}^1(\heis)\cap V_0^+$.

\paragraph{Plan of the paper} The paper is organized as follows. In section \ref{section:existence_and_limit}, we prove the existence of the profiles $Q_\beta$ and their convergence to a ground state solution to the limiting system \eqref{eq:Hplus}. We then determine all the limiting profiles (part  \ref{subsection:ground_state_solutions_limit}), in particular, we show that they are unique up to symmetries. In section \ref{section:limiting_problem}, we focus on the linear stability of the limiting profile $Q_+$. After recalling some results about orthogonal projections on Bergman spaces (part \ref{subsection:bergman_spaces}) and about the spectrum of the sub-Laplacian on the CR sphere (part \ref{subsection:cayley_transform}), we prove the coercivity of the linearized operator around $Q_+$. Finally, in section \ref{section:uniqueness_schrodinger}, we retrieve the uniqueness of the profiles $Q_\beta$ up to symmetries for $\beta$ close to $1$. In order to do so, we first need to collect some regularity properties and decay estimates on the profiles $Q_\beta$, which come from the theory of elliptic and hypoelliptic equations (part \ref{subsection:regularity_decay}).

\paragraph{Acknowledgements} The author is grateful to her PhD advisor P. Gérard for introducing her to this problem and for his patient guidance. She also thanks F. Rousset and J. Sabin for enlightening discussions and references.

\section{Notation}

\subsection{The Heisenberg group}

Let us now recall some facts about the Heisenberg group. We identify the Heisenberg group $\heis$ with $\R^3$. The group multiplication is given by
$$
(x,y,s)\cdot(x',y',s')=(x+x',y+y',s+s'+2(x'y-xy')).
$$
The Lie algebra of left-invariant vector fields on $\heis$ is spanned by the vector fields $X=\partial_x+2y\partial_s$, $Y=\partial_y-2x\partial_s$ and $T=\partial_s=\frac{1}{4}[Y,X]$. The sub-Laplacian is defined as
\begin{align*}
\mathcal{L}_0 
	:=\frac{1}{4}(X^2+Y^2)
	=\frac{1}{4}(\partial_x^2+\partial_y^2)+(x^2+y^2)\partial_s^2+(y\partial_x -x\partial_y)\partial_s.
\end{align*}
When $u$ is a radial function, the sub-Laplacian coincides with the operator
$$
\hlapl :=\frac{1}{4}(\partial_x^2+\partial_y^2)+(x^2+y^2)\partial_s^2.
$$

The space $\heis$ is endowed with a smooth left invariant measure, the Haar measure, which in the coordinate system $(x,y,s)$ is the Lebesgue measure $\d\lambda_3(x,y,s)$. Sobolev spaces of positive order can then be constructed on $\heis$ from powers of the operator $-\hlapl$, for example, $\dot{H}^1(\heis)$ is the completion of the Schwarz space $\scw(\heis)$ for the norm
$$
\|u\|_{\dot{H}^1(\heis)}:=\|(-\hlapl)^{\frac{1}{2}}u\|_{L^2(\heis)}.
$$
The distance  between two points $(x,y,s)$ and $(x',y',s')$ in $\heis$ is defined as
$$
d((x,y,s),(x',y',s')):=\left(\left((x-x')^2+(y-y')^2\right)^2+\left(s-s'+2(x'y-xy')\right)^2\right)^{\frac{1}{4}}
$$
For convenience, the distance to the origin is denoted by
$$
\rho(x,y,s):=((x^2+y^2)^2+s^2)^{\frac{1}{4}}.
$$

\subsection{Decomposition along the Hermite functions}\label{subsection:hermite}

In order to study radial functions valued on the Heisenberg group $\heis$, it is convenient to use their decomposition along Hermite-type functions (see for example \cite{stein_harmonic}, Chapters 12 and 13). The Hermite functions
$$
h_m(x)=\frac{1}{\pi^{\frac{1}{4}}2^{\frac{m}{2}}(m!)^{\half}}(-1)^m\e^{\frac{x^2}{2}}\partial_x^m(\e^{-x^2}), \quad x\in\R, m\in\N,
$$
form an orthonormal basis of $L^2(\R)$. In $L^2(\R^2)$, the family of products of two Hermite functions $(h_m(x)h_p(y))_{m,p\in\N}$ diagonalizes the two-dimensional harmonic oscillator~: for all $m,p\in\N$,
$$
(-\Delta_{x,y}+x^2+y^2)h_m(x)h_p(y)=2(m+p+1)h_m(x)h_p(y).
$$

Given $u\in \scw(\heis)$, we will denote by $\widehat{u}$ its usual Fourier transform under the $s$ variable, with corresponding variable $\sigma$
$$
\widehat{u}(x,y,\sigma)=\frac{1}{\sqrt{2\pi}}\int_{\R}\e^{-is\sigma}u(x,y,s)\d s.
$$
For $m,p\in\N$, set $\widehat{h_{m,p}}(x,y,\sigma):=h_m(\sqrt{2|\sigma|}x)h_p(\sqrt{2|\sigma|}y)$. Then
$$
\widehat{\hlapl h_{m,p}}=-(m+p+1)|\sigma|\widehat{h_{m,p}}.
$$
Let $k\in\{-1,0,1\}$, and denote by $\dot{H}^k(\heis)\cap V_n^{\pm}$ the subspace of functions in $\dot{H}^k(\heis)$ spanned by $\{h_{m,p};\; m,p\in\N, m+p=n\}$. A function $u_n^{\pm}\in \dot{H}^k(\heis)$ belongs to $\dot{H}^k(\heis)\cap V_n^{\pm}$ if there exist functions $f_{m,p}^\pm$ such that
$$
\widehat{u_n^\pm}(x,y,\sigma)=\sum_{\substack{m,p\in\N;\\ m+p=n}}f_{m,p}^{\pm}(\sigma)\widehat{h_{m,p}}(x,y,\sigma)\un_{\sigma\gtrless 0}.
$$
For $u_n^\pm\in\dot{H}^k(\heis)\cap V_n^{\pm}$, the $\dot{H}^k$ norm of $u_n^\pm$ writes 
\begin{align*}
\|u_n^{\pm}\|_{\dot{H}^k(\heis)}^2
	&=\int_{\R_{\pm}}((n+1)|\sigma|)^k\int_{\R^2}|\widehat{u_n}(x,y,\sigma)|^2\d x\d y\d\sigma\\
	&=\sum_{\substack{m,p\in\N;\\ m+p=n}}\int_{\R_{\pm}}((n+1)|\sigma|)^k|f_{m,p}^{\pm}(\sigma)|^2\frac{\d\sigma}{2|\sigma|}.
\end{align*}
Any function $u\in\dot{H}^k(\heis)$ admits a decomposition along the orthogonal sum of the subspaces $\dot{H}^k(\heis)\cap V_n^{\pm}$. Let us write $u=\sum_{\pm}\sum_{n\in\N}u_n^{\pm}$ where $u_n^{\pm}\in \dot{H}^k(\heis)\cap V_n^{\pm}$ for all $(n,\pm)$. Then 
\begin{align*}
\|u\|_{\dot{H}^k(\heis)}^2
	&=\sum_{\pm}\sum_{n\in\N}\|u_n^{\pm}\|_{\dot{H}^k(\heis)}^2.
\end{align*}
Note that rotations of the $(x,y)$ variable commute with $-\hlapl$ so $u\in \dot{H}^k(\heis)$ is radial if and only if for all $(n,\pm)$, $u_n^\pm$ is radial.
Moreover, $u\in\dot{H}^k(\heis)$ belongs to $\dot{H}^{k}(\heis)\cap V_n^\pm$ if and only if $-\hlapl u$ belongs to $\dot{H}^{k-2}(\heis)\cap V_n^\pm$, and the same holds for $D_su$.

For $k=0$, we get an orthogonal decomposition of the space $L^2(\heis)$, and denote by $\Pi_n^{\pm}$ the associated orthogonal projectors.

The particular space $V_0^+$ will be especially interesting in our discussion below. This space is spanned by a unique radial function $h_0^+$, satisfying
$$
\widehat{h_0^+}(x,y,\sigma)=\frac{1}{\sqrt{\pi}}\e^{-(x^2+y^2)\sigma}\un_{\sigma\geq 0}.
$$
Set $u\in \dot{H}^k(\heis)\cap V_0^+$, then there exists $f$ such that
$$
\widehat{u}(x,y,s)=f(\sigma)\widehat{h_0^+}(x,y,\sigma),
$$
and in this case
$$
\|u\|_{\dot{H}^k(\heis)}^2=\int_{\R_+}|f(\sigma)|^2\frac{\d\sigma}{2\sigma^{1-k}}.
$$

\section{Existence of traveling waves and limiting profile}\label{section:existence_and_limit}

In this section, we prove the existence of ground states $Q_\beta$ for equation \eqref{eq:Hbeta} with speed $\beta\in(-1,1)$ (part \ref{subsection:existence_Qbeta}). Then, we show the convergence in $\dot{H}^1(\heis)$ of the profiles $Q_\beta$ to a limiting profile $Q_+$ as $\beta$ tends to $1$ (part \ref{subsection:limit}). The profile $Q_+$ is a ground state solution of equation \eqref{eq:Hplus}, which will determine explicitly in part \ref{subsection:ground_state_solutions_limit}.

\subsection{Existence of traveling waves with speed \texorpdfstring{$\beta\in(-1,1)$}{beta in (-1,1)}}\label{subsection:existence_Qbeta}

A family of traveling wave solutions to the Schrödinger equation on the Heisenberg group \eqref{eq:H} can be found under the form
$$
u(t,x,y,s)=Q_\beta\Big(\frac{x}{\sqrt{1-\beta}},\frac{y}{\sqrt{1-\beta}},\frac{s+\beta t}{1-\beta}\Big),
$$
$Q_\beta$ satisfying the equation
\begin{equation*}
-\frac{\hlapl +\beta D_s}{1-\beta}Q_\beta=|Q_\beta|^2Q_\beta.
\end{equation*}
The $Q_\beta$ are constructed as minimizers of some Gagliardo-Nirenberg inequalities. We will be adapting the proofs of Krieger, Lenzmann and Raphaël \cite{KriegerLenzmannRaphael2013} which concern the $L^2$-critical half-wave equation on the real line. Our starting point is the Folland-Stein embedding \cite{Folland1974}.

\begin{thm}[Folland-Stein]
Let $p\in(1,4)$ and set $p^*=\frac{4p}{4-p}$. Then there exists $C_p>0$ such that, for $u\in\classeC_c^{\infty}(\heis)$,
$$
\left(\int_{\heis}|u(x,y,s)|^{p^*}\d x\d y\d s\right)^{\frac{1}{p^*}}\leq C_p \left(\int_{\heis} |(-\hlapl)^{\half}u(x,y,s)|^p\d x\d y\d s\right)^{\frac{1}{p}}.
$$ 
\end{thm}

In particular, from the embedding $\dot{H}^1(\heis)\hookrightarrow L^4(\heis)$, we deduce some Gagliardo-Nirenberg inequalities.

\begin{prop}[Gagliardo-Nirenberg]
Set $\beta\in(-1,1)$. Then there exists some constant $C>0$ such that for every $u\in \dot{H}^1(\heis)$,
$$
\|u\|_{L^4(\heis)}^4\leq C (-(\hlapl +\beta D_s)u,u)_{\dot{H}^{-1}(\heis)\times\dot{H}^1(\heis)}^2.
$$
\end{prop}

\begin{proof}

Fix $u\in \dot{H}^1(\heis)$, and decompose $u$ along the spaces $V_n^+\cup V_n^-$~: $u=\sum_{n\in \N}u_n$, where $u_n=u_n^++u_n^-$. Then
$$
(-(\hlapl +\beta D_s)u,u)_{\dot{H}^{-1}(\heis)\times\dot{H}^1(\heis)}
	=\sum_{n\in\N} \int_{\R^3}((n+1)|\sigma|-\beta\sigma) |\widehat{u_n}(x,y,\sigma)|^2\d x\d y\d \sigma
$$
and
$$
\|u\|_{\dot{H}^1(\heis)}^2
	=\sum_{n\in\N}\int_{\R^3}(n+1)|\sigma| |\widehat{u_n}(x,y,\sigma)|^2\d x\d y\d \sigma.
$$
We deduce the equivalence of norms
\begin{equation}\label{eq:equivalence_norms}
(1-|\beta|)\|u\|_{\dot{H}^1(\heis)}^2\leq (-(\hlapl +\beta D_s)u,u)_{\dot{H}^{-1}(\heis)\times\dot{H}^1(\heis)} \leq (1+|\beta|)\|u\|_{\dot{H}^1(\heis)}^2.
\end{equation}
The result follows from the Folland-Stein embedding $\dot{H}^1(\heis)\hookrightarrow L^4(\heis)$. 
\end{proof}

From the Gagliardo-Nirenberg inequalities, one knows that the infimum over non-zero radial functions $u\in \dot{H}^{1}(\heis)$ of the functional
$$
J_\beta(u):=\frac{(-(\hlapl +\beta D_s)u,u)_{\dot{H}^{-1}(\heis)\times\dot{H}^1(\heis)}^2}{\|u\|_{L^4}^4}
$$
is positive. Let us denote by $I_\beta$ the minimal value of $J_\beta$. We want to show that it is attained by some $Q_\beta\in\dot{H}^1(\heis)$. We consider a minimizing sequence for $J_\beta$. Then this sequence converges to a minimizer for $J_\beta$ thanks to the following profile decomposition theorem.

\begin{mydef}
The couples of scalings and cores $((\tilde{h_i})_{i\in\N},(\tilde{s_{i}})_{i\in\N})$ and $(({h_i})_{i\in\N},({s_{i}})_{i\in\N})$ of $(\R_+^*)^{\N}\times\R^{\N}$ are said to be strange if
$$
\left(\Big|\log\Big(\frac{\tilde{h_n}}{h_n}\Big)\Big|\longrightarroww{n\to+\infty}{}+\infty\right)
\text{ or if }
\left((\tilde{h_n})_n=(h_n)_n\text{ and } \frac{|\tilde{s_n}-s_n|}{h_n^2}\longrightarroww{n\to+\infty}{}+\infty\right).
$$
\end{mydef}

\begin{thm}[Concentration-compactness]\label{thm:concentration_compactness}
Fix a bounded sequence $\underline{u}=(u_n)_{n\in\N}$ of radial functions in $\dot{H}^1(\heis)$. Then there exist a subsequence $(u_{n_i})_{i\in\N}$, of $\underline{u}$, and sequences of cores $(s_{n_i}^{(j)})_{i,j\in\N}\subset\R$, scalings $(h_{n_i}^{(j)})_{i,j\in\N}\subset\R$, and radial functions $(U^{(j)})_{j\in\N}\subset\dot{H}^1(\heis)$ such that~:
\begin{enumerate}
\item the couples $((h_{n_i}^{(j)})_{i},(s_{n_i}^{(j)})_{i})$, $j\in\N$, are pairwise strange ;
\item let
$$
r_{n_i}^{(l)}(x,y,s)=u_{n_i}(x,y,s)-\sum_{j=1}^l \frac{1}{h_{n_i}^{(j)}} U^{(j)}\left(\frac{x}{h_{n_i}^{(j)}},\frac{y}{h_{n_i}^{(j)}}, \frac{s-s_{n_i}^{(j)}}{(h_{n_i}^{(j)})^2}\right),
$$
then
$$
\lim_{l\to+\infty}\limsup_{i\to+\infty} \|r_{n_i}^{(l)}\|_{L^4(\heis)}=0.
$$
\end{enumerate}
Moreover, for all $l\geq 1$, one has the following orthogonality relations as $i$ goes to $+\infty$~:
\begin{align*}
\|u_{n_i}\|_{\dot{H}^1(\heis)}^2
	=\sum_{j=1}^l\|U^{(j)}\|_{\dot{H}^1(\heis)}^2
	+\|r_{n_i}^{(l)}\|_{\dot{H}^1(\heis)}^2
	+o(1),
\end{align*}
$$
(D_su_{n_i},u_{n_i})_{\dot{H}^{-1}(\heis)\times \dot{H}^1(\heis)}
	=\sum_{j=1}^l(D_sU^{(j)},U^{(j)})_{\dot{H}^{-1}(\heis)\times \dot{H}^1(\heis)}
	+(D_sr_{n_i}^{(l)},r_{n_i}^{(l)})_{\dot{H}^{-1}(\heis)\times \dot{H}^1(\heis)}+o(1),
$$
and
$$
\|u_{n_i}\|_{L^4(\heis)}^4 \longrightarroww{i\to+\infty}{} \sum_{j=1}^{+\infty}\|U^{(j)}\|_{L^4(\heis)}^4.
$$
\end{thm}

This result is an adaptation of a concentration-compactness argument due to Cazenave and Lions \cite{Cazenave1982}, which was refined into a profile decomposition theorem as above by Gérard \cite{Gerard1998} for Sobolev spaces on $\R^n$. One can find a proof of this profile decomposition theorem for Sobolev spaces on the Heisenberg group in Benameur's work \cite{Benameur2008}, which is here restricted to the subspace of radial functions.

\subsection{The limit \texorpdfstring{$\beta\to 1^-$}{beta -> 1}}\label{subsection:limit}

In this part, we study the behavior of the traveling waves $Q_\beta$ as $\beta$ tends to the limit $1^-$. We show that these traveling waves converge up to symmetries to a limiting profile. The strategy is similar to \cite{GerardLenzmannPocovnicuRaphael2018} for the half-wave equation.

For $\beta\in(-1,1)$, let $Q_\beta$ be a minimizer of $J_\beta$ : $I_\beta=J_\beta(Q_\beta)$. Up to a change of functions $Q_\beta\rightsquigarrow \alpha Q_\beta$, one can choose $Q_\beta$ such that
$$
\frac{(-(\hlapl +\beta D_s)Q_\beta,Q_\beta)_{\dot{H}^{-1}(\heis)\times\dot{H}^1(\heis)}}{1-\beta}=\|Q_\beta\|_{L^4(\heis)}^4,
$$
so that $Q_\beta$ is a solution to equation \eqref{eq:Hbeta}.

\begin{mydef}[Minimizers in $\Qbeta$]\label{def:Qbeta}
For all $\beta \in (-1,1)$, denote by $\Qbeta$ the set of minimizers $Q_\beta$ of $J_\beta:u\mapsto \frac{(-(\hlapl +\beta D_s)u,u)_{\dot{H}^{-1}(\heis)\times\dot{H}^1(\heis)}^2}{\|u\|_{L^4}^4}$ which are satisfying
\begin{equation}\label{eq:nbeta}
\frac{(-(\hlapl +\beta D_s)Q_\beta,Q_\beta)_{\dot{H}^{-1}(\heis)\times\dot{H}^1(\heis)}}{1-\beta}  =\|Q_\beta\|_{L^4}^4=\frac{I_\beta}{(1-\beta)^2},\quad
	I_\beta=J_\beta(Q_\beta).
\end{equation}
Note that for $Q_\beta\in\Qbeta$, equation \eqref{eq:Hbeta} is verified
$$
-\frac{\hlapl +\beta D_s}{1-\beta}Q_\beta=|Q_\beta|^2Q_\beta.
$$
\end{mydef}

\begin{mydef}[Minimizers in $\Qplus$]
For all radial functions $u\in \dot{H}^1(\heis)\cap V_{0}^+\setminus \{0\}$ whose Fourier transform have a non-zero component only along the Hermite-type function $\widehat{h_0^+}$, define
$$
J_+(u):=\frac{\|u\|_{\dot{H}^1(\heis)}^4}{\|u\|_{L^4(\heis)}^4}
$$
(note that on the space $\dot{H}^1(\heis)\cap V_0^+$, $-\hlapl=D_s$). Denote by $I_+$ its infimum
$$
I_+:=\inf\left\{J_+(u);\; u\in  \dot{H}^1(\heis)\cap V_{0}^+\setminus \{0\}\right\}.
$$
Let $\Qplus$ be the set of minimizers $Q_+$ of $J_+$ such that
\begin{equation*}
\|Q_+\|_{\dot{H}^1(\heis)}^2 =  \|Q_+\|_{L^4}^4=I_+, \quad I_+=J_+(Q_+).
\end{equation*}
Then any $Q_+\in\Qplus$ is a solution to equation \eqref{eq:Hplus}
\begin{equation*}
D_sQ_+=\Pi_0^+(|Q_+|^2Q_+).
\end{equation*}
\end{mydef}

Here are some remarks about this definition.

The minimum $I_+$ is attained and positive. The proof is similar as for the minimum $I_\beta$, all there is to do is to restrict the profile decomposition theorem to the closed subspace $ \dot{H}^1(\heis)\cap V_0^+$ of $\dot{H}^1(\heis)$.

The term $\Pi_0^+(|Q_+|^2Q_+)$ may not seem suitable since $|Q_+|^2Q_+$ belongs to $L^{\frac{4}{3}(\heis)}\hookrightarrow\dot{H}^{-1}(\heis)$ whereas $\Pi_0^+$ is a projector defined on $L^2(\heis)$. Several arguments make sense to this term in later parts. On the one hand, we will see that $|Q_+|^2Q_+\in L^2(\heis)$ (cf. part \ref{subsection:ground_state_solutions_limit}). On the other hand, the projector $\Pi_0^+$ extends to $L^p(\heis)$ for all $p>1$ (see Theorem \ref{thm:P0_bounded}).

The convergence result is as follows.

\begin{thm}[Convergence]\label{thm:limit_beta1}
For all $\beta\in(-1,1)$, fix $Q_\beta\in\Qbeta$. Then, there exist a subsequence $\beta_n\to 1^-$, scalings $(\alpha_n)_{n\in\N}\in(\R_+^*)^{\N}$, cores $(s_n)_{n\in\N}\in\R^{\N}$ and a function $Q_+\in\Qplus$ such that
$$
\|\alpha_nQ_{\beta_n}(\alpha_n\cdot,\alpha_n\cdot,\alpha_n^2(\cdot+s_n))-Q_+\|_{\dot{H}^1(\heis)}
	\longrightarroww{n\to+\infty}{} 0.
$$
\end{thm}

We introduce the quantity $\delta(u)$, which quantifies the gap between the norms of a function $u$ in $\dot{H}^1(\heis)$ and those of the profiles $Q_+\in\Qplus$. We prove that $\delta(Q_\beta)$ is small, and then show that $\delta(u)$ controls the distance up to symmetries from $u$ to the profiles $Q_+$ in $ \Qplus$.

\begin{mydef}\label{def:delta_u}
For  $u\in \dot{H}^{1}(\heis)$, define
$$
\delta(u)=\big|\|u\|_{\dot{H}^1(\heis)}^2-I_+\big|+\big|\|u\|_{L^4(\heis)}^4-I_+\big|.
$$
\end{mydef}

We first show a lemma about $\delta(Q_\beta)$, $Q_\beta\in\Qbeta$.
\begin{lem}\label{lem:delta_Qbeta}
There exist $C>0$ and $\beta_*\in(0,1)$ such that the following holds. For all $\beta\in(\beta_*,1)$ fix $Q_\beta\in\Qbeta$, and decompose $Q_\beta$ along the Hermite-type functions from part \ref{subsection:hermite}
$$
Q_\beta=Q_\beta^++R_\beta,
$$
where $Q_\beta^+\in \dot{H}^1(\heis)\cap V_0^+$ and $R_\beta\in \dot{H}^1(\heis)\cap \bigoplus_{(n,\pm)\neq (0,+)}V_n^\pm$. Then $\|R_\beta\|_{\dot{H}^1(\heis)}\leq C(1-\beta)^\half$, $\delta(Q_\beta^+)\leq C(1-\beta)^\half$ and $\delta(Q_\beta)\leq C(1-\beta)^\half$.
\end{lem}

\begin{proof}
Fix $u\in \dot{H}^1(\heis)$. Thanks to inequality \eqref{eq:equivalence_norms},
$$
(1-|\beta|)\|u\|_{\dot{H}^1}^2\leq (-(\hlapl +\beta D_s)u,u)_{\dot{H}^{-1}(\heis)\times\dot{H}^1(\heis)} \leq (1+|\beta|)\|u\|_{\dot{H}^1}^2,
$$
one knows that
$I_\beta\geq (1-\beta)^{2}I_0$ when $\beta\in(0,1)$.

Furthermore, let $Q_+\in \Qplus$. Then, using the fact that $-\hlapl Q_+=D_sQ_+$,
\begin{align*}
I_\beta
	&\leq J_\beta(Q_+)\\
   &=\frac{(1-\beta)^2(D_sQ_+,Q_+)_{\dot{H}^{-1}(\heis)\times\dot{H}^1(\heis)}^2}{\|Q_+\|_{L^4}^4}\\
   &=(1-\beta)^{2}I_+.
\end{align*}
Consequently,  $(\frac{I_\beta}{(1-\beta)^{2}})_\beta$ is bounded above and below~:
$$
I_0\leq\frac{I_\beta}{(1-\beta)^{2}}\leq I_+.
$$
We will show that actually $\frac{I_\beta}{(1-\beta)^{2}}\to I_+$ as $\beta$ tends to $1$.

Let us decompose a minimizer $Q_\beta\in\Qbeta$ along the Hermite-type functions from part \ref{subsection:hermite}
$$
Q_\beta=Q_\beta^++R_\beta,
$$
where $Q_\beta^+\in \dot{H}^1(\heis)\cap V_0^+$, and $R_\beta\in \dot{H}^1(\heis)\cap \bigoplus_{(n,\pm)\neq (0,+)}V_n^\pm$ is a remainder term which will go to zero.

Multiplying equation \eqref{eq:Hbeta} by $\overline{R_\beta}$, we get that for all $n$,
$$
\big(-\frac{\hlapl +\beta D_s}{1-\beta}Q_\beta,R_\beta\big)_{\dot{H}^{-1}(\heis)\times\dot{H}^1(\heis)}
	=\big(|Q_\beta|^2Q_\beta,R_\beta\big)_{L^{\frac{4}{3}}(\heis)\times L^4(\heis)}.
$$
Since the operators $\hlapl $ and $D_s$ let invariant the spaces $V_n^{\pm}$, we can replace $Q_\beta$ by $R_\beta$ in the left term of the equality
$$
\big(-\frac{\hlapl +\beta D_s}{1-\beta}R_\beta,R_\beta\big)_{\dot{H}^{-1}(\heis)\times\dot{H}^1(\heis)}
	=\big(|Q_\beta|^2Q_\beta,R_\beta\big)_{L^{\frac{4}{3}}(\heis)\times L^4(\heis)}.
$$
Applying Hölder's inequality, we deduce that
\begin{equation}\label{eq:Rbeta}
\big(-\frac{\hlapl +\beta D_s}{1-\beta}R_\beta,R_\beta\big)_{\dot{H}^{-1}(\heis)\times\dot{H}^1(\heis)}
	\leq \|Q_\beta\|_{L^4(\heis)}^3\|R_\beta\|_{L^4(\heis)}.
\end{equation}

Now, let us write more precisely the equivalence \eqref{eq:equivalence_norms} between the norms $\|u\|_{\dot{H}^1(\heis)}$ and $\big(-(\hlapl +\beta D_s)u,u\big)_{\dot{H}^{-1}(\heis)\times\dot{H}^1(\heis)}^{\half}$. The left inequality can be controlled with sharper constants which do not depend on $\beta$ when we impose the function $u\in\dot{H}^1(\heis)$ to have a zero component $u_0^+$. Indeed, remark that when $n\geq 1$,
$$
n+1-\beta\geq n\geq (n+1)/2,
$$
and when $n\geq 0$,
$$
n+1+\beta\geq n+1\geq (n+1)/2.
$$
We deduce that for all $u\in\dot{H}^1(\heis)\cap \bigoplus_{(n,\pm)\neq(0,+)}V_n^\pm$, decomposing $u$ as $u=\sum_{(n,\pm)\neq (0,+)}u_n^\pm$, $u_n^\pm\in \dot{H}^1(\heis)\cap V_n^\pm$,
\begin{align*}
(-(\hlapl +\beta D_s)u,u)_{\dot{H}^{-1}(\heis)\times\dot{H}^1(\heis)}
	&=\sum_{(n,\pm)\neq(0,0)}\int_{\R^3}((n+1)|\sigma|-\beta\sigma) |\widehat{u_n^\pm}(x,y,\sigma)|^2\d x\d y\d \sigma\\
	&\geq \half \sum_{(n,\pm)\neq(0,0)}\int_{\R^3}(n+1)|\sigma|  |\widehat{u_n^\pm}(x,y,\sigma)|^2\d x\d y\d \sigma.
\end{align*}
This implies the inequality
\begin{equation}\label{eq:normesV0perp}
\|u\|_{\dot{H}^1(\heis)}^2
	\leq 2\big(-(\hlapl +\beta D_s)u,u\big)_{\dot{H}^{-1}(\heis)\times\dot{H}^1(\heis)},\quad
	u\in \dot{H}^1(\heis)\cap \bigoplus_{(n,\pm)\neq(0,+)}V_n^\pm,
\end{equation}
which we can use for $u=R_\beta$. Combining this inequality and the Folland-Stein inequality $\|u\|_{L^4(\heis)}\leq C\|u\|_{\dot{H}^1(\heis)}$ in \eqref{eq:Rbeta} , we get 
\begin{multline*}
\big(-\frac{\hlapl +\beta D_s}{1-\beta}R_\beta,R_\beta\big)
	_{\dot{H}^{-1}(\heis)\times\dot{H}^1(\heis)}\\
	\leq C \|Q_\beta\|_{L^4(\heis)}^3 \left(2(1-\beta)\big(-\frac{\hlapl +\beta D_s}{1-\beta}R_\beta,R_\beta\big)_{\dot{H}^{-1}(\heis)\times\dot{H}^1(\heis)}\right)^\half,
\end{multline*}
so
$$
\big(-\frac{\hlapl +\beta D_s}{1-\beta}R_\beta,R_\beta\big)_{\dot{H}^{-1}(\heis)\times\dot{H}^1(\heis)}
	\leq 2C^2(1-\beta) \|Q_\beta\|_{L^4(\heis)}^6.
$$
Since $(\|Q_\beta\|_{L^4(\heis)})_\beta$ is bounded independently of $\beta$ thanks to the norm conditions \eqref{eq:nbeta} and the boundedness of $(\frac{I_\beta}{(1-\beta)^2})_\beta$, we deduce that as $\beta$ goes to $1$,
$$\
\big(-\frac{\hlapl +\beta D_s}{1-\beta}R_\beta,R_\beta\big)_{\dot{H}^{-1}(\heis)\times\dot{H}^1(\heis)}
	=\gdO(1-\beta).
$$
This implies immediately that
$
\|R_\beta\|_{\dot{H}^1(\heis)}^2=\gdO(1-\beta)
$
and
$
\|R_\beta\|_{L^4(\heis)}^2=\gdO(1-\beta).
$
Using the orthogonal decomposition $Q_\beta=Q_\beta^++R_\beta$ in $\dot{H}^1(\heis)$ and the fact that $-\hlapl=D_s$ on $\dot{H}^1(\heis)\cap V_0^+$, we get
\begin{align*}
\|Q_\beta^+\|_{\dot{H}^1(\heis)}^2
	&=\big(-\frac{\hlapl +\beta D_s}{1-\beta}Q_\beta^+,Q_\beta^+\big)_{\dot{H}^{-1}(\heis)\times\dot{H}^1(\heis)}\\
	&=\big(-\frac{\hlapl +\beta D_s}{1-\beta}Q_\beta,Q_\beta\big)_{\dot{H}^{-1}(\heis)\times\dot{H}^1(\heis)}+\gdO(1-\beta)\\
	&=\frac{I_\beta}{(1-\beta)^2}+\gdO(1-\beta),
\end{align*}
and
\begin{align*}
\|Q_\beta^+\|_{L^4(\heis)}^4
	&=\|Q_\beta\|_{L^4(\heis)}^4+\gdO((1-\beta)^\half)\\
	&=\frac{I_\beta}{(1-\beta)^2}+\gdO((1-\beta)^\half).
\end{align*}

We are now in position to prove that 
$
\frac{I_\beta}{(1-\beta)^{2}}\longrightarroww{\beta\to 1^-}{} I_+.
$
From the definition of $I_+$ as a minimum on $\dot{H}^1(\heis)\cap V_0^+$,
\begin{align*}
I_+
	&\leq \frac{\|Q_\beta^+\|_{\dot{H}^1(\heis)}^4}{\|Q_\beta^+\|_{L^4(\heis)}^4}\\
	&=\frac{\left(\frac{I_\beta}{(1-\beta)^2}+\gdO(1-\beta)\right)^2}{\frac{I_\beta}{(1-\beta)^2}+\gdO((1-\beta)^\half)}\\
	&=\frac{I_\beta}{(1-\beta)^2}(1+\gdO((1-\beta)^\half)).
\end{align*}
We already know that $\frac{I_\beta}{(1-\beta)^2}\leq I_+$ for all $\beta$, so we conclude that
$$
\frac{I_\beta}{(1-\beta)^{2}}\longrightarroww{\beta\to 1^-}{} I_+.
$$
Therefore, the norms of $Q_\beta^+$ rewrite
$
\|Q_\beta^+\|_{\dot{H}^1(\heis)}^2= I_++\gdO((1-\beta)^\half)
$
and
$
\|Q_\beta^+\|_{L^4(\heis)}^4= I_++\gdO((1-\beta)^\half).
$
We conclude that
$$
\delta(Q_\beta^+)
	=\gdO((1-\beta)^\half)
$$
and
$$
\delta(Q_\beta)
	=\delta(Q_\beta^++R_{\beta})
	=\gdO((1-\beta)^\half).
$$
\end{proof}

The following stability result allows us to complete the proof of Theorem \ref{thm:limit_beta1}.

\begin{prop}\label{prop:stability1}
Fix a sequence $(u_n)_{n\in\N}$ of radial functions in $\dot{H}^{1}(\heis)\cap V_0^+$. Suppose that $\delta(u_n)\longrightarroww{n\to+\infty}{} 0$. Then, up to a subsequence, there exist scalings $(\alpha_n)_{n\in\N}\in (\R_+^*)^{\N}$, cores $(s_n)_{n\in\N}\in\R^{\N}$ and a ground state $Q_+\in\Qplus$ optimizing
$$
I_+=\inf\left\{J_+(u)=\frac{\|u\|_{\dot{H}^1(\heis)}^4}{\|u\|_{L^4(\heis)}^4};\quad u\in  \dot{H}^1(\heis)\cap V_{0}^+\setminus \{0\}\right\},
$$
such that
$$
\Big\|\alpha_nu_n(\alpha_n\cdot,\alpha_n\cdot,\alpha_n^2(\cdot+s_n))-Q_+\Big\|_{\dot{H}^{1}(\heis)}\longrightarroww{n\to+\infty}{} 0.
$$
\end{prop}

\begin{proof}
Let $(u_n)_{n\in\N}\in(\dot{H}^{1}(\heis)\cap V_0^+)^{\N}$ such that $\delta(u_n)\longrightarroww{n\to+\infty}{} 0$. Since $\dot{H}^{1}(\heis)\cap V_0^+$ is a closed subspace of $\dot{H}^{1}(\heis)$, one can restrict the concentration-compactness theorem \ref{thm:concentration_compactness} to this subspace. In consequence, one can assume that the profiles $U^{(j)}$ from the theorem lie in $\dot{H}^{1}(\heis)\cap V_0^+$. Therefore, up to a subsequence, there exist a core sequence $(s_{n}^{(j)})_{n,j\in\N}\subset\R$, a scaling sequence $(h_{n}^{(j)})_{n,j\in\N}\subset\R$, and radial functions $(U^{(j)})_{j\in\N}\subset \dot{H}^{1}(\heis)\cap V_0^+$ such that
\begin{itemize}
\item for all $j,k\in\N$, $j\neq k$, the couples $((h_{n}^{(j)})_{n},(s_{n}^{(j)})_{n})$ are pairwise strange ;
\item let
$$
r_{n}^{(l)}(x,y,s)=u_{n}(x,y,s)-\sum_{j=1}^l \frac{1}{h_{n}^{(j)}} U^{(j)}\left(\frac{x}{h_{n}^{(j)}},\frac{y}{h_{n}^{(j)}}, \frac{s-s_{n}^{(j)}}{(h_{n}^{(j)})^2}\right),
$$
then
$$
\lim_{l\to+\infty}\limsup_{n\to+\infty} \|r_{n}^{(l)}\|_{L^4(\heis)}=0.
$$
\end{itemize}
Moreover, for all $l$, as $n$ goes to $+\infty$,
\begin{equation}\label{eq:orthogonality}
\|u_n\|_{\dot{H}^1(\heis)}^2=\sum_{j=1}^l\|U^{(j)}\|_{\dot{H}^1(\heis)}^2+\|r_{n}\|_{\dot{H}^1(\heis)}^2+o(1),
\end{equation}
and
$$
\|u_{n}\|_{L^4(\heis)}^4 \longrightarroww{n\to+\infty}{} \sum_{j=1}^{+\infty}\|U^{(j)}\|_{L^4(\heis)}^4.
$$
By construction, since $\delta(u_n)$ goes to $0$, $\sum_{j=1}^{+\infty}\|U^{(j)}\|_{\dot{H}^1(\heis)}^4=I_+$, $\sum_{j=1}^{+\infty}\|U^{(j)}\|_{\dot{H}^1(\heis)}^2\leq I_+$ and $\frac{\|u_n\|_{\dot{H}^1(\heis)}^4}{\|u_n\|_{L^4(\heis)}^4}$ tends to $I_+$.
But from the definition of $I_+$ as a minimum,
\begin{align*}
I_+^2
	&\geq \left(\sum_{j=1}^{+\infty}\|U^{(j)}\|_{\dot{H}^1(\heis)}^2\right)^2\\
	&\geq \sum_{j=1}^{+\infty}\|U^{(j)}\|_{\dot{H}^1(\heis)}^4\\
	&\geq I_+\sum_{j=1}^{+\infty}\|U^{(j)}\|_{\dot{H}^1(\heis)}^4\\
	&\geq I_+\sum_{j=1}^{+\infty}\|U^{(j)}\|_{L^4(\heis)}^4\\
	&=I_+^2.
\end{align*}
All the above inequalities must then be equalities. 

In particular, only one of the profiles $U^{(j)}$ is allowed to be non-zero, we denote this profile by $Q_+$, and by $r_n, h_n$ and $s_n$ the corresponding rests, scalings and cores. Then $Q_+$ must be a ground state of the functional $J_+$, and
\begin{align*}
u_{n}(x,y,s)
	&= \frac{1}{h_{n}} Q_+\left(\frac{x}{h_{n}},\frac{y}{h_{n}}, \frac{s-s_{n}}{h_{n}^2}\right)+r_{n}(x,y,s).
\end{align*}
From relation \eqref{eq:orthogonality}, as $n$ goes to $+\infty$,
$$
\|u_n\|_{\dot{H}^1(\heis)}^2=\|Q_+\|_{\dot{H}^1(\heis)}^2+\|r_n\|_{\dot{H}^1(\heis)}^2+o(1).
$$
Since $\|u_n\|_{\dot{H}^1(\heis)}^2$ must converge to $\|Q_+\|_{\dot{H}^1(\heis)}^2$ because of the inequalities turned into equalities, we get that $\|r_n\|_{\dot{H}^1(\heis)}^2\longrightarroww{n\to+\infty}{}0$, therefore the sequence $h_nu_n(h_n\cdot,h_n\cdot,h_n^2(\cdot+s_n))$ converges to $Q_+$ in $\dot{H}^1(\heis)$.
\end{proof}

\begin{proof}[Proof of Theorem \ref{thm:limit_beta1}]
Consider the sequence $(Q_\beta^+)_{\beta\in(-1,1)}$ from Lemma \ref{lem:delta_Qbeta}. We know that $\delta(Q_\beta^+)=\gdO((1-\beta)^\half)$.

Applying Proposition \ref{prop:stability1}, there exist a subsequence $(Q_{\beta_n}^+)_{n\in\N}$ with $\beta_n\longrightarroww{n\to+\infty}{} 1^-$, a core sequence $(s_n)_{n\in\N}\in\R^{\N}$, a scaling sequence $(\alpha_n)_{n\in\N}\in(\R_+^*)^{\N}$, and a ground state $Q_+\in\Qplus$ such that
$$
\|\alpha_nQ_{\beta_n}^+(\alpha_n\cdot,\alpha_n\cdot,\alpha_n^2(\cdot+s_n))-Q_+ \|_{\dot{H}^1(\heis)}
	\longrightarroww{n\to+\infty}{}0.
$$

To conclude, since $R_{\beta_n}=Q_{\beta_n}-Q_{\beta_n}^+$ satisfies
$
\|R_{\beta_n}\|_{\dot{H}^1(\heis)}\longrightarroww{n\to+\infty}{}0,
$
and since the $\dot{H}^1$ norm is invariant by translation and scaling, we deduce that
$$
\|\alpha_nQ_{\beta_n}(\alpha_n\cdot,\alpha_n\cdot,\alpha_n^2(\cdot+s_n))-Q_+ \|_{\dot{H}^1(\heis)}
	\longrightarroww{n\to+\infty}{}0.
$$
\end{proof}

\subsection{Ground state solutions to the limiting equation}\label{subsection:ground_state_solutions_limit}

We now show that the optimizers for
$$
I_+:=\inf\left\{\frac{\|u\|_{\dot{H}^1(\heis)}^4}{\|u\|_{L^4(\heis)}^4};\quad u\in  \dot{H}^1(\heis)\cap V_{0}^+\setminus \{0\}\right\}
$$
are unique up to symmetries (translation, phase multiplication and scaling).

\begin{prop}
The minimum $I_+$ is equal to $\pi^2$. Moreover,
\begin{itemize}
\item the set composed of all minimizing functions for $I_+$ is
$$
\{(x,y,s)\in\heis\mapsto \frac{C}{s+s_0+i(x^2+y^2)+i\alpha};\quad (s_0,C,\alpha)\in\R\times\C\times\R_+^*\};
$$
\item the set $\Qplus$ composed of all minimizing functions for $I_+$ which satisfy
$$\|Q_+\|_{\dot{H}^1(\heis)}^2=\|Q_+\|_{L^4(\heis)}^4=I_+$$
(so that $Q_+$ is a solution to equation \eqref{eq:Hplus}) is
$$
\Qplus=\{(x,y,s)\in\heis\mapsto\frac{i\e^{i\theta}\sqrt{2\alpha}}{s+s_0+i(x^2+y^2)+i\alpha}; \quad (s_0,\theta,\alpha)\in\R\times\T\times\R_+^*\}.
$$
\end{itemize}
\end{prop}

\begin{proof}
Let $U\in \dot{H}^{1}(\heis)\cap V_0^+$. Let us transform the expression of the $L^4$ norm of $U$ as follows
\begin{align*}
\|U\|_{L^4(\heis)}^4
	&=\|U^2\|_{L^2(\heis)}^2\\
	&=\|\widehat{U^2}\|_{L^2(\heis)}^2\\
	&=\frac{1}{2\pi}\|\widehat{U}*\widehat{U}\|_{L^2(\heis)}^2.
\end{align*}
Let $f$ be the function associated to $U$ in the decomposition along $\widehat{h_0^+}$
$$
\widehat{U}(x,y,s)=f(\sigma)\frac{1}{\sqrt{\pi}}\e^{-(x^2+y^2)\sigma}.
$$
Then
$$
\|U\|_{\dot{H}^1(\heis)}^2=\half\int_{0}^{+\infty}|f(\sigma)|^2\d\sigma
$$
and
\begin{align*}
\|U\|_{L^4(\heis)}^4
	&=\frac{1}{2\pi}\int_{\R_+}\int_{\R}\int_{\R}\left|\int_{0}^{\sigma}f(\sigma-\sigma')\frac{1}{\sqrt{\pi}}\e^{-(x^2+y^2)(\sigma-\sigma')}f(\sigma')\frac{1}{\sqrt{\pi}}\e^{-(x^2+y^2)\sigma'}\d\sigma'\right|^2\d x\d y\d\sigma\\
	&=\frac{1}{2\pi^2}\int_{\R_+}\left|\int_0^{\sigma}f(\sigma-\sigma')f(\sigma')\d\sigma'\right|^2\frac{\d\sigma}{2\sigma}.
\end{align*}
Applying Cauchy-Schwarz's inequality,
\begin{align*}
\|U\|_{L^4(\heis)}^4
	&\leq \frac{1}{4\pi^2} \int_{0}^{+\infty}\int_0^{\sigma}|f(\sigma-\sigma')f(\sigma')|^2\d\sigma'\int_{0}^{\sigma}1\d\sigma'\frac{\d\sigma}{\sigma}\\
	&=\frac{1}{4\pi^2} \int_{0}^{+\infty}\int_0^{\sigma}|f(\sigma-\sigma')f(\sigma')|^2\d\sigma'\d\sigma\\
	&=\frac{1}{4\pi^2} \|f\|_{L^2(\R_+)}^4\\
	&=\frac{1}{\pi^2} \|U\|_{\dot{H}^1(\heis)}^4.
\end{align*}
Consequently,
$
I_+\geq \pi^2.
$

Let us discuss the equality case. Equality holds if and only if there is equality in Cauchy-Schwarz's inequality, that is to say, for almost every $\sigma>0$, and almost every $\sigma'\in]0,\sigma[$,
$$
f(\sigma')f(\sigma-\sigma')=C(\sigma).
$$
Fix an open interval $I$ contained in $]0,\sigma[$ with positive length $|I|$. Then
$$
\int_If(\sigma')f(\sigma-\sigma')\d\sigma'=|I|C(\sigma),
$$
therefore, $C$ is continuous on $\R_+^*$ as a product of two $L^2$ functions. Since $f$ is not identically zero, one can find an interval $J\subset\R_+^*$ such that
$$
\int_Jf(\zeta)\d \zeta\neq 0.
$$
Integrating equality
\begin{equation}\label{eq:eqf}
f(\sigma)f(\zeta)=C(\sigma+\zeta),\quad (\sigma,\zeta)\in(\R_+^*)^2
\end{equation}
along the $\zeta$ variable, one gets that for all $\sigma\in\R_+^*$,
$$
f(\sigma)\int_Jf(\zeta)\d \zeta=\int_J C(\sigma+\zeta)\d\zeta=\int_{J+\sigma}C(\zeta)\d\zeta.
$$
Therefore,  $f$ has $\mathcal{C}^1$ regularity on $\R_+^*$, so $C$ also has $\mathcal{C}^1$ regularity on $\R_+^*$. Fix $\zeta>0$ such as $f(\zeta)\neq0$. Letting $\sigma\to 0^+$ in equality \eqref{eq:eqf}, one knows that $f$ admits a finite limit as  $\sigma\to 0^+$ which is equal to 
$$
f(0^+)=\frac{C(\zeta)}{f(\zeta)}.
$$
Likewise, computing the derivative along the $\sigma$ variable of equality \eqref{eq:eqf},
$$
f'(\sigma)f(\zeta)=C'(\sigma+\zeta),
$$
one gets that $f'$ admits a finite limit at $0^+$ which is equal to
$$
f'(0^+)=\frac{C'(\zeta)}{f(\zeta)}.
$$
We deduce that $f$ satisfies the differential equation
$$
f'(\sigma)f(0^+)=f(\sigma)f'(0^+)=C'(\sigma),\quad \sigma\in\R_+^*.
$$
Let us show that $f(0^+)\neq 0$. Supposing $f(0^+)=0$, we would get that for all $\sigma>0$, $C'(\sigma)=0$. Then $C$ would be a constant function, so $f$ would be constant too since
$$
f(\sigma)=\frac{C(\sigma+\zeta)}{f(\zeta)}.
$$
As $f$ is in $L^2(\R_+)$, this would imply that $f$ is identically zero, which is a contradiction.

Therefore, solving the differential equation, there exist some constants $K$ and $\alpha$ such that, for all $\sigma\geq0$,
$$
f(\sigma)=K\e^{-\alpha\sigma}.
$$
The assumption $f\in L^2(\R_+)$ implies that $\textnormal{Re}(\alpha)>0$.

Computing the inverse Fourier transform leads to 
\begin{align*}
U(x,y,s)
	&=\frac{1}{\sqrt{2\pi}}\int_0^{+\infty}\e^{is\sigma}f(\sigma)\frac{1}{\sqrt{\pi}}\e^{-(x^2+y^2)\sigma}\d \sigma\\
	&=\frac{K}{\pi\sqrt{2}}\int_0^{+\infty}\e^{is\sigma-\alpha\sigma-(x^2+y^2)\sigma}\d \sigma
\end{align*}
so
\begin{equation*}
U(x,y,s)
	=\frac{K}{\pi\sqrt{2}}\frac{1}{x^2+y^2+\alpha-is}.
\end{equation*}
This is the first point of the proposition. Let us now prove the second point.

Since the equation and the result we want to show are both invariant under translation of the $s$ variable, up to translating of a factor $s_0$, we will assume from now on that $\alpha$ is a (positive) real number.

Now,
\begin{align*}
\|U\|_{\dot{H}^1(\heis)}^2
	=\half|K|^2\int_{0}^{+\infty}\e^{-2\alpha\sigma}\d\sigma
	=\half\frac{|K|^2}{2\alpha}
\end{align*}
and
\begin{align*}
\|U\|_{L^4(\heis)}^4
	&=\frac{1}{4\pi^2}|K|^4\int_{0}^{+\infty}\left|\int_0^{\sigma}\e^{-\alpha(\sigma-\sigma')}\e^{-\alpha\sigma'}\d\sigma'\right|^2\frac{\d\sigma}{\sigma}\\
	&=\frac{1}{4\pi^2}|K|^4\int_{0}^{+\infty}\sigma \e^{-2\alpha\sigma}\d\sigma\\
	&=\frac{1}{4\pi^2}\frac{|K|^4}{(2\alpha)^2},
\end{align*}
so $U$ satisfies $\|U\|_{\dot{H}^1(\heis)}^2=\|U\|_{L^4(\heis)}^4=I_+$ if and only if
$|K|^2=4\pi^2\alpha$. In this case, write $K=2\pi\sqrt{\alpha}\e^{i\theta}$ for some $\theta\in \T$, then,
\begin{align*}
U(s,x,y)
	&=\frac{K}{\pi\sqrt{2}}\frac{1}{x^2+y^2+\alpha-is}\\
	&=\frac{\e^{i\theta}\sqrt{2\alpha}}{x^2+y^2+\alpha-is}.
\end{align*}
\end{proof}

We proved that up to the symmetries of the equation, there is a unique minimizer $Q_+$ in $\Qplus$, which is equal with the choice of parameters $(s_0,\theta,\alpha)=(0,0,1)$ to
$$
Q_+(s,x,y)
	=\frac{i\sqrt{2}}{s+i(x^2+y^2)+i}
,
$$
with Fourier transform
$$
\widehat{Q_+}(x,y,\sigma)=2\pi \e^{-\sigma}\widehat{h_0^+}(x,y,\sigma).
$$
Note that the profile $Q_+$ has infinite mass.

\section{The limiting problem}\label{section:limiting_problem}

We now focus on the stability of $Q_+$, which is the unique ground state solution up to symmetry to \eqref{eq:Hplus}
$$
D_sQ_+=\Pi_0^+(|Q_+|^2Q_+).
$$
Let us study the linearized operator $\L$ close to $Q_+$
$$
\L h=- \hlapl h-2\Pi_0^+(|Q_+|^2h)-\Pi_0^+(Q_+^2\overline{h}), \quad h\in \dot{H}^1(\heis)\cap V_0^+.
$$
We first study the linearized operator on the real subspace spanned by $(Q_+,iQ_+,\partial_sQ_+,i\partial_sQ_+)$ with the help of the correspondence with Bergman spaces (parts \ref{subsection:bergman_spaces} and \ref{subsection:othogonality}). Then, on the orthogonal of this subspace in $\dot{H}^1(\heis)\cap V_0^+$, we prove the coercivity of $\L$ by using the spectral properties of the sub-Laplacian on the CR sphere via the Cayley transform (parts \ref{subsection:cayley_transform} and \ref{subsection:coercivity}). We conclude this section with some estimates about the invertibility of the linearized operator $\L$ (part \ref{subsection:L_invertible}).

\subsection{Bergman spaces on the upper half plane}\label{subsection:bergman_spaces}

In order to better understand the spaces $\dot{H}^k(\heis)\cap V_0^+$, $k\in\{-1,0,1\}$, we need to introduce their link with Bergman spaces on the upper half-plane $\C_+$. The space $\dot{H}^k(\heis)\cap V_0^+$ is the subspace of $\dot{H}^k(\heis)$ spanned (after a Fourier transform under the $s$ variable) by $\widehat{h_0^+}(x,y,\sigma)=\frac{1}{\sqrt{\pi}}\exp(-(x^2+y^2)\sigma) \un_{\sigma\geq 0}$~: $u\in \dot{H}^k(\heis)\cap V_0^+$ if $u\in\dot{H}^k(\heis)$ and
$$
\widehat{u}(x,y,s)=f(\sigma)\widehat{h_0^+}(x,y,\sigma),
$$
where
$$
\|u\|_{\dot{H}^k(\heis)}^2=\|(-\hlapl)^{\frac{k}{2}}u\|_{L^2(\heis)}^2=\int_{\R_+}|f(\sigma)|^2\frac{\d\sigma}{2\sigma^{1-k}}.
$$

Let us define the weighted Bergman spaces as follows.
\begin{mydef}[Weighted Bergman spaces] Given $k<1$ and $p\in[1,+\infty)$, the weighted Bergman space $A_{1-k}^p$ is the subspace of
$
L^p_{1-k}:=L^p(\C_+,\Im(z)^{-k}\d\lambda(z))
$
composed of holomorphic functions of the complex upper half-plane $\C_+$~:
$$
A_{1-k}^p:=\left\{F\in\holo(\C_+);\quad \|F\|_{L_{1-k}^p}^p:=\int_0^{+\infty}\int_{\R}|F(s+it)|^p\d s\frac{\d t}{t^{k}}<+\infty\right\}.
$$
\end{mydef}

Thanks to the following Paley-Wiener theorem on weighted Bergman spaces \cite{bergman_projectors}, one can associate to each element of $\dot{H}^k(\heis)\cap V_0^+$ a function of the weighted Bergman space $A_{1-k}^2$.
\begin{thm}[Paley-Wiener] Let $k<1$. Then for every $f\in L^2(\R_+,\sigma^{k-1}\d\sigma)$, the following integral is absolutely convergent on $\C_+$
\begin{equation}\label{eq:paley1}
F(z)=\frac{1}{\sqrt{2\pi}}\int_0^{+\infty}\e^{iz\sigma}f(\sigma)\d\sigma,
\end{equation}
and defines a function $F\in A^2_{1-k}$ which satisfies
\begin{equation}\label{eq:paley2}
\|F\|^2_{L^2_{1-k}}=\frac{\Gamma(1-k)}{2^{1-k}}\int_{0}^{+\infty}|f(\sigma)|^2\frac{\d\sigma}{\sigma^{1-k}}.
\end{equation}
Conversely, for every $F\in A^2_{1-k}$, there exists $f\in L^2(\R_+,\sigma^{k-1}\d\sigma)$ such that \eqref{eq:paley1} and \eqref{eq:paley2} hold.
\end{thm}

When dealing with functions from the space $\dot{H}^1(\heis)$, we use the usual Paley-Wiener theorem \cite{Rudin1987}.

\begin{mydef}
The Hardy space $\hardy^2(\C_+)$ space of holomorphic functions of the upper half-plane $\C_+$ such that the following norm is finite~:
$$
\|F\|^2_{\hardy^2(\C_+)}:=\sup_{t>0}\int_{\R}|F(s+it)|^2\d s<+\infty.
$$
\end{mydef}
\begin{thm}[Paley-Wiener] For every $f\in L^2(\R_+)$, the following integral is absolutely convergent on $\C_+$
\begin{equation}\label{eq:paley3}
F(z)=\frac{1}{\sqrt{2\pi}}\int_0^{+\infty}\e^{iz\sigma}f(\sigma)\d\sigma,
\end{equation}
and defines a function $F$ in the Hardy space $\hardy^2(\C_+)$ which satisfies
\begin{equation}\label{eq:paley4}
\|F\|^2_{\hardy^2(\C_+)}=\int_{0}^{+\infty}|f(\sigma)|^2\d\sigma.
\end{equation}
Conversely, for every $F\in \hardy^2(\C_+)$, there exists $f\in L^2(\R_+)$ such that \eqref{eq:paley3} and \eqref{eq:paley4} hold.
\end{thm}

Given any $h\in\dot{H}^k(\heis)$ radial, one can define
\begin{equation*}
F_h(s+i(x^2+y^2)):=h(x,y,s).
\end{equation*}

If $h\in \dot{H}^k(\heis)\cap V_0^+$, $k\in\{-1,0,1\}$, then $F_h$ is holomorphic, since the holomorphic representation given by the suitable Paley-Wiener theorem is  given by $\sqrt{\pi}F_h$.
Note that
$$
F_{-\hlapl h}=-iF_{\partial_s h}=-iF_h',\quad h\in\dot{H}^k(\heis)\cap V_0^+
$$
and
$$
F_{gh}=F_gF_h,\quad g,h\in\dot{H}^k(\heis)\cap V_0^+.
$$
Moreover, if $h\in L^2(\heis)$,
\begin{equation}\label{eq:paley_cste}
\|h\|_{L^2(\heis)}^2=\pi\|F_h\|_{L^2(\C_+)}^2.
\end{equation}

For example, the holomorphic representation in the Hardy space $\hardy^2(\C_+)$ of
\begin{align*}
Q_+(x,y,s)
	&=\frac{i\sqrt{2}}{i(x^2+y^2)+i+s}
\end{align*}
is
$$
F_{Q_+}(z)=\frac{i\sqrt{2}}{z+i}.
$$

One can now identify the orthogonal projector $\Pi_0^+$ from the Hilbert space $L^2(\heis)$ onto its closed subspace $L^2(\heis)\cap V_0^+$ as a projector $P_0$ from $L^2(\C_+)$ to $A^2_1=L^2(\C_+)\cap\holo(\C_+)$. More generally, for $k<1$, the orthogonal projector from the Hilbert space $\dot{H}^k(\heis)$ onto its closed subspace $\dot{H}^k(\heis)\cap V_0^+$ corresponds to the Bergman projector $P_k$ from $L^2_{1-k}$ onto $A^2_{1-k}$. For general $k<1$, the Bergman projector $P_k$ can be expressed as a convolution through a reproducing kernel called Bergman kernel \cite{bergman_projectors}. We are here interested in the case $k=0$.


\begin{prop}\label{prop:formula_P0}
For all $F\in L^2(\C_+)$,
$$
P_0(F)(z)
	=-\frac{1}{\pi}\int_{\C_+}\frac{1}{(z-s+it)^2}F(s+it)\d s\d t.
$$
\end{prop}


For $h\in L^2(\heis)$, the holomorphic function $F_{\Pi_0^+(h)}$ is the projection of $F_h$ on the subspace $A^2_1$ of $L^2(\C_+)$~:
$$
F_{\Pi_0^+(h)}(z)
	=P_0(F_h)(z),
$$
so
\begin{equation*}
F_{\Pi_0^+(h)}(z)
	=-\frac{1}{\pi}\int_{\C_+}\frac{1}{(z-s+it)^2}F_h(s+it)\d s\d t.
\end{equation*}

For $p\in(1,+\infty)$, the orthogonal projector $P_0$ can be extended as a bounded operator from the space $L^p(\C_+,\d\lambda(z))$ onto the Bergman space $A_{1}^p$ \cite{bergman_projectors}.
\begin{thm}\label{thm:P0_bounded}
Let $p\in[1,+\infty)$. Then the Bergman projector $P_0$ is a bounded operator in $L^p(\C_+)$ if and only if $p>1$.
\end{thm}

One has
$
\|h\|_{L^p(\heis)}^p=\pi\|F_h\|_{L^p(\C_+)}^p
$
when this quantity is finite. Therefore, if $h_1,h_2,h_3\in\dot{H}^1(\heis)$ (which embeds in $L^4(\heis)$), it makes sense to consider $\Pi_0^+(h_1h_2h_3)$.

\subsection{Symmetries of the equation and orthogonality conditions}\label{subsection:othogonality}

In this part, we focus on the linearized operator $\L$ around $Q_+$
$$
\L h=- \hlapl h-2\Pi_0^+(|Q_+|^2h)-\Pi_0^+(Q_+^2\overline{h}), \quad h\in \dot{H}^1(\heis)\cap V_0^+.
$$
This operator is self-adjoint acting on $L^2(\heis)$, but we are interested in elements of $\dot{H}^1(\heis)$ endowed with its own scalar product. After studying the action of $\L$ on the real subspace $V$ spanned by $(Q_+,iQ_+,\partial_s Q_+,i\partial_s Q_+)$, we will try to find a new form for $(\L h,h)_{\dot{H}^{-1}(\heis)\times\dot{H}^1(\heis)}$ on the orthogonal of $V$ in $\dot{H}^1(\heis)$ which is more suitable for a spectral study.

\begin{prop}
In the real subspace $V$ of $\dot{H}^1(\heis)\cap V_0^+$ spanned by the orthogonal basis of vectors $(\partial_s Q_+,iQ_+-\partial_sQ_+,Q_++2i\partial_s Q_+,Q_+)$, the linearized operator $\L$ has the form
$$
\L_{|V}=
\left(
\begin{matrix}
0 & 0 & 0 & 0\\
0 & 0 & 0 & 0\\
0 & 0 & 0 & 1\\
0 & 0 & 0 & -1\\
\end{matrix}
\right).
$$
\end{prop}

\begin{proof}
We define
$$
\tilde{\L}(F):=-iF'-2P_0(|F_{Q_+}|^2F_h)-P_0(F_{Q_+}^2\overline{F_h}),\quad F\in\hardy^2(\C_+).
$$
For $h\in\dot{H}^1(\heis)\cap V_0^+$, the holomorphic function $F_h\in\hardy^2(\C_+)$ satisfies
$$
\tilde{\L}(F_h)=F_{\L h}.
$$
We study $\tilde{L}$ on $\hardy^2(\C_+)$. For $F\in\hardy^2(\C_+)$, define
$$
\F(F):=-iF'-P_0(|F|^2F).
$$
Let $U$ be a $\mathcal{C}^1$ function defined on a neighbourhood of $t=0$, valued in $\hardy^2(\C_+)$, and satisfying $U(0)=F_{Q_+}$ and $U'(0)=F$. Then
$$
\tilde{\L} F=\frac{\d}{\d t}\Big|_{t=0}\F(U(t)).
$$

Thanks to the invariance under translation in the $s$ variable, we consider $U:s_0\in\R\mapsto F_{Q_+}(\cdot+s_0)$. For all $s_0\in\R$, $\F(U(s_0))=0$, so
$$
\tilde{\L} (F_{Q_+}')=0=\L(\partial_sQ_+).
$$

Following the same pattern, the invariance under phase multiplication gives, with  $U: \theta\in\R\mapsto \e^{i\theta}F_{Q_+}$, that  $\F(U(\theta))=0$ for all $\theta$, so
$$
\tilde{\L}(iF_{Q_+})=0=\L (i Q_+).
$$

Finally, let $U:\lambda\in]-1,1[\mapsto(1+\lambda)F_{Q_+}((1+\lambda)^2\cdot)$, then $\F(U(\lambda))=0$ for all $\lambda$ thanks to the scaling invariance, so
$$
\tilde{\L} (F_{Q_+}+2zF_{Q_+}')=0.
$$
Remark that
$$
zF_{Q_+}'=-\frac{i\sqrt{2}z}{(z+i)^2}=-F_{Q_+}-iF_{Q_+}'.
$$
Consequently,
$$
\L (Q_++2i\partial_sQ_+)=0.
$$

In order to determine $\L $ entirely on the subspace $V$, it is sufficient to calculate $\L (Q_+)$. Yet
\begin{align*}
\L (Q_+)
	&=-i\partial_sQ_+-3\Pi_0^+(|Q_+|^2Q_+)
   =2i\partial_sQ_+.
\end{align*}

We have proved that in the orthogonal basis $(\partial_sQ_+,iQ_+-\partial_sQ_+,Q_++2i\partial_sQ_+,Q_+)$ of $V$, $\L $ admits the matrix representation
$$
\left(
\begin{matrix}
0 & 0 & 0 & 0\\
0 & 0 & 0 & 0\\
0 & 0 & 0 & 1\\
0 & 0 & 0 & -1\\
\end{matrix}
\right).
$$
\end{proof}

We want now to work on the orthogonal of $V$, so we will study the orthogonality conditions. For this part, it is more natural to work with the complex scalar product in $\dot{H}^1(\heis)$
\begin{align*}
\langle h_1,h_2\rangle_{\dot{H}^{1}(\heis)}
	&=\int_{\heis}(-\hlapl h_1) \overline{h_2}\d x\d y\d s\\
	&=\langle -\hlapl h_1, h_2\rangle_{\dot{H}^{-1}(\heis)\times\dot{H}^{1}(\heis)}.
\end{align*}
We have
$$
\langle h,Q_+\rangle_{\dot{H}^1(\heis)}=(h,Q_+)_{\dot{H}^1(\heis)}+i(h,iQ_+)_{\dot{H}^1(\heis)}.
$$

\begin{prop}\label{prop:formula_orthogonal}
Let $h\in \dot{H}^1(\heis)\cap V_0^+$, $F_h(s+i(x^2+y^2)) = h(x,y,s)$ its holomorphic counterpart. Then
\begin{equation*}
\langle h,Q_+\rangle_{\dot{H}^1(\heis)}=\sqrt{2}\pi^2 F_h(i).
\end{equation*}
Consequently,
\begin{itemize}
\item $h$ is orthogonal to $Q_+$ and $iQ_+$ in $\dot{H}^1(\heis)$ if and only if $F_h(i)=0$;
\item  $h$ is orthogonal to $\partial_sQ_+$ and $i\partial_sQ_+$ if and only if $F_h'(i)=0$.
\end{itemize}
\end{prop}

Note that this proposition enables us to check easily that the basis $(\partial_sQ_+,iQ_+-\partial_sQ_+,Q_++2i\partial_sQ_+,Q_+)$ of $V$ is orthogonal in $\dot{H}^1(\heis)$.

\begin{proof}
We study of the duality bracket in $\dot{H}^{-1}(\heis)\times \dot{H}^1(\heis)$ between $-\hlapl Q_+=D_sQ_+$ and $h$, for which we use the holomorphic function $F_h$. Knowing that
$$
F_{\partial_sQ_+}(z)=F_{Q_+}'(z)=-\frac{i\sqrt{2}}{(z+i)^2},
$$
equality \eqref{eq:paley_cste} $\|u\|_{L^2(\heis)}^2=\pi\|F_u\|_{L^2(\C_+)}^2$ for $u\in L^2(\heis)$ leads to
\begin{align*}
\langle h, \partial_sQ_+\rangle_{\dot{H}^{1}(\heis)\times\dot{H}^{-1}(\heis)}
	=\pi\int_{\C_+}\frac{i\sqrt{2}}{(\overline{z+i })^2}F_h(z)\d \lambda(z).
\end{align*}

Let $t>0$, and define $f_t:z\mapsto F_h(z+it)$ on $\{z\in\C;\;  \textnormal{Im}(z)>-t\}$. Applying the residue formula to $z\mapsto \frac{1}{(z-it-i )^2}f_t(z)$, which is holomorphic on $\{z\in\C;\;  \textnormal{Im}(z)>-t\}\setminus \{it+i \}$ with a simple pole at $it+i$, we get that on every rectangle $\mathscr{R}:=[-a,a]+i[0,b]$ containing $it+i $,
\begin{equation}\label{eq:rectangles}
\int_{\partial\mathscr{R}}\frac{1}{(z-it-i )^2}f_t(z)\d z
	=2i\pi f_t'(it+i )
	=2i\pi F_h'(2it+i ).
\end{equation}

Since the integral of $z\mapsto \frac{1}{(z-it-i )^2}f_t(z)$ is absolutely convergent on $\{z\in\C;\; \textnormal{Im}(z)>-t\}$, there are some sequences $(a_j)_{j\in\N}$ and $(b_j)_{j\in\N}$ of real numbers converging to $+\infty$ and satisfying
$$
\int_{\R_+}\frac{1}{(-a_j+it'-it-i )^2}f_t(-a_j+it')\d t'\to 0,
$$
$$
\int_{\R_+}\frac{1}{(a_j+it'-it-i )^2}f_t(a_j+it')\d t'\to 0,
$$
and
$$
\int_{\R}\frac{1}{(s+ib_j-it-i )^2}f_t(s+ib_j)\d s\to 0.
$$

Applying formula \eqref{eq:rectangles} to the rectangles $[-a_j,a_j]\times[0,b_j]$ and passing to the limit $j\to+\infty$, one gets
\begin{align*}
\int_{\R}\frac{1}{(s-it-i )^2}f_t(s)\d s
	&=2i\pi F_h'(2it+i ).
\end{align*}
Consequently
\begin{align*}
\langle h, \partial_sQ_+\rangle_{\dot{H}^{1}(\heis)\times\dot{H}^{-1}(\heis)}
	&=i\pi\sqrt{2}2i\pi\int_{\R_+}F_h'(2it+i )\d t\\
	&=-i\pi^2\sqrt{2}F_h(i ),
\end{align*}
since $F_h(it)$ goes to $0$ as $t$ goes to $+\infty$. This latter fact can be established by using the function $f\in L^2(\R_+)$ associated to $F_h$, which satisfies for all $t\in\R_+^*$
\begin{equation*}
F_h(it)
	=\frac{1}{\sqrt{2\pi}}\int_0^{+\infty}\e^{-t\sigma}f(\sigma)\d\sigma,
\end{equation*}
indeed,
\begin{align*}
|F_h(it)|
	&\leq\frac{1}{\sqrt{2\pi}}\left(\int_0^{+\infty}\e^{-2t\sigma}\d\sigma\right)^{\half}\left(\int_0^{+\infty}|f(\sigma)|^2\d\sigma\right)^{\half}
	=\frac{1}{2\sqrt{\pi t}}\|f\|_{L^2},
\end{align*}
which goes to $0$ as $t$ goes to $+\infty$. 

We have shown as wanted that
\begin{equation*}
\langle h, Q_+\rangle_{\dot{H}^{1}(\heis)}
	=\langle h, -i\partial_sQ_+\rangle_{\dot{H}^{1}(\heis)\times\dot{H}^{-1}(\heis)}
	=\sqrt{2}\pi^2F_h(i ).
\end{equation*}

In particular,
$$
\langle h,\partial_sQ_+\rangle_{\dot{H}^{1}(\heis)}
	=-\langle \partial_s h, -i\partial_sQ_+\rangle_{\dot{H}^{1}(\heis)\times\dot{H}^{-1}(\heis)}
	=-\sqrt{2}\pi^2F_h'(i ).
$$
\end{proof}


We now check that $\L h$, $h\in\dot{H}^1(\heis)$, decomposes in the Hilbert space $\dot{H}^{-1}(\heis)$ as an orthogonal sum $\L h=\L_{|V}h+\L_{|V^\perp}h$, where $V^\perp$ is the orthogonal of $V$ in $\dot{H}^1(\heis)$.

\begin{cor}\label{cor:decomposition_h+h-}
Let $h\in\dot{H}^1(\heis)$ and decompose $h$ as $h=h_0+h_-+h_+$, where $h_0\in\dot{H}^1(\heis)\cap V_0^+\cap \Vect_{\R}(\partial_sQ_+,iQ_+,Q_++2i\partial_s Q_+)$, $h_-\in\dot{H}^1(\heis)\cap V_0^+\cap \Vect_{\R}(Q_+)$ and $h_+\in \dot{H}^1(\heis)\cap V_0^+\cap (Q_+,iQ_+,\partial_s Q_+,i\partial_sQ_+)^{\perp,\dot{H}^1(\heis)}$. Then
$$
(\L h,h)_{\dot{H}^{-1}(\heis)\times\dot{H}^{1}(\heis)}
	=(\L h_+,h_+)_{\dot{H}^{-1}(\heis)\times\dot{H}^{1}(\heis)}
	+(\L h_-,h_-)_{\dot{H}^{-1}(\heis)\times\dot{H}^{1}(\heis)}
$$
and
$$
\|\L h\|_{\dot{H}^{-1}(\heis)}^2
	=\|\L h_+\|_{\dot{H}^{-1}(\heis)}^2
	+\|\L h_-\|_{\dot{H}^{-1}(\heis)}^2.
$$
\end{cor}

\begin{proof}
We decompose $\L h$ as $\L h=\L h_++\L h_-$.

Let us show that $\L h_+$ is orthogonal to $Q_+$, $iQ_+$, $\partial_s Q_+$ and $i\partial_s Q_+$ for the duality product  $\dot{H}^{-1}(\heis)\times\dot{H}^{1}(\heis)$. Let us treat separately each term of
$$
\L h_+=- \hlapl h_+-2\Pi_0^+(|Q_+|^2h_+)-\Pi_0^+(Q_+^2\overline{h_+}).
$$
By assumption on $h_+$, $-\hlapl h_+=D_sh_+$ and
$$
\langle D_sh_+, Q_+\rangle_{\dot{H}^{-1}(\heis)\times\dot{H}^{1}(\heis)}
	=\langle h_+, Q_+\rangle_{\dot{H}^{1}(\heis)}
	=0
$$
and
$$
\langle D_sh_+, \partial_sQ_+\rangle_{\dot{H}^{-1}(\heis)\times\dot{H}^{1}(\heis)}
	=\langle h_+, \partial_sQ_+\rangle_{\dot{H}^{1}(\heis)}
	=0.
$$
Moreover, using Proposition \ref{prop:formula_orthogonal},
\begin{align*}
\langle \Pi_0^+(|Q_+|^2h_+),Q_+\rangle_{\dot{H}^{-1}(\heis)\times\dot{H}^{1}(\heis)}
	&=\langle Q_+ h_+,Q_+^2\rangle_{L^2(\heis)\times L^2(\heis)}\\
	&=\langle Q_+ h_+,-i\sqrt{2}\partial_s Q_+\rangle_{L^2(\heis)\times L^2(\heis)}\\
	&=2\pi^2F_{Q_+h_+}(i)\\	
	&=0,
\end{align*}
since $F_{Q_+h_+}=F_{Q_+}F_{h_+}$ and $F_{h_+}(i)=0$. In the same way,
\begin{align*}
\langle \Pi_0^+(Q_+^2\overline{h_+}),Q_+\rangle_{\dot{H}^{-1}(\heis)\times\dot{H}^{1}(\heis)}
	&=\langle Q_+^2,Q_+h_+\rangle_{L^2(\heis)\times L^2(\heis)}\\
	&=0.
\end{align*}

Finally,
\begin{align*}
\langle \Pi_0^+(|Q_+|^2h_+),\partial_sQ_+\rangle_{\dot{H}^{-1}(\heis)\times\dot{H}^{1}(\heis)}
	&=\half \langle Q_+ h_+,\partial_s(Q_+^2)\rangle_{L^2(\heis)\times L^2(\heis)}\\
	&=-\half \langle \partial_s(Q_+ h_+),Q_+^2\rangle_{\dot{H}^{-2}(\heis)\times \dot{H}^2(\heis)}\\
	&=-\frac{1}{2}\langle \partial_s(Q_+ h_+),-i\sqrt{2}\partial_s Q_+\rangle_{\dot{H}^{-2}(\heis)\times \dot{H}^2(\heis)}\\
	&=-\pi^2F_{Q_+h_+}'(i)\\	
	&=0,
\end{align*}
and in the same way,
\begin{align*}
\langle \Pi_0^+(Q_+^2\overline{h_+}),\partial_s Q_+\rangle_{\dot{H}^{-1}(\heis)\times\dot{H}^{1}(\heis)}
	&=\langle Q_+^2,\partial_s (Q_+)h_+\rangle_{L^2(\heis)\times L^2(\heis)}\\
	&=2\pi^2\overline{F_{\partial_s(Q_+)h_+}(i)}\\
	&=0.
\end{align*}

Therefore, $\L h_+\in \dot{H}^{-1}(\heis)\cap V_0^+\cap (Q_+,iQ_+,\partial_s Q_+,i\partial_sQ_+)^{\perp, L^2(\heis)}$, where the orthogonal is taken for the duality product $\dot{H}^{-1}(\heis)\times\dot{H}^{1}(\heis)$. In particular,
$$
(\L h_+,h)_{\dot{H}^{-1}(\heis)\times\dot{H}^1(\heis)}
	=(\L h_+,h_+)_{\dot{H}^{-1}(\heis)\times\dot{H}^1(\heis)}.
$$

Now, since $\L h_-$ is in $\Vect_{\R}(i\partial_s Q_+)$, write $\L h_-=\lambda i\partial_s Q_+=\lambda \hlapl Q_+$ for some real number $\lambda$. One has
\begin{align*}
(\L h_-,h)_{\dot{H}^{-1}(\heis)\times\dot{H}^1(\heis)}
	&=-\lambda(Q_+,h)_{\dot{H}^1(\heis)}\\
	&=-\lambda(Q_+,h_-)_{\dot{H}^1(\heis)}\\
	&=(\L h_-,h_-)_{\dot{H}^{-1}(\heis)\times\dot{H}^1(\heis)},
\end{align*}
which gives the first part of the proposition.

Then,
$$
(\L h_+,\L h_-)_{\dot{H}^{-1}(\heis)}
	=(\L h_+,-\lambda Q_+)_{\dot{H}^{-1}(\heis)\times\dot{H}^1(\heis)}
	=0,
$$
so we conclude that
$$
\|\L h\|_{\dot{H}^{-1}(\heis)}^2
	=\|\L h_+\|_{\dot{H}^{-1}(\heis)}^2+\|\L h_-\|_{\dot{H}^{-1}(\heis)}^2.
$$
\end{proof}

We now give a simplified expression of $(\L h,h)_{\dot{H}^{-1}(\heis)\times\dot{H}^{1}(\heis)}$ when $h$ is orthogonal to $Q_+$ and $iQ_+$.

\begin{prop}\label{prop:L_on_S3}
For $h\in\dot{H}^1(\heis)\cap V_0^+ \cap \Vect_{\R}(Q_+,iQ_+)^{\perp,\dot{H}^1(\heis)}$, the following identity is true
$$
(\L h,h)_{\dot{H}^{-1}(\heis)\times\dot{H}^{1}(\heis)}
	=\langle-\hlapl h-2|Q_+|^2h,h\rangle_{\dot{H}^{-1}(\heis)\times\dot{H}^{1}(\heis)}
$$
\end{prop}
Note that it is more convenient to switch to a complex scalar product because $ - \hlapl h-2|Q_+|^2h$ is a complex linear operator of the variable $h$.

\begin{proof}
We only have to show that $(\Pi_0^+(Q_+^2\overline{h}),h)_{\dot{H}^{-1}(\heis)\times\dot{H}^{1}(\heis)}
$ is zero. We calculate
\begin{align*}
(\Pi_0^+(Q_+^2\overline{h}),h)_{\dot{H}^{-1}(\heis)\times\dot{H}^{1}(\heis)}
   &= (Q_+^2\overline{h},h)_{\dot{H}^{-1}(\heis)\times\dot{H}^{1}(\heis)}\\
   &= (Q_+^2,h^2)_{\dot{H}^{-1}(\heis)\times\dot{H}^{1}(\heis)}\\
   &=(-i\sqrt{2}\partial_sQ_+,h^2)_{\dot{H}^{-1}(\heis)\times\dot{H}^{1}(\heis)}\\
   &=2\pi^2\Re(F_{h^2}(i)).
\end{align*}
Now, $F_{h^2}=F_h^2$, therefore, $F_{h^2}(i)=0$ as soon as $h\in \dot{H}^1(\heis)\cap V_0^+ \cap \Vect_{\R}(Q_+,iQ_+)^{\perp,\dot{H}^1(\heis)}$.
\end{proof}

\subsection{Study of the limiting profile through the Cayley transform}\label{subsection:cayley_transform}

We now study the spectrum of $-\hlapl-2|Q_+|^2$, which is now natural since we search for a coercivity estimate on $\L$ and we just proved (Proposition \ref{prop:L_on_S3}) that
$$
(\L h,h)_{\dot{H}^{-1}(\heis)\times\dot{H}^{1}(\heis)}
	=\langle-\hlapl h-2|Q_+|^2h,h\rangle_{\dot{H}^{-1}(\heis)\times\dot{H}^{1}(\heis)}.
$$
This spectrum can be determined via the equivalence between the Heisenberg group $\heis$ and the CR sphere $\S^{3}$ in $\C^2$ called the Cayley transform. We rely on \cite{Branson2013} in order to introduce this equivalence and its spectral consequences. In this part, we will denote by $(w,s)$ the elements of the Heisenberg group, bearing in mind that $w=x+iy$ with the former notations. The Cayley transform writes
$$
\begin{matrix}
\cayley: &\heis & \to & \S^3\setminus (0,-1)\\
~ & (w,s) & \mapsto & (\frac{2w}{1+|w|^2+is},\frac{1-|w|^2-is}{1+|w|^2+is})
\end{matrix}.
$$
The inverse of $\cayley$ is $\cayley^{-1}(\zeta_1,\zeta_2)=(\frac{\zeta_1}{1+\zeta_2},\Im(\frac{1-\zeta_2}{1+\zeta_2}))$.
The Jacobian of the Cayley transform is
$$
|\Jc(w,s)|
	=\frac{8}{\left((1+|w|^2)^2+s^2\right)^2}.
$$
Notice that
 $|\Jc|$ is linked to $Q_+$ as follows
$$
|\Jc(x+iy,s)|
	=2|Q_+(x,y,s)|^4.
$$
For any integrable function $F$ on $\S^3$, we have the relation
$$
\int_{\S^3}F\d\zeta=\int_{\heis}(F\circ \cayley)|\Jc| \d \lambda_3(w,s).
$$
Here, $\d\zeta$ denotes the standard Euclidean volume element of $\S^3$. We consider the complex scalar product on $L^2(\S^3)$
$$
\langle F,G\rangle_{L^2(\S^3)}=\int_{\S^3}F\overline{G}\d\zeta, \quad F,G\in L^2(\S^3).
$$
One can notice that
\begin{align*}
\int_{\S^3}|F|^2\d\zeta
	&=\int_{\heis}|F\circ \cayley|^2|\Jc| \d \lambda_3(w,s).
\end{align*}
In particular, $|J_C|=2|Q_+|^4$ is in $L^2(\heis)$, so if a function $F$ is such that $F\circ\cayley$ belongs to $L^4(\heis)$ (for example if $F\circ\cayley\in \dot{H}^1(\heis)$), then $|F\circ \cayley|^2$ belongs to $L^2(\heis)$, and therefore $F$ is in $L^2(\S^3)$.

On the standard sphere $\S^3$, denote
$$
\mathcal{R}=\zeta_1\partial_{\zeta_1}+\zeta_2\partial_{\zeta_2}.
$$
Then the vector fields
$$
T_i=\partial_{\zeta_i}-\overline{\zeta_i}\mathcal{R}, \quad i=1,2,
$$
generate the holomorphic tangent space to $\S^3$.

The conformal sub-Laplacian is defined as
$$
\Dlapl=-\half\sum_{i=1}^2(T_i\overline{T_i}+\overline{T_i}T_i)+\frac{1}{4},
$$
where $\Dlapl-\frac{1}{4}$ is the sub-Laplacian. One can construct the Sobolev space
$$
H^1(\S^3):=\{v\in L^2(\S^3);\; \|v\|_{H^1(\S^3)}:=\|\Dlapl^\half v\|_{L^2(\S^3)}<+\infty\}.
$$

The operator $\Dlapl$ on the sphere has a direct link with the sub-Laplacian on the Heisenberg group via the Cayley transform~: for any radial function $F\circ\cayley$ in $\dot{H}^1(\heis)$,
\begin{equation*}
-\hlapl\left((2|\Jc|)^{\frac{1}{4}}(F\circ \cayley)\right)
	=(2|\Jc|)^{\frac{3}{4}}(\Dlapl F)\circ \cayley.
\end{equation*}

Notice that a function in  $\dot{H}^1(\heis)$ maps to a function in  $H^1(\S^3)$ via the following transformation.
\begin{prop}\label{prop:lapl_on_S3}
Let $h$ be a function on $\heis$, and define a function $v_h$ on $\S^3$ by
\begin{equation}\label{eq:vh}
h(x,y,s)
	=(2|\Jc|)^{\frac{1}{4}}(v_h\circ \cayley)(x+iy,s)
	=\sqrt{2}|Q_+|(v_h\circ \cayley)(x+iy,s).
\end{equation}
Then for radial $h$,
$$
\langle\Dlapl v_h,v_h\rangle_{L^2(\S^3)}
	=\half \langle -\hlapl h,h\rangle_{\dot{H}^{-1}(\heis)\times\dot{H}^1(\heis)}
$$
and
\begin{equation*}
\langle v_h,v_h\rangle_{L^2(\S^3)}
	=\int_{\heis} |h|^2|Q_+|^2\d\lambda_3.
\end{equation*}
Therefore, $v_h$ defines a function in $H^1(\S^3)$ if and only if $h$ is in $\dot{H}^1(\heis)$.
\end{prop}

\begin{proof}
Fix a radial function $h$, and define $v_h$ by \eqref{eq:vh}.
Then
\begin{align*}
(-\hlapl h)\cdot\overline{h}
	&=(2|\Jc|)^{\frac{3}{4}}(\Dlapl v_h)\circ \cayley
	\cdot \overline{(2|\Jc|)^{\frac{1}{4}}(v_h\circ \cayley)}\\
	&=2|\Jc|(\Dlapl v_h)\circ \cayley \cdot \overline{v_h\circ \cayley},
\end{align*}
so
$$
\langle-\hlapl h,h\rangle_{\dot{H}^{-1}(\heis)\times\dot{H}^1(\heis)}=2\langle\Dlapl v_h,v_h\rangle_{L^2(\S^3)}.
$$
Moreover, when $h\in L^4(\heis)$, then $v_h\in L^2(\S^3)$ and
\begin{align*}
\langle v_h,v_h\rangle_{L^2(\S^3)}
	&=\int_{\S^3} |v_h|^2\d\zeta\\
	&=\int_{\heis} |v_h\circ \cayley|^2|\Jc|\d\lambda_3(w,s)\\
	&=\frac{1}{\sqrt{2}}\int_{\heis} |h|^2|\Jc|^\half\d\lambda_3(w,s)\\
	&=\int_{\heis} |h|^2|Q_+|^2\d\lambda_3.
\end{align*}
\end{proof}

Propositions \ref{prop:L_on_S3} and \ref{prop:lapl_on_S3} combined imply the following corollary.

\begin{cor}
Let $h$ in $\dot{H}^1(\heis)\cap V_0^+\cap (Q_+,iQ_+)^{\perp,\dot{H}^1(\heis)}$. Then
$$
(\L h,h)_{\dot{H}^{-1}(\heis)\times\dot{H}^{1}(\heis)}
	=2\langle\Dlapl v_h,v_h\rangle_{L^2(\S^3)}-2\langle v_h,v_h\rangle_{L^2(\S^3)}.
$$
\end{cor}

The spectrum of the operator $\Dlapl$ on $H^1(\S^3)$ is well known.
Indeed, the space $L^2(\S^3)$ endowed with the inner product $\langle F,G\rangle_{L^2(\S^3)}=\int_{\S^3}F\overline{G}\d\zeta$ admits the orthogonal decomposition
$$
L^2(\S^3)=\bigoplus_{j,k\geq 0}\harm_{j,k},
$$
where $\harm_{j,k}$ is the space of harmonic polynomials on $\C^2$ that are homogeneous of degree $j$ in $\zeta_1,\zeta_2$ and $k$ in $\overline{\zeta_1},\overline{\zeta_2}$, restricted to the sphere $\S^3$. Fix $j,k\geq 0$, then the dimension of $\harm_{j,k}$ is
$$
m_{j,k}:=\dim(\harm_{j,k})=j+k+1.
$$
The spectrum of $\Dlapl$ is as follows \cite{Stanton1989}.

\begin{prop}
Let $\lambda_j=j+\half$. Then for all $Y_{j,k}\in \harm_{j,k}$,
$$
\Dlapl Y_{j,k}=\lambda_j\lambda_k Y_{j,k}.
$$
\end{prop}
In particular, the smallest eigenvalue of $\Dlapl-\id$ is $\lambda_{0,0}-1=-\frac{3}{4}$, with multiplicity $1$ and eigenvectors the constant functions on $\S^3$. The second one is also negative, equal to $ \lambda_{1,0}-1=\lambda_{0,1}-1=-\frac{1}{4}$, with eigenvectors spanned by $\zeta_1,\zeta_2,\overline{\zeta_1},\overline{\zeta_2}$. The third one is positive, equal to $\lambda_{2,0}-1=\lambda_{0,2}-1=\frac{1}{4}$.

Let us study the radial property on $\S^3$. Let $h\in\dot{H}^1(\heis)$ be a radial function, $v_h$ as in \eqref{eq:vh}
$$
h(x,y,s)
	=(2|\Jc|)^{\frac{1}{4}}(v_h\circ \cayley)(x+iy,s).
$$
Since $h$ and $|\Jc|$ only depend on $|x+iy|$ and $s$, so does $v_h\circ \cayley$, which means that $v_h$ only depends on $|\zeta_1|$, $\zeta_2$ and $\overline{\zeta_2}$. This discards the eigenfunctions $\zeta_1$ and $\overline{\zeta_1}$ in the above orthogonal decomposition of $v_h$.

The last step left is to treat the remaining eigenvectors with negative eigenvalues for the operator $\Dlapl-\id$, in order to find a lower bound in the quadratic form
$$
(\L h,h)_{\dot{H}^{-1}(\heis)\times\dot{H}^{1}(\heis)}
	=2\langle\Dlapl v_h,v_h\rangle_{L^2(\S^3)}
	-2\langle v_h,v_h\rangle_{L^2(\S^3)}
$$
for $h\in \dot{H}^1(\heis)\cap V_0^+\cap\Vect_{\R}(Q_+,iQ_+,\partial_s Q_+,i\partial_sQ_+)^{\perp,\dot{H}^1(\heis)}.$ These eigenvectors are the constant function $e_1=\un$ (with eigenvalue $-\frac{3}{4}$) and the harmonic polynomials $e_2=\overline{\zeta_2}$ and $e_3=\zeta_2$ (with eigenvalue $-\frac{1}{4}$). In order to do so, we reformulate the above spectral study back to the setting of holomorphic functions of the upper complex plane.

For fractional Sobolev embeddings on $\R^n$ and fractional Folland-Stein embeddings on $\mathbb{H}^n$ (\cite{ChenFrankWeth2013} and \cite{LiuZhang2015}), the potential negative eigenvalues are naturally discarded by the orthogonality conditions, since they correspond to the tangent space to the manifold of functions equal, up to translation, dilation and multiplication by a non-zero constant, to the respective optimizers $U$ and $H$~:
$$
\mathcal{M}(\R^n)=\left\{c U\left(\frac{\cdot-x_0}{\varepsilon}\right);\quad c\in\R^*,x_0\in\R^n,\varepsilon>0\right\}
$$
resp.
$$
\mathcal{M}(\heis)=\left\{c H(\delta(u\cdot));\quad c\in\R^*,u\in\mathbb{H}^n,\delta>0\right\}.
$$

\subsection{Coercivity of the linearized operator}\label{subsection:coercivity}

In this part, we use the spectrum of $\Dlapl$ on the CR sphere in order to get a coercivity estimate on $\L$. The lowest eigenvalues of $\Dlapl-\id$ are, in increasing order,
$$
\lambda_{0,0}-1=-\frac{3}{4}
	< \lambda_{0,1}-1=\lambda_{1,0}-1=-\frac{1}{4}
	< \lambda_{0,2}-1=\lambda_{2,0}-1=\frac{1}{4}.
$$
The negative eigenfunctions are $e_1=\un$ (for $\lambda_{0,0}$), $e_2=\overline{\zeta_2}$ (for $\lambda_{0,1}$) and $e_3=\zeta_2$ (for $\lambda_{1,0}$).

Let $h\in \dot{H}^1(\heis)\cap V_0^+\cap\Vect_{\R}(Q_+,iQ_+,\partial_s Q_+,i\partial_sQ_+)^{\perp,\dot{H}^1(\heis)},$ $v$ as in \eqref{eq:vh}
$$
h(x,y,s)
	=\sqrt{2}|Q_+|(v\circ \cayley)(x+iy,s).
$$
Then decompose $v$ as~:
$$
v
	=v_+
	+\frac{\langle v,e_1\rangle_{L^2(\S^3)}}{\langle e_1,e_1\rangle_{L^2(\S^3)}}e_1
	+\frac{\langle v,e_2\rangle_{L^2(\S^3)}}{\langle e_2,e_2\rangle_{L^2(\S^3)}}e_2
	+\frac{\langle v,e_1\rangle_{L^2(\S^3)}}{\langle e_3,e_3\rangle_{L^2(\S^3)}}e_3,\quad
v_+\in\Vect_{\C}(e_1,e_2,e_3)^{\perp}.
$$
Remark that since $e_1\in \textnormal{Ha}_{0,0}$, $e_2\in \textnormal{Ha}_{0,1}$ and $e_3\in \textnormal{Ha}_{1,0}$, these three vectors are pairwise orthogonal in $L^2(\S^3)$, and they are orthogonal to $\bigoplus_{(j,k)\not\in \{(0,0),(0,1),(1,0)\}}\textnormal{Ha}_{j,k}$. The knowledge of the eigenvalues of $\Dlapl -\id$ enables us to say that
\begin{align*}
\half (\L h,h)_{\dot{H}^{-1}(\heis)\times\dot{H}^1(\heis)}
	&=\langle\Dlapl v,v\rangle_{L^2(\S^3)}-\langle v,v\rangle_{L^2(\S^3)}\\
	&\geq \frac{1}{4}\|v_+\|_{L^2(\S^3)}^2
	-\frac{1}{4}\frac{\left| \langle v,e_1\rangle_{L^2(\S^3)}\right|^2}{\langle e_1,e_1\rangle_{L^2(\S^3)}}
	-\frac{3}{4}\frac{\left| \langle v,e_2\rangle_{L^2(\S^3)}\right|^2}{\langle e_2,e_2\rangle_{L^2(\S^3)}}
	-\frac{3}{4}\frac{\left| \langle v,e_3\rangle_{L^2(\S^3)}\right|^2}{\langle e_3,e_3\rangle_{L^2(\S^3)}}.
\end{align*}
But
$$
\|v\|_{L^2(\S^3)}^2
	=\|v_+\|_{L^2(\S^3)}^2
	+\frac{\left| \langle v,e_1\rangle_{L^2(\S^3)}\right|^2}{\langle e_1,e_1\rangle_{L^2(\S^3)}}
	+\frac{\left| \langle v,e_2\rangle_{L^2(\S^3)}\right|^2}{\langle e_2,e_2\rangle_{L^2(\S^3)}}
	+\frac{\left| \langle v,e_3\rangle_{L^2(\S^3)}\right|^2}{\langle e_3,e_3\rangle_{L^2(\S^3)}},
$$
so
\begin{align*}
\half (\L h,h)_{\dot{H}^{-1}(\heis)\times\dot{H}^1(\heis)}
	&\geq \frac{1}{4}\|v\|_{L^2(\S^3)}^2
	-\frac{1}{2}\frac{\left| \langle v,e_1\rangle_{L^2(\S^3)}\right|^2}{\langle e_1,e_1\rangle_{L^2(\S^3)}}
	-\frac{\left| \langle v,e_2\rangle_{L^2(\S^3)}\right|^2}{\langle e_2,e_2\rangle_{L^2(\S^3)}}
	-\frac{\left| \langle v,e_3\rangle_{L^2(\S^3)}\right|^2}{\langle e_3,e_3\rangle_{L^2(\S^3)}}.
\end{align*}

Let us replace these last terms by their expression on the Heisenberg group. We define
$$
f_j=\sqrt{2}|Q_+|e_j\circ\cayley, \quad j=1,2,3.
$$
From the identity
\begin{align*}
\zeta_2\circ\cayley(w,s)
	=\frac{1-|w|^2-is}{1+|w|^2+is}
	=\sqrt{2}\overline{Q_+}(w,s)-\un,
\end{align*}
we get that
$$
f_1=\sqrt{2}|Q_+|,
$$
$$
f_2=\sqrt{2}|Q_+|(\sqrt{2}Q_+-1),
$$
and
$$
f_3=\sqrt{2}|Q_+|(\sqrt{2}\overline{Q_+}-1).
$$

Thanks to Proposition \ref{prop:lapl_on_S3}, one knows that
$$
\langle v,v\rangle_{L^2(\S^3)}=\langle hQ_+,hQ_+\rangle_{L^2(\heis)},
$$
so
\begin{multline*}
(\L h,h)_{\dot{H}^{-1}(\heis)\times\dot{H}^1(\heis)}
	\geq \frac{1}{2}\|hQ_+\|_{L^2(\heis)}^2
	-\frac{\left| \langle hQ_+,f_3Q_+\rangle_{L^2(\heis)}\right|^2}{\| f_3Q_+\|_{L^2(\heis)}^2}
	-\frac{\left| \langle hQ_+,f_2Q_+\rangle_{L^2(\heis)}\right|^2}{\| f_2Q_+\|_{L^2(\heis)}^2}\\
	-2\frac{\left| \langle hQ_+,f_1Q_+\rangle_{L^2(\heis)}\right|^2}{\| f_1Q_+\|_{L^2(\heis)}^2}.
\end{multline*}
For $h\in\dot{H}^1(\heis) \cap V_0^+ \cap \Vect_{\R}(Q_+,iQ_+,\partial_s Q_+,i\partial_sQ_+)^{\perp,\dot{H}^1(\heis)}$, let us consider the space in which $F_{hQ_+}$ lies.

Since $h\in\dot{H}^1(\heis)$, from the embedding $\dot{H}^1(\heis)\hookrightarrow L^4(\heis)$, one knows that $hQ_+$ is in $L^2(\heis)$ so $F_{hQ_+}$ belongs to $L^2(\C_+)$.

From part \ref{subsection:bergman_spaces}, $h$ being in $\dot{H}^1(\heis)\cap V_0^+$, $F_h$ (defined by $h(x,y,s)=F_h(s+i|x+iy|^2)$ for $(x,y,s)\in\heis$) is a holomorphic function ($F_h$ lies in the Hardy space $\hardy^2(\C_+)$). This implies that the function $F_{hQ_+}=F_hF_{Q_+}$ is holomorphic too~: we have shown that $F_h$ is in the Bergman space $A^2_1=L^2(\C_+)\cap \holo(\C_+)$.

Moreover, the fact that $h$ is orthogonal to $\Vect_{\R}(Q_+,iQ_+,\partial_sQ_+,i\partial_s Q_+)$ in $\dot{H}^1(\heis)$ is equivalent by Proposition \ref{prop:formula_orthogonal} to $F_h(i)=F_h'(i)=0$. But then, $F_{hQ_+}=F_hF_{Q_+}$ has a double zero at~$i$. Proposition \ref{prop:formula_orthogonal} again implies that
$$
\langle hQ_+,\partial_sQ_+\rangle_{L^2(\heis)}=\langle hQ_+,\partial_s^2Q_+\rangle_{L^2(\heis)}=0,
$$
which is equivalent to
$$
\langle F_{hQ_+},F_{Q_+}'\rangle_{L^2(\C_+)}=\langle F_{hQ_+},F_{Q_+}''\rangle_{L^2(\C_+)}=0.
$$

Now, define $W:=A^2_1\cap \Vect_{\C}(F_{Q_+}',F_{Q_+}'')^{\perp,L^2(\C_+)}$ and denote by $P_W$ the orthogonal projection from $L^2(\C_+)$ onto $W$. We have shown that if  $h\in\dot{H}^1(\heis)\cap V_0^+\cap\Vect_{\R}(Q_+,iQ_+,\partial_s Q_+,i\partial_sQ_+)^{\perp,\dot{H}^1(\heis)}$, then $F_{hQ_+}\in W$. In particular, for $u\in L^2(\heis)$,
\begin{align*}
\langle hQ_+,u\rangle_{L^2(\heis)}
	&=\pi\langle F_{hQ_+},F_u\rangle_{L^2(\C_+)}\\
	&=\pi\langle F_{hQ_+},P_W(F_u)\rangle_{L^2(\C_+)}.
\end{align*}

Back to the quadratic form, we deduce that
\begin{multline*}
(\L h,h)_{\dot{H}^{-1}(\heis)\times\dot{H}^1(\heis)}
	\geq
	\pi\Big(
	\frac{1}{2}\|F_{hQ_+}\|_{L^2(\C_+)}^2
	-\frac{\left| \langle F_{hQ_+},P_W(F_{f_3Q_+})\rangle_{L^2(\C_+)}\right|^2}{\| F_{f_3Q_+}\|_{L^2(\C_+)}^2}\\
	-\frac{\left| \langle F_{hQ_+},P_W(F_{f_2Q_+})\rangle_{L^2(\C_+)}\right|^2}{\| F_{f_2Q_+}\|_{L^2(\C_+)}^2}
	-2\frac{\left| \langle F_{hQ_+},P_W(F_{f_1Q_+})\rangle_{L^2(\C_+)}\right|^2}{\| F_{f_1Q_+}\|_{L^2(\C_+)}^2}
	\Big).
\end{multline*}
Let us denote
$$
X_j=\frac{P_W(F_{f_j Q_+})}{\| F_{f_j Q_+}\|_{L^2(\C_+)}}=\frac{P_W(F_{1})}{\| F_{1}\|_{L^2(\C_+)}}, \quad j=1,2,3,
$$
with
$$
F_1(z)=\frac{1}{|z+i|(z+i)},
$$
$$
F_2(z)=\frac{1}{|z+i|(z+i)}\Big(\frac{2i}{z+i}-1\Big),
$$
and
$$
F_3(z)=\frac{1}{|z+i|(z+i)}\Big(\frac{-2i}{\overline{z}-i}-1\Big).
$$

We try to find an upper bound on the quadratic form on $L^2(\C_+)$
$$
q(F):=
	2\left| \langle F,X_1\rangle_{L^2(\C_+)}\right|^2
	+\left| \langle F,X_2\rangle_{L^2(\C_+)}\right|^2
	+\left| \langle F,X_3\rangle_{L^2(\C_+)}\right|^2,
	\quad F\in L^2(\C_+).
$$
In particular, we want to show that this upper bound is strictly less than $\half$.

Let us first write explicitly the orthogonal projector $P_W$ from  $L^2(\C_+)$ onto the subspace $W=A^2_1\cap \Vect_{\C}(F_{Q_+}',F_{Q_+}'')^{\perp,L^2(\C_+)}$
We start by finding an orthogonal basis of $\Vect_{\C}(F_{Q_+}',F_{Q_+}'')$ for the scalar product on $L^2(\C_+)$. We know by Proposition \ref{prop:formula_orthogonal} that
$$
\langle u,\partial_sQ_+\rangle_{L^2(\heis)}
	=-i\sqrt{2}\pi^2F_u(i ), \quad u\in L^2(\heis),
$$
so
$$
\langle F,F_{Q_+}'\rangle_{L^2(\C_+)}
	=-i\sqrt{2}\pi F(i ), \quad F\in L^2(\C_+).
$$
Recall that
$$
F_{Q}(z)=\frac{i\sqrt{2}}{z+i},
\quad
F_{Q}'(z)=\frac{-i\sqrt{2}}{(z+i)^2},
\quad
F_{Q}''(z)=\frac{2i\sqrt{2}}{(z+i)^3},
\quad
\text{and}
\quad
F_{Q_+}'''(z)=-\frac{6i\sqrt{2}}{(z+i)^4},
$$
so
$$
F_{Q_+}(i)=\frac{1}{\sqrt{2}},
\quad
F_{Q_+}'(i)=\frac{i}{2\sqrt{2}},
\quad
F_{Q_+}''(i)=-\frac{1}{2\sqrt{2}}
\quad
\text{and}
\quad
F_{Q_+}'''(i)=\frac{3i}{4\sqrt{2}}.
$$
Therefore,
\begin{align*}
\langle F_{Q_+}'',F_{Q_+}'\rangle_{L^2(\C_+)}
	=-i\sqrt{2}\pi F_{Q_+}''(i)
	=i\frac{\pi}{2}.
\end{align*}
In the same way,
\begin{align*}
\langle F_{Q_+}',F_{Q_+}'\rangle_{L^2(\C_+)}
	=-i\sqrt{2}\pi F_{Q_+}'(i)
	=\frac{\pi}{2},
\end{align*}
so
$
\tilde{F}
	:=F_{Q_+}'-\frac{\langle F_{Q_+}', F_{Q_+}'\rangle_{L^2(\C_+)}
}{\langle F_{Q_+}'',F_{Q_+}'\rangle_{L^2(\C_+)}
}F_{Q_+}''
	=F_{Q_+}'+iF_{Q_+}''
$
is orthogonal to $F_{Q_+}'$~:
$$
\langle \tilde{F},F_{Q_+}'\rangle_{L^2(\C_+)}
	=0.
$$
Moreover,
\begin{align*}
\langle \tilde{F},\tilde{F}\rangle_{L^2(\C_+)}
	&=\langle \tilde{F},F_{Q_+}'\rangle_{L^2(\C_+)}
+\langle \tilde{F},iF_{Q_+}''\rangle_{L^2(\C_+)}\\
	&=0+\langle i\tilde{F}',F_{Q_+}'\rangle_{L^2(\C_+)}\\
	&=\sqrt{2}\pi\tilde{F}'(i).
\end{align*}
Since
$
\tilde{F}'(i)
	=F_{Q_+}''(i)+iF_{Q_+}'''(i)
	=\frac{1}{4\sqrt{2}},
$
 $\tilde{F}$ is of norm
$$
\langle \tilde{F},\tilde{F}\rangle_{L^2(\C_+)}
	=\frac{\pi}{4}.
$$

The orthogonal projection on $\Vect_{\C}(F_{Q_+}',F_{Q_+}'')^{\perp,L^2(\C_+)}$ in $L^2(\C_+)$ then writes
$$
F\in L^2(\C_+)
	\mapsto F
	-\frac{2}{\pi}\langle F,F_{Q_+}'\rangle_{L^2(\C_+)}F_{Q_+}'
	-\frac{4}{\pi}\langle F,F_{Q_+}'+iF_{Q_+}''\rangle_{L^2(\C_+)}(F_{Q_+}'+iF_{Q_+}'').
$$
Besides, from Proposition \ref{prop:formula_P0}, we know that the orthogonal projection $P_0$ from $\in L^2(\C_+)$ onto $A^2_1$ is given by
\begin{equation*}
P_0F(s+it)
	=-\frac{1}{\pi}\int_{\C_+}\frac{1}{(s-u+it+iv)^2}F(u+iv)\d u\d v,\quad F\in L^2(\C_+).
\end{equation*}
Therefore, the orthogonal projection $P_W$ on the space $W=A^2_1\cap \Vect_{\C}(F_{Q_+}',F_{Q_+}'')^{\perp,L^2(\C_+)}$ writes, for $F \in L^2(\C_+)$,
\begin{multline*}
P_WF(s+it)
	=-\frac{1}{\pi}\int_{\C_+}\frac{1}{(s-u+it+iv)^2}F(u+iv)\d u\d v
	-\frac{2}{\pi}\langle F,F_{Q_+}'\rangle_{L^2(\C_+)}F_{Q_+}'(s+it)\\
	-\frac{4}{\pi}\langle F,F_{Q_+}'+iF_{Q_+}''\rangle_{L^2(\C_+)}(F_{Q_+}'+iF_{Q_+}'')(s+it).
\end{multline*}

We use the following estimates of $\langle \pi P_0F_j,F_j\rangle_{L^2(\C_+)}$, $j=1,2,3$.

\begin{lem}\label{lem:calculation_proj}
Set $\varepsilon=10^{-10}$, then
$$
|\langle \pi P_0F_1,F_1\rangle_{L^2(\C_+)}-2|\leq \varepsilon,
$$
$$
|\langle \pi P_0F_2,F_2\rangle_{L^2(\C_+)}-\frac{10}{9}|\leq \varepsilon
$$
and
$$
|\langle \pi P_0F_3,F_3\rangle_{L^2(\C_+)}-0.1303955989|\leq \varepsilon.
$$
\end{lem}

The proof of this lemma is rather technical and postponed to Appendix \ref{appendix:calculation_proj}. It involves simplifying the integrals defining $P_0F_j$, $j=1,2,3$ : we determine explicitly the holomorphic function which coincides with $P_0F_j$ on $\C_+$ thanks to a massive use of the residue formula. This part is necessary in order to compute numerically $\langle P_0F_j,F_j\rangle_{L^2(\C_+)}$. Without this preliminary work, there is a four-dimensional numerical integration to perform and the error estimate is big with a naive approach.

Moreover, a direct calculation gives
$$
\langle F_1,F_1\rangle_{L^2(\C_+)}=\frac{\pi}{4}, \quad
\langle F_2,F_2\rangle_{L^2(\C_+)}=\langle F_3,F_3\rangle_{L^2(\C_+)}=\frac{\pi}{8},
$$
$$
\langle F_1,F_{Q_+}'\rangle_{L^2(\C_+)}=-\frac{2 \sqrt{2}}{3}, \quad
\langle F_2,F_{Q_+}'\rangle_{L^2(\C_+)}=-\frac{2 \sqrt{2}}{9}, \quad
\langle F_3,F_{Q_+}'\rangle_{L^2(\C_+)}=\frac{2 \sqrt{2}}{15},
$$
and
$$
\langle F_1,\tilde{F}\rangle_{L^2(\C_+)}=-\frac{2 \sqrt{2}}{15},\quad
\langle F_2,\tilde{F}\rangle_{L^2(\C_+)}=\frac{14 \sqrt{2}}{45},\quad
\langle F_3,\tilde{F}\rangle_{L^2(\C_+)}=\frac{2 \sqrt{2}}{35}.
$$
We deduce that
\begin{align*}
\left|2\frac{\langle P_WF_1,F_1\rangle_{L^2(\C_+)}}{\langle F_1,F_1\rangle_{L^2(\C_+)}}
	+\frac{\langle P_WF_2,F_2\rangle_{L^2(\C_+)}}{\langle F_2,F_2\rangle_{L^2(\C_+)}}
	+\frac{\langle P_WF_3,F_3\rangle_{L^2(\C_+)}}{\langle F_3,F_3\rangle_{L^2(\C_+)}}
	-0.2046049976\right|
	&\leq 24\varepsilon.
\end{align*}

This enables us to get a sufficiently precise estimate for the quadratic form. Indeed, we want to show that the norm of the following quadratic form is smaller that $\half$
$$
q(F)=
	2\left| \langle F,X_1\rangle_{L^2(\C_+)}\right|^2
	+\left| \langle F,X_2\rangle_{L^2(\C_+)}\right|^2
	+\left| \langle F,X_3\rangle_{L^2(\C_+)}\right|^2,
	\quad F\in L^2(\C_+).
$$
Applying Cauchy-Schwarz's inequality, for $F\in W$,
\begin{align*}
q(F)
	&=2\left|\frac{ \langle F,F_1\rangle_{L^2(\C_+)}}{\|F_1\|_{L^2(\C_+)}}\right|^2
	+\left|\frac{ \langle F,F_2\rangle_{L^2(\C_+)}}{\|F_2\|_{L^2(\C_+)}}\right|^2
	+\left|\frac{ \langle F,F_3\rangle_{L^2(\C_+)}}{\|F_3\|_{L^2(\C_+)}}\right|^2\\
	&\leq \|F\|_{L^2(\C_+)}^2
	\left(2\frac{\langle P_WF_1,F_1\rangle_{L^2(\C_+)}}{\|F_1\|_{L^2(\C_+)}^2}
	+\frac{\langle P_WF_2,F_2\rangle_{L^2(\C_+)}}{\|F_2\|_{L^2(\C_+)}^2}
	+\frac{\langle P_WF_3,F_3\rangle_{L^2(\C_+)}}{\|F_3\|_{L^2(\C_+)}^2}\right).
\end{align*}
But we just estimated
$$
C:=2\frac{\langle P_WF_1,F_1\rangle_{L^2(\C_+)}}{\|F_1\|_{L^2(\C_+)}^2}
	+\frac{\langle P_WF_2,F_2\rangle_{L^2(\C_+)}}{\|F_2\|_{L^2(\C_+)}^2}
	+\frac{\langle P_WF_3,F_3\rangle_{L^2(\C_+)}}{\|F_3\|_{L^2(\C_+)}^2}
$$
as
$$
C\approx 0.2046049976<\half.
$$
Going back to $h$ in $\dot{H}^1(\heis)\cap V_0^+\cap\Vect_{\R}(Q_+,iQ_+,\partial_s Q_+,i\partial_sQ_+)^{\perp,\dot{H}^1(\heis)},$
\begin{align*}
(\L h,h)_{\dot{H}^{-1}(\heis)\times\dot{H}^1(\heis)}
	&\geq \pi\left(\frac{1}{2}\|F_{hQ_+}\|_{L^2(\C_+)}^2-q(F_{hQ_+})\right)\\
	&\geq \pi\left(\frac{1}{2}\|F_{hQ_+}\|_{L^2(\C_+)}^2-C\|F_{hQ_+}\|_{L^2(\C_+)}^2\right)\\
	&=\frac{1-2C}{2}\|hQ_+\|_{L^2(\heis)}^2\\
	&=\frac{1-2C}{2}\|v_h\|_{L^2(\S^3)}^2.
\end{align*}
But
$$
\half(\L h,h)_{\dot{H}^{-1}(\heis)\times\dot{H}^1(\heis)}
	=\langle\Dlapl v_h,v_h\rangle_{L^2(\S^3)}-\langle v_h,v_h\rangle_{L^2(\S^3)}
$$
so
$$
\langle\Dlapl v_h,v_h\rangle_{L^2(\S^3)}
	\geq (1+\frac{1-2C}{4})\langle v_h,v_h\rangle_{L^2(\S^3)}
$$
and
$$
\langle\Dlapl v_h,v_h\rangle_{L^2(\S^3)}-\langle v_h,v_h\rangle_{L^2(\S^3)}
	\geq (1-\frac{1}{1+(1-2C)/4})\langle\Dlapl v_h,v_h\rangle_{L^2(\S^3)}.
$$
Set $\delta =2(1-\frac{1}{1+(1-2C)/4})$. Since $ \langle\Dlapl v_h,v_h\rangle_{L^2(\S^3)}=\|h\|_{\dot{H}^1(\heis)}^2$, the following theorem holds.

\begin{thm}\label{thm:L_coercive}
The linearized operator $\L$ around $Q_+$
$$
\L h=-\hlapl h-2\Pi_0^+(|Q_+|^2h)-\Pi_0^+(Q_+^2\overline{h})
$$
is coercive outside the finite-dimensional subspace spanned by $Q_+$, $iQ_+$, $\partial_s Q_+$ and $i\partial_sQ_+$~: there exists $\delta>0$ such that for all $h$ in $\dot{H}^1(\heis)\cap V_0^+\cap (Q_+,iQ_+,\partial_s Q_+,i\partial_sQ_+)^{\perp,\dot{H}^1(\heis)},$ then
$$
(\L h,h)_{\dot{H}^{-1}(\heis)\times\dot{H}^1(\heis)}
	\geq \delta \|h\|_{\dot{H}^1(\heis)}^2.
$$
\end{thm}

For the Szeg\H{o} equation, Pocovnicu proved in \cite{Pocovnicu2012} that the linearized operator  is coercive in directions which are symplectically orthogonal to the manifold of solitons
$$
\left\{\frac{\alpha \mu \e^{i\theta}}{\mu(x-a)+i};\; \mu\in\R_+^*, \alpha\in\R_+^*,\theta\in\T,a\in\R\right\}.
$$

The non degeneracy follows from this theorem and the study of $\L$ on the finite-dimensional subspace $V=\dot{H}^1(\heis)\cap V_0^+\cap\Vect_{\R}(Q_+,iQ_+,\partial_sQ_+,i\partial_sQ_+)$ (part \ref{subsection:othogonality}).
\begin{cor}
The linearized operator $\L$ is non degenerate~:
$$
\textnormal{Ker}(\L)=\Vect_{\R}(\partial_s Q_+, i Q_+, Q_++2i\partial_sQ_+).
$$
\end{cor}

\subsection{Invertibility of \texorpdfstring{$\L$}{L}}\label{subsection:L_invertible}

The following corollaries of Theorem \ref{thm:L_coercive} make precise the invertibility of $\L$ and the linear stability up to symmetries of the ground state $Q_+$. These estimates will be useful in order to prove the invertibility of the linearized operators $\L_{Q_\beta}$ around $Q_\beta$ in section \ref{section:uniqueness_schrodinger}.

\begin{cor}\label{cor:estimates_invertibility_L}
There exists $c>0$ such that for all $h\in\dot{H}^1(\heis)\cap V_0^+$,
$$
\|\L h\|_{\dot{H}^{-1}(\heis)}
	+|(h,\partial_sQ_+)_{\dot{H}^1(\heis)}|
	+|(h,iQ_+)_{\dot{H}^1(\heis)}|
	+|(h,Q_++2i\partial_sQ_+)_{\dot{H}^1(\heis)}|
	\geq c\|h\|_{\dot{H}^1(\heis)}.
$$
\end{cor}

\begin{proof}
Let $h\in\dot{H}^1(\heis)\cap V_0^+$. We decompose $h$ into three orthogonal components 
$
h=h_0+h_-+h_+,
$
where $h_0\in \dot{H}^1(\heis)\cap V_0^+\cap\Vect_{\R}(\partial_sQ_+,iQ_+,Q_++2i\partial_sQ_+)$, $h_-\in\dot{H}^1(\heis)\cap V_0^+\cap\Vect_{\R}(Q_+)$ and $h_+\in \dot{H}^1(\heis)\cap V_0^+\cap \Vect_{\R}(Q_+,iQ_+,\partial_s Q_+,i\partial_sQ_+)^{\perp,\dot{H}^1(\heis)}$. Then
$\L h_0=0$, and $\L h_+$ satisfies the above coercivity estimate \ref{thm:L_coercive}~: for some $\delta>0$,
$$
\|\L h_+\|_{\dot{H}^{-1}(\heis)}\geq \delta\|h_+\|_{\dot{H}^1(\heis)}.
$$
Write $h_-=\lambda Q_+$ for some real number $\lambda$. Then $\L h_-=2\lambda i\partial_sQ_+$, so
$$
(\L h_-,h_-)_{\dot{H}^{-1}(\heis)\times\dot{H}^{1}(\heis)}
	=2\lambda^2(i\partial_sQ_+,Q_+)_{\dot{H}^{-1}(\heis)\times\dot{H}^{1}(\heis)}.
$$
But
$$
\|h_-\|_{\dot{H}^{1}(\heis)}^2
	=(-i\lambda \partial_sQ_+,\lambda Q_+)_{\dot{H}^{-1}(\heis)\times\dot{H}^{1}(\heis)},
$$
so
$(\L h_-,h)_{\dot{H}^{-1}(\heis)\times\dot{H}^{1}(\heis)}=-2\|h_-\|_{\dot{H}^{1}(\heis)}^2$. In particular, $\|\L h_-\|_{\dot{H}^{-1}(\heis)}\geq 2\|h_-\|_{\dot{H}^1(\heis)}$.

Thanks to Corollary \ref{cor:decomposition_h+h-}, we deduce that
\begin{align*}
\|\L h\|_{\dot{H}^{-1}(\heis)}^2
	&=\|\L h_-\|_{\dot{H}^{-1}(\heis)}^2+\|\L h_+\|_{H^{-1}(\heis)}^2\\
	&\geq  4\|h_-\|_{\dot{H}^1(\heis)}^2+\delta^2\| h_+\|_{H^1(\heis)}^2\\
	&\geq (\min(2,\delta))^2\|h_-+h_+\|_{\dot{H}^{1}(\heis)}^2.
\end{align*}
Moreover, since $h_0$ is in the space spanned by $\partial_sQ_+$, $iQ_+$ and $Q_++2iQ_+$, there exists some constant $0<c\leq \min(2,\delta)$ such that
$$
|(h,\partial_sQ_+)_{\dot{H}^1(\heis)}|
	+|(h,iQ_+)_{\dot{H}^1(\heis)}|
	+|(h,Q_++2i\partial_sQ_+)_{\dot{H}^1(\heis)}|
	\geq c\|h_0\|_{ \dot{H}^{1}(\heis)}.
$$
Therefore,
$$
\|\L h\|_{H^{-1}(\heis)}
	+|(h,\partial_sQ_+)_{\dot{H}^1(\heis)}|
	+|(h,iQ_+)_{\dot{H}^1(\heis)}|
	+|(h,Q_++2i\partial_sQ_+)_{\dot{H}^1(\heis)}|
	\geq c\|h\|_{\dot{H}^1(\heis)}.
$$
\end{proof}

Let us remind that for $h\in \dot{H}^1(\heis)\cap V_0^+$, we have set in Definition \ref{def:delta_u}
$$
\delta(u)=\left|\|u\|_{\dot{H}^{1}(\heis)}^2-\|Q_+\|_{\dot{H}^{1}(\heis)}^2\right|
	+\left|\|u\|_{L^4(\heis)}^4-\|Q_+\|_{L^4(\heis)}^4\right|.
$$

\begin{cor}\label{cor:estimates_delta}
There exists $\varepsilon_0>0$ and $c>0$ such that for all $u\in \dot{H}^1(\heis)\cap V_0^+$, if $\|u-Q_+\|_{\dot{H}^1(\heis)}\leq \varepsilon_0$, then
$$
\delta(u)
	+|(u,\partial_sQ_+)_{\dot{H}^1(\heis)}|
	+|(u,iQ_+)_{\dot{H}^1(\heis)}|
	+|(u,Q_++2i\partial_sQ_+)_{\dot{H}^1(\heis)}|
	\geq c\|u-Q_+\|_{\dot{H}^1(\heis)}^2.
$$
\end{cor}

\begin{proof}
Let $u\in \dot{H}^1(\heis)\cap V_0^+$ and set $h=u-Q_+$. We decompose $h$ as above in three orthogonal parts
$
h=h_0+h_-+h_+,
$
where $h_0\in \dot{H}^1(\heis)\cap V_0^+\cap\Vect_{\R}(\partial_sQ_+,iQ_+,Q_++2i\partial_sQ_+)$, $h_-\in\dot{H}^1(\heis)\cap V_0^+\cap\Vect_{\R}(Q_+)$ and $h_+\in \dot{H}^1(\heis)\cap V_0^+\cap (Q_+,iQ_+,\partial_s Q_+,i\partial_sQ_+)^{\perp,\dot{H}^1(\heis)}$.

The link between $\delta(u)$ and the linearized operator $\L$ appears through the functional
$$
E(u):=\|u\|_{\dot{H}^{1}(\heis)}^2-\half\|u\|_{L^4(\heis)}^4.
$$
Indeed,
$$
|E(u)-E(Q_+)|
	\leq\delta(u),
$$
but since $Q_+$ is a solution to $D_sQ_+=\Pi_0^+(|Q_+|^2Q_+)$ and $h$ belongs to $\dot{H}^1(\heis)\cap V_0^+$, we have the Taylor expansion
$$
E(u)-E(Q_+)
	=(\L h,h)_{\dot{H}^{-1}(\heis)\times\dot{H}^1(\heis)}+\gdO(\|h\|_{\dot{H}^1(\heis)}^3).
$$


Therefore,
$$
\delta(u)
	\geq (\L h,h)_{\dot{H}^{-1}(\heis)\times\dot{H}^1(\heis)}-\gdO(\|h\|_{\dot{H}^1(\heis)}^3).
$$
From Corollary \ref{cor:decomposition_h+h-}, we know that
$$
(\L h,h)_{\dot{H}^{-1}(\heis)\times\dot{H}^1(\heis)}
	=(\L h_+,h_+)_{\dot{H}^{-1}(\heis)\times\dot{H}^1(\heis)}
	+(\L h_-,h_-)_{\dot{H}^{-1}(\heis)\times\dot{H}^1(\heis)}.
$$
Consequently, the coercivity estimate on $\L$ implies that for some constants $c_1, C_1>0$,
\begin{equation}\label{eq:ineg1}
\delta(u)
	\geq c_1 \|h_+\|_{\dot{H}^1(\heis)}^2-C_1(\|h_-\|_{\dot{H}^1(\heis)}^2+\|h\|_{\dot{H}^1(\heis)}^3).
\end{equation}

Let us focus on the term $\|h_-\|_{\dot{H}^1(\heis)}^2$. We use the fact that
\begin{align*}
\delta(u)
	&\geq \big|(Q_++h,Q_++h)_{\dot{H}^1(\heis)}-(Q_+,Q_+)_{\dot{H}^1(\heis)}\big|\\
	&\geq 2\big|(Q_+,h)_{\dot{H}^1(\heis)}\big|-\|h\|_{\dot{H}^1(\heis)}^2\\
	&= 2\|Q_+\|_{\dot{H}^1(\heis)}\|h_-\|_{\dot{H}^1(\heis)}-\|h\|_{\dot{H}^1(\heis)}^2,
\end{align*}
so
$$
\delta(u)^2\geq 4\|Q_+\|_{\dot{H}^1(\heis)}^2\|h_-\|_{\dot{H}^1(\heis)}^2-\gdO(\|h\|_{\dot{H}^1(\heis)}^3).
$$
We use this estimate to control $\|h_-\|_{\dot{H}^1(\heis)}^2$ in the lower bound \eqref{eq:ineg1} of $\delta(u)$. Up to decreasing $\varepsilon_0$, one can absorb the term $\delta(u)^2$ into the term $\delta(u)$ : there exist $c_2,C_2>0$ and $\varepsilon_0>0$ such that if $\|h\|_{\dot{H}^1(\heis)}=\|u-Q_+\|_{\dot{H}^1(\heis)}\leq \varepsilon_0$,
$$
2\delta(u)
	\geq\delta(u)+C_2\delta(u)^2
	\geq c_2 \|h_+\|_{\dot{H}^1(\heis)}^2+c_2\|h_-\|_{\dot{H}^1(\heis)}^2-C_2\|h\|_{\dot{H}^1(\heis)}^3.
$$

We now control $\|h_0\|_{\dot{H}^1(\heis)}^2$. If $\varepsilon_0\leq 1$, we have an upper bound
$$
\|h_0\|_{\dot{H}^1(\heis)}^2\leq \|h_0\|_{\dot{H}^1(\heis)}\leq C(|(h,\partial_sQ_+)_{\dot{H}^1(\heis)}|
	+|(h,iQ_+)_{\dot{H}^1(\heis)}|
	+|(h,Q_++2i\partial_sQ_+)_{\dot{H}^1(\heis)}|).
$$

In the end, there exist $c_3>0$ and $C_3>0$ such that for all $u\in\dot{H}^1(\heis)\cap V_0^+$,
\begin{align*}
\delta(u)
	+|(h,\partial_sQ_+)_{\dot{H}^1(\heis)}|
	&+|(h,iQ_+)_{\dot{H}^1(\heis)}|
	+|(h,Q_++2i\partial_sQ_+)_{\dot{H}^1(\heis)}|\\
	&\geq c_3 ( \|h_+\|_{\dot{H}^1(\heis)}^2+\|h_-\|_{\dot{H}^1(\heis)}^2+\|h_0\|_{\dot{H}^1(\heis)}^2)
	-C_3\|h\|_{\dot{H}^1(\heis)}^3\\
	&= c_3 \|h\|_{\dot{H}^1(\heis)}^2	-C_3\|h\|_{\dot{H}^1(\heis)}^3.
\end{align*}
Up to decreasing $\varepsilon_0$ again, we can absorb the term $\|h\|_{\dot{H}^1(\heis)}^3$ into the term $\|h\|_{\dot{H}^1(\heis)}^2$. Note that $Q_+$ is orthogonal in $\dot{H}^1(\heis)$ to $\partial_sQ_+,iQ_+$ and $Q_++2i\partial_sQ_+$, therefore 
$(h,\partial_sQ_+)_{\dot{H}^1(\heis)}=(u,\partial_sQ_+)_{\dot{H}^1(\heis)}$,
$(h,iQ_+)_{\dot{H}^1(\heis)}=(u,iQ_+)_{\dot{H}^1(\heis)}$
and
$(h,Q_++2i\partial_sQ_+)_{\dot{H}^1(\heis)}=(u,Q_++2i\partial_sQ_+)_{\dot{H}^1(\heis)}$.
\end{proof}

We now control the distance of a function $u\in\dot{H}^1(\heis)\cap V_0^+$ to the profile $Q_+$ up to symmetries by the difference of their norms $\delta(u)$.
\begin{mydef}
Fix $h\in\dot{H}^1(\heis)\cap V_0^+$, $s_0\in\R$, $\theta\in\T$ and $\alpha\in\R_+^*$. We denote by $T_{s_0,\theta,\alpha}h$ the function in $\dot{H}^1(\heis)\cap V_0^+$ defined by
$$
T_{s_0,\theta,\alpha}h(x,y,s):=\e^{i\theta}\alpha h\big(\alpha x,\alpha y,\alpha^2(s+s_0)\big), \quad (x,y,s)\in\heis.
$$
\end{mydef}

\begin{cor}\label{cor:estimates_stability}
There exist $\delta_0>0$ and $C>0$ such that for all $u\in\dot{H}^1(\heis)\cap V_0^+$, if $\delta(u)\leq \delta_0$, then
$$
\inf_{(s_0,\theta,\alpha)\in \R\times\T\times\R_+^*}\|T_{s_0,\theta,\alpha}u-Q_+\|_{\dot{H}^1(\heis)}^2
	\leq C\delta(u).
$$
\end{cor}

\begin{proof}
Assume by contradiction that there exists a sequence $(u_n)_{n\in\N}\subset \dot{H}^1(\heis)\cap V_0^+$ such that $\delta(u_n)\to 0$, but
$$
\frac{1}{\delta(u_n)}\inf_{(s_0,\theta,\alpha)\in \R\times\T\times\R_+^*}\|T_{s_0,\theta,\alpha}u_n-Q_+\|_{\dot{H}^1(\heis)}^2
	\longrightarroww{n\to+\infty}{} +\infty.
$$
According to the consequence of the profile decomposition theorem stated in Proposition \ref{prop:stability1}, since $\delta(u_n)\to 0$, then, up to a subsequence, there exist cores $(s_n)_{n\in\N}\in\R^{\N}$, an angle $\theta_0\in\T$, and scalings $(\alpha_n)_{n\in\N}\in (\R_+^*)^{\N}$ such that
$$
\|T_{s_n,\theta_0,\alpha_n}u_n-Q_+\|_{\dot{H}^{1}(\heis)}\longrightarroww{n\to+\infty}{} 0.
$$

We will make use of the implicit function theorem in order to apply Corollary \ref{cor:estimates_delta} with some functions $T_{s_n,\theta_n,\alpha_n}u_n$ orthogonal to $\partial_sQ_+,iQ_+$ and $Q_++2i\partial_sQ_+$ and get a contradiction. Consider the maps
\begin{align*}
F:\dot{H}^1(\heis)\cap V_0^+ &\rightarrow \R^3\\
	 u & \mapsto \big((u,\partial_sQ_+)_{\dot{H}^1(\heis)},(u,iQ_+)_{\dot{H}^1(\heis)},(u,Q_++2i\partial_sQ_+)_{\dot{H}^1(\heis)}\big),
\end{align*}
and
\begin{align*}
G: \R\times\T \times \R_+^* \times (\dot{H}^1(\heis)\cap V_0^+)&\rightarrow \R^3\\
	 (s,\theta,\alpha,u) & \mapsto F(T_{s,\theta,\alpha}u).
\end{align*}
Then $F(Q_+)=0$ so $G(0,0,1,Q_+)=0$. Moreover, $G$ is smooth in $(s,\theta,\alpha)$ and the Jacobian $\d_{s,\theta,\alpha} G(0,0,1,Q_+)$ of this application along $(s,\theta,\alpha)$ at $(s,\theta,\alpha,u)=(0,0,1,Q_+)$ is equal to
$$
\left(\begin{matrix}
\|\partial_sQ_+\|_{\dot{H}^1(\heis)}^2 & (iQ_+,\partial_sQ_+)_{\dot{H}^1(\heis)} & (Q_++2i\partial_sQ_+,\partial_sQ_+)_{\dot{H}^1(\heis)}\\
(\partial_sQ_+,iQ_+)_{\dot{H}^1(\heis)} & \|iQ_+\|_{\dot{H}^1(\heis)}^2 & (Q_++2i\partial_sQ_+,iQ_+)_{\dot{H}^1(\heis)}\\
(\partial_sQ_+,Q_++2i\partial_sQ_+)_{\dot{H}^1(\heis)} & (iQ_+,Q_++2i\partial_sQ_+)_{\dot{H}^1(\heis)} & \|Q_++2i\partial_sQ_+\|_{\dot{H}^1(\heis)}^2
\end{matrix}\right).
$$
Replacing all the terms by their values, we get
$$
\d_{s,\theta,\alpha} G(0,0,1,Q_+)=
\left(\begin{matrix}
\frac{\pi^2}{2} & \frac{\pi^2}{2} & 0\\
0 & \frac{\pi^2}{2} & 0\\
0 & 0 & \pi^2
\end{matrix}\right),
$$
which is invertible. By the implicit function theorem, we get continuously differentiable functions $S_0(u)$, $\Theta(u)$ and $A(u)$, defined in a neighbourhood $\mathcal{V}$ of $Q_+$ and valued in a neighbourhood of $(0,0,1)$ : if $u\in \mathcal{V}$, then $\|T_{S_0(u),\Theta(u),A(u)}u-Q_+\|_{\dot{H}^1(\heis)}\leq \varepsilon_0$ (where $\varepsilon_0$ is taken from Corollary \ref{cor:estimates_delta}). These functions satisfy $(S_0(Q_+),\Theta(Q_+),A(Q_+))=(0,0,1)$ and
$$
G(S_0(u),\Theta(u),A(u),u)=0.
$$
Now, since $\|T_{s_n,\theta_0,\alpha_n}u_n-Q_+\|_{\dot{H}^{1}(\heis)}\longrightarroww{n\to+\infty}{} 0$, there exists $N\in\N$ such that for all $n\geq N$, $T_{s_n,\theta_0,\alpha_n}u_n\in\mathcal{V}$.
Therefore, defining $s_n'=s_n+S_0(T_{s_n,\theta_0,\alpha_n}u_n)$, $\theta_n'=\theta_0+\Theta(T_{s_n,\theta_0,\alpha_n}u_n)$ and $\alpha_n'=\alpha_n+A(T_{s_n,\theta_0,\alpha_n}u_n)$, we get $\tilde{u_n}:=T_{s_n',\theta_n',\alpha_n'}u_n\in\dot{H}^1(\heis)\cap V_0^+$ such that $\|\tilde{u_n}-Q_+\|_{\dot{H}^1(\heis)}\leq \varepsilon_0$ and
$$
(\tilde{u_n},\partial_sQ_+)_{\dot{H}^1(\heis)}=(\tilde{u_n},iQ_+)_{\dot{H}^1(\heis)}=(\tilde{u_n},Q_++2i\partial_sQ_+)_{\dot{H}^1(\heis)}=0.
$$ Moreover, by invariance under symmetries,
$$
\delta(\tilde{u_n})=\delta(u_n),
$$
so applying Corollary \ref{cor:estimates_delta} to $\tilde{u_n}=T_{s_n',\theta_n',\alpha_n'}u_n$, we get that for some constant $C>0$,
$$
\|T_{s_n',\theta_n',\alpha_n'}u_n-Q_+\|_{\dot{H}^1(\heis)}^2\leq C \delta(u_n).
$$
This is a contradiction with the assumption that
$$
\frac{1}{\delta(u_n)}\inf_{(s_0,\theta,\alpha)\in \R\times\T\times\R_+^*}\|T_{s_0,\theta,\alpha}u_n-Q_+\|_{\dot{H}^1(\heis)}^2
	\longrightarroww{n\to+\infty}{} +\infty.
$$
\end{proof}

\section{Uniqueness of traveling waves for the Schrödinger equation}\label{section:uniqueness_schrodinger}

In this section, we show that the study of the limiting profile $Q_+$, and in particular the linear stability, enables us to prove some uniqueness results about the sequence of traveling waves $Q_\beta$ with speed $\beta$ sufficiently close to $1$. The argument is similar as in \cite{GerardLenzmannPocovnicuRaphael2018} for the half-wave equation : for $\beta$ close to $1$, $Q_\beta$ is close to $Q_+$ so we can make a link between the respective linearized operators. 

In order to do so, we first need to show some regularity properties and decay estimates on the profiles $Q_\beta$ (part \ref{subsection:regularity_decay}). For the half-wave equation, these estimates came from the Sobolev embedding $H^{\half}(\R)\hookrightarrow L^p(\R)$, $2 \leq p<+\infty$ and the convergence in $H^{\half}(\R)$. 

Recall that from Definition \ref{def:Qbeta}, $\Qbeta$ denotes the set of ground states $Q_\beta$ satisfying \eqref{eq:Hbeta}
\begin{equation*}
-\frac{\hlapl +\beta D_s}{1-\beta}Q_\beta=|Q_\beta|^2Q_\beta.
\end{equation*}

One can summarize the convergence of $(Q_\beta)_\beta$ from part \ref{subsection:limit} combined with the uniqueness result for $Q_+$ from section \ref{subsection:ground_state_solutions_limit} as follows.

\begin{prop}\label{prop:convergence_Qbeta}
For all $\beta\in(-1,1)$, fix a ground state $Q_\beta^0\in\Qbeta$ of speed $\beta$. Then there exist scalings $(\alpha_\beta)_\beta$ in $\R_+^*$, cores $(s_\beta)_\beta$ in $\R$, and an angle $\theta$ in $\T$ such that after a change of functions $Q_\beta:= \e^{i\theta} \alpha_\beta Q_\beta(\alpha_\beta\cdot,\alpha_\beta\cdot,\alpha_\beta^2(\cdot+s_\beta))$, the sequence $(Q_\beta )_\beta$ of solutions to \eqref{eq:Hbeta}
$$
-\frac{\hlapl +\beta D_s}{1-\beta}Q_\beta=|Q_\beta|^2Q_\beta
$$
converges as $\beta \to 1$ in $\dot{H}^1(\heis)$ to the unique (up to symmetries) ground state solution to \eqref{eq:Hplus}
$$
D_sQ_+=\Pi_0^+(|Q_+|^2Q_+),
$$
which writes
$$
Q_+(x,y,s)=\frac{i\sqrt{2}}{s+i(x^2+y^2)+i}.
$$
\end{prop}

\subsection{Regularity and decay of the traveling waves \texorpdfstring{$Q_{\beta}$}{Q_beta}}\label{subsection:regularity_decay}

In this part, we collect information on the regularity of the profiles $Q_\beta$. We show that after the transformations from Proposition \ref{prop:convergence_Qbeta}, they are uniformly bounded in $L^p(\heis)$ for all $p>2$ when $\beta$ is close to $1$. We deduce an uniform bound in $L^{\infty}(\heis)$, from which we estimate the decay of these profiles when the variable $(x,y,s)\in\heis$ tends to infinity. Finally, we show that $(Q_\beta)_{\beta}$ is bounded in $\dot{H}^k(\heis)$ for $\beta$ close to $1$ and fixed $k\geq 1$.

The operator $-\frac{\hlapl +\beta D_s}{1-\beta}$ admits an explicit fundamental solution \cite{stein_harmonic}.

\begin{thm}\label{thm:fundamental_sol}
Let
$$
m_\beta(x,y,s)=-\frac{1-\beta}{2\pi^2}\Gamma\left(\frac{1-\beta}{2}\right)\Gamma\left(\frac{1+\beta}{2}\right)
	\frac{1}{(x^2+y^2-is)^{\frac{1-\beta}{2}}(x^2+y^2+is)^{\frac{1+\beta}{2}}}.
$$
Then $m_\beta$ is a fundamental solution for $-\frac{\hlapl +\beta D_s}{1-\beta}$ : in the sense of distributions,
$$
-\frac{\hlapl +\beta D_s}{1-\beta}m_\beta=\delta_0.
$$
\end{thm}

The proof of regularity for the $Q_\beta$ relies on the use of generalized Hölder's and Young's inequalities in weak Lebesgue spaces (see \cite{Vetois2016} for the strategy). We define the Lorentz spaces as follows.

\begin{mydef}[Lorentz spaces]
Fix $p\in[1,\infty)$ and $q\in[1,\infty]$. The Lorentz space $L^{p,q}(\heis)$ is the set of all functions $f:\heis\to\C$ with finite $L^{p,q}(\heis)$ norm, where
\begin{equation*}
\|f\|_{L^{p,q}}(\heis)
	:=
\begin{cases}
\left(p\int_0^{+\infty}R^{q-1}\lambda_3(\{u\in\heis; |f(u)|\geq R\})^{\frac{q}{p}}\d R\right)^{\frac{1}{q}} &\textnormal{ if } q<\infty
\\
\sup_{R>0}\left(R^p\lambda_3(\{u\in\heis; |f(u)|\geq R\})\right) &\textnormal{ if } q=\infty
\end{cases}.
\end{equation*}
\end{mydef}

The usual $L^p(\heis)$ spaces coincide with the $L^{p,p}(\heis)$ spaces. In general, $\|\cdot\|_{L^{p,q}(\heis)}$ is not a norm since the Minkowski inequality may fail. The following inclusion relations are true \cite{Stein1971}.

\begin{prop}[Growth of $L^{p,q}$ spaces]
Let $p\in[1,\infty)$ and $q_1,q_2\in[1,\infty]$ such that $q_1\leq q_2$. Then $L^{p,q_1}(\heis)\subset L^{p,q_2}(\heis)$.
\end{prop}

Note that the functions $m_\beta$, $\beta\in[0,1)$, are uniformly bounded in $L^{2,\infty}$. Indeed, let $R>0$, then
$$
\lambda_3(\{(x,y,s)\in\heis;\; |x|^2+|y|^2+|s|\leq R\})
	=R^2\lambda_3(\{(x',y',s')\in\heis;\; |x'|^2+|y'|^2+|s'|\leq 1\}),
$$
moreover, the constants
$$
c_\beta:=-\frac{1-\beta}{2\pi^2}\Gamma\left(\frac{1-\beta}{2}\right)\Gamma\left(\frac{1+\beta}{2}\right)
$$
are bounded for $\beta\in[0,1).$

\begin{mydef}[Convolution] The convolution product of two functions $f$ and $g$ on $\heis$ is defined by
$$
f\star g(u)=\int_{\heis}f(v)g(v^{-1}u)\d\lambda_3(v)=\int_{\heis}f(uv^{-1})g(v)\d\lambda_3(v).
$$
\end{mydef}

Note that the convolution in $\heis$ is not commutative, and that the relation
$$
P(f\star g)=f\star Pg
$$
holds for every left-invariant vector field $P$ in $\heis$ (for example, $P=-\frac{\hlapl +\beta D_s}{1-\beta}$), whereas in general $P(f\star g)\neq Pf\star g$.

Let us recall the generalizations of Hölder's and Young inequalities for Lorentz spaces.

\begin{lem}[Hölder]\label{holder}
Let $p_1,p_2,p\in(0,\infty)$ and $q_1,q_2,q\in(0,\infty]$ such that
$$
\frac{1}{p_1}+\frac{1}{p_2}=\frac{1}{p} \text{ and } \frac{1}{q_1}+\frac{1}{q_2}\geq\frac{1}{q}
$$
with the convention $1/\infty=0$. Then there exists $C=C(p_1,p_2,p,q_1,q_2,q)$ such that for any $f\in L^{p_1,q_1}(\heis)$ and any $g\in L^{p_2,q_2}(\heis)$, we have $fg\in L^{p,q}(\heis)$ and
$$
\|fg\|_{L^{p,q}(\heis)}
	\leq C \|f\|_{L^{p_1,q_1}(\heis)}\|g\|_{L^{p_2,q_2}(\heis)}.
$$
\end{lem}

\begin{lem}[Young]\label{young}
Let $p_1,p_2,p\in(1,\infty)$ and $q_1,q_2,q\in(0,\infty]$ such that
$$
\frac{1}{p_1}+\frac{1}{p_2}=\frac{1}{p}+1 \text{ and } \frac{1}{q_1}+\frac{1}{q_2}\geq\frac{1}{q}
$$
with the convention $1/\infty=0$. Then there exists $C=C(p_1,p_2,p,q_1,q_2,q)$ such that for any $f\in L^{p_1,q_1}(\heis)$ and any $g\in L^{p_2,q_2}(\heis)$, we have $f\star g\in L^{p,q}(\heis)$ and
$$
\|f\star g\|_{L^{p,q}(\heis)}
	\leq C \|f\|_{L^{p_1,q_1}(\heis)}\|g\|_{L^{p_2,q_2}(\heis)}.
$$
\end{lem}

Theorem \ref{thm:fundamental_sol} implies the following formula for $Q_\beta$.

\begin{cor}
For all $\beta\in(-1,1)$,
$$
Q_\beta=(|Q_\beta|^2Q_\beta)\star m_\beta.
$$
\end{cor}

Let us now prove the boundedness of $Q_\beta$ in $L^p(\heis)$, $p>2$.
\begin{thm}\label{thm:bound_Lp}
For all $p>2$, there exist $C_p>0$ and $\beta_*(p)\in(0,1)$ such that for all $\beta\in(\beta_*(p),1)$, $\|Q_\beta\|_{L^p(\heis)}\leq C_p$.
\end{thm}

\begin{proof}
We proceed by contradiction. Fix $p>2$. Assume that there exists a sequence $(\beta_n)_{n\in\N}$ in $(0,1)$ converging to $1$ and such that $\|Q_{\beta_n}\|_{L^p(\heis)}\in [n,+\infty]$ for all $n\in\N$. By duality and density of $\classeC_c^{\infty}(\heis)$ in $L^q(\heis)$, $\frac{1}{p}+\frac{1}{q}=1$, there exists a sequence $(\varphi_n)_{n\in\N}$ in $ L^q(\heis)\cap L^{\frac{4}{3}}(\heis)$ such that $\|\varphi_n\|_{L^q(\heis)}\leq 1$ for all $n$ and
$$
\left|\int_{\heis}Q_{\beta_n}\varphi_n\d\lambda_3\right|
	\longrightarroww{n\to+\infty}{} +\infty.
$$
Let us define
$$
K_n:=\big\{\varphi\in L^q(\heis)\cap L^{\frac{4}{3}}(\heis);\; \|\varphi\|_{L^q(\heis)}\leq \|\varphi_n\|_{L^q(\heis)}\textnormal{ and }\|\varphi\|_{L^{\frac{4}{3}}(\heis)}\leq \|\varphi_n\|_{L^{\frac{4}{3}}(\heis)}\big\}.
$$
Since $Q_{\beta_n}\in L^4(\heis)$, the supremum over functions $\varphi\in K_n$ of $\int_{\heis} Q_{\beta_n}\varphi\d\lambda_3$ is finite. Thus,  if we change $\varphi_n$ to an other function $\varphi$ from $K_n$ where $\int_{\heis}Q_{\beta_n}\varphi\d\lambda_3$ is closer to this supremum, the $K_n$ corresponding to $\varphi$ and thus the new supremum will decrease. We can therefore assume up to changing $\varphi_n$ that
$$
2\left|\int_{\heis} Q_{\beta_n}\varphi_n\d\lambda_3\right|
	\geq \sup_{\varphi\in K_n} \left|\int_{\heis} Q_{\beta_n}\varphi\d\lambda_3\right|.
$$

By density, let $(f_k)_{k\in\N}$ be a sequence in $\classeC_c^{\infty}(\heis)$ such that $\||Q_+|^2-f_k\|_{L^2(\heis)}\longrightarroww{k\to+\infty}{} 0$. Denote, for $k,n\in\N$, $g_{n,k}:=|Q_{\beta_n}|^2-f_k$. We will use the fact that the functions $g_{n,k}$ have a small norm in $L^2(\heis)$ when $k$ and $n$ are large enough thanks to Proposition \ref{prop:convergence_Qbeta}.
Let us cut
\begin{align*}
\int_{\heis}Q_{\beta_n}\varphi_n\d\lambda_3
	&=\int_{\heis}\big((|Q_{\beta_n}|^2Q_{\beta_n})\star m_{\beta_n}\big)\varphi_n\d\lambda_3\\
	&=\int_{\heis}\big((f_kQ_{\beta_n})\star m_{\beta_n}\big)\varphi_n\d\lambda_3
	 +\int_{\heis}\big((g_{n,k}Q_{\beta_n})\star m_{\beta_n}\big)\varphi_n\d\lambda_3
\end{align*}
in order to evaluate these terms separately.

Concerning the first term in the right hand side, using Lemmas \ref{holder} and \ref{young},
\begin{align*}
\left|
\int_{\heis}\big((f_kQ_{\beta_n})\star m_{\beta_n}\big)\varphi_n\d\lambda_3
\right|
	&\leq \|\big((f_kQ_{\beta_n})\star m_{\beta_n}\big)\varphi_n\|_{L^{1,1}(\heis)}\\
	&\leq C_1(p) \|(f_kQ_{\beta_n})\star m_{\beta_n}\|_{L^{p,p}(\heis)}\|\varphi_n\|_{L^{q,q}(\heis)}\\
	&\leq C_2(p) \|f_kQ_{\beta_n}\|_{L^{\frac{2p}{2+p},p}(\heis)}\|m_{\beta_n}\|_{L^{2,\infty}(\heis)}\|\varphi_n\|_{L^q(\heis)}
\end{align*}
(we used that $\frac{2p}{2+p}>1$ since $p>2$).
Using again Lemma \ref{holder}, choosing any $\tau\in(0,+\infty)$ such that $\frac{1}{\tau}\geq\frac{4-p}{4p}$ and $\sigma=\frac{4p}{4+p}>1$, we get
$$
\left|
\int_{\heis}\big((f_kQ_{\beta_n})\star m_{\beta_n}\big)\varphi_n\d\lambda_3
\right|
	\leq C_3(p) \|f_k\|_{L^{\sigma,\tau}(\heis)}\|Q_{\beta_n}\|_{L^{4,4}(\heis)}\|m_{\beta_n}\|_{L^{2,\infty}(\heis)}\|\varphi_n\|_{L^q(\heis)}.
$$
We know that $\|\varphi_n\|_{L^q(\heis)}\leq 1$ for all $n$, that $\|m_{\beta_n}\|_{L^{2,\infty}(\heis)}$ is bounded independently of $n$ and that $(Q_{\beta})_{\beta\in[0,1)}$ is bounded in $L^4(\heis)$, so there exists $C_4(p)>0$ such that for all $k,n\in\N$,
$$
\int_{\heis}\big((f_kQ_{\beta_n})\star m_{\beta_n}\big)\varphi_n\d\lambda_3
	\leq C_4(p)\|f_k\|_{L^{\sigma,\tau}(\heis)}.
$$

Applying Fubini's theorem to the second term in the right hand side,
\begin{align*}
\int_{\heis}\big((g_{n,k}Q_{\beta_n})\star m_{\beta_n}\big)\varphi_n\d\lambda_3
	&=\int_{\heis}\int_{\heis}(g_{n,k}Q_{\beta_n})(v) m_{\beta_n}(v^{-1}u)\varphi_n(u) \d\lambda_3(v)\d\lambda_3(u)\\
	&=\int_{\heis}\int_{\heis}(g_{n,k}Q_{\beta_n})(v) m_{\beta_n}(v^{-1}u)\varphi_n(u) \d\lambda_3(u)\d\lambda_3(v)\\
	&=\int_{\heis}\int_{\heis}(g_{n,k}Q_{\beta_n})(v) \overset{\vee}{m}_{\beta_n}(u^{-1}v)\varphi_n(u) \d\lambda_3(u)\d\lambda_3(v)\\
	&=\int_{\heis}(g_{n,k}Q_{\beta_n})(v) (\varphi_n\star\overset{\vee}{m}_{\beta_n})(v)\d\lambda_3(v),
\end{align*}
where
\begin{align*}
\overset{\vee}{m}_{\beta}(x,y,s)
	&=m_{\beta}((x,y,s)^{-1})\\
	&=-\frac{1-\beta}{2\pi^2}\Gamma\left(\frac{1-\beta}{2}\right)\Gamma\left(\frac{1+\beta}{2}\right)
	\frac{1}{(x^2+y^2+is)^{\frac{1-\beta}{2}}(x^2+y^2-is)^{\frac{1+\beta}{2}}}
\end{align*}
has the same bounds in $L^{2,\infty}(\heis)$ as $m_{\beta}.$

But thanks to Lemmas \ref{holder} and \ref{young},
\begin{align*}
\|g_{n,k}(\varphi_n\star\overset{\vee}{m}_{\beta_n})\|_{L^q(\heis)}
	&\leq C'_1(p) \|g_{n,k}\|_{L^{2,\infty}(\heis)}\|\varphi_n\star\overset{\vee}{m}_{\beta_n}\|_{L^{\frac{2p}{p-2},q}(\heis)}\\
	&\leq C'_2(p) \|g_{n,k}\|_{L^{2,\infty}(\heis)}\|\varphi_n\|_{L^{q,q}(\heis)}\|\overset{\vee}{m}_{\beta_n}\|_{L^{2,\infty}(\heis)}.
\end{align*}
Note that the assumption $p>2$ ensures that $\frac{2p}{p-2}\in(1,\infty).$

Moreover, this last inequality still holds with the same reasoning when replacing $p$ by $4$ and its conjugate exponent $q$ by $\frac{4}{3}.$ Fix
$$
C=\max\big(C'_2(p),C'_2(4)\big)\times\sup_{\beta\in[0,1)}\|\overset{\vee}{m}_{\beta}\|_{L^{2,\infty}(\heis)}.
$$
Then, when $g_{n,k}$ is non-zero in $L^2(\heis)$, the function
$$
\psi_{n,k}:=\frac{1}{C\|g_{n,k}\|_{L^{2,\infty}(\heis)}}g_{n,k}(\varphi_n\star\overset{\vee}{m}_{\beta_n})
$$
belongs to $K_n$.
Therefore by definition of $\varphi_n$, for all $k,n\in\N$,
$$
\left|
\int_{\heis}Q_{\beta_n} g_{n,k}(\varphi_n\star\overset{\vee}{m}_{\beta_n})\d\lambda_3
\right|
	\leq 2C\|g_{n,k}\|_{L^{2,\infty}(\heis)}\left|\int_{\heis}Q_{\beta_n}\varphi_n\d\lambda_3\right|.
$$

But
\begin{align*}
\|g_{n,k}\|_{L^{2,\infty}(\heis)}
	&\leq \||Q_{\beta_n}|^2-f_k\|_{L^{2}(\heis)}\\
	&\leq\||Q_{\beta_n}|^2-|Q_+|^2\|_{L^{2}(\heis)}+\||Q_+|^2-f_k\|_{L^{2}(\heis)},
\end{align*}
and this quantity converges to $0$ as $\min(n,k)$ goes to $+\infty$ thanks to Proposition \ref{prop:convergence_Qbeta} and the construction of $(f_k)_{k\in\N}$. Therefore, there exists $n_0$ such that, for all $k\geq n_0$ and $n\geq n_0$, $2C\|g_{n,k}\|_{L^{2,\infty}(\heis)}\leq \half$, or in other words,
$$
\left|
\int_{\heis}Q_{\beta_n} g_{n,k}(\varphi_n\star\overset{\vee}{m}_{\beta_n})\d\lambda_3
\right|
	\leq \half\left|\int_{\heis}Q_{\beta_n}\varphi_n\d\lambda_3\right|.
$$ Since
$$
\int_{\heis}Q_{\beta_n}\varphi_n\d\lambda_3=
	\int_{\heis}\big((f_kQ_{\beta_n})\star m_{\beta_n}\big)\varphi_n\d\lambda_3
	 +\int_{\heis}Q_{\beta_n} g_{n,k}(\varphi_n\star\overset{\vee}{m}_{\beta_n})\d\lambda_3,
$$
we get that for all $k\geq n_0$ and $n\geq n_0$,
$$
\left|\int_{\heis}Q_{\beta_n}\varphi_n\d\lambda_3\right|
	\leq 2\left|\int_{\heis}\big((f_kQ_{\beta_n})\star m_{\beta_n}\big)\varphi_n\d\lambda_3\right|.
$$

Fix $k\geq n_0$ and consider this inequality. There is a contradiction when $n$ goes to $+\infty$, since the right-hand side $2\left|\int_{\heis}\big((f_kQ_{\beta_n})\star m_{\beta_n}\big)\varphi_n\d\lambda_3\right|$ remains bounded by $C_4(p)\|f_k\|_{L^{\sigma,\tau}(\heis)}$, whereas the left-hand side $\left|\int_{\heis}Q_{\beta_n}\varphi_n\d\lambda_3\right|$ tends to $+\infty$.
\end{proof}

\begin{cor}\label{cor:bound_Lpq}
For all $p\in(2,\infty)$ and $q\in(1,\infty)$, there exist $C_{p,q}>0$ and $\beta_*(p,q)\in(0,1)$ such that for all $\beta\in(\beta_*(p,q),1)$, $\|Q_\beta\|_{L^{p,q}(\heis)}\leq C_{p,q}$.
\end{cor}

We now collect some estimates on the decay of $Q_{\beta}$ when $\beta$ is close to $1$.

\begin{thm}
There exist $C>0$ and $\beta_*\in(0,1)$ such that, for all $\beta\in(\beta_*,1)$ and all $(x,y,s)\in\heis$,
$$
|Q_{\beta}(x,y,s)|\leq \frac{C}{\rho(x,y,s)^2+1},
$$
where
$
\rho(x,y,s)=((x^2+y^2)^2+s^2)^{\frac{1}{4}}
$
is the distance from $(x,y,s)\in\heis$ to the origin.
\end{thm}

\begin{proof}
Let us first show that the $Q_{\beta}$ are uniformly bounded in $L^{\infty}(\heis)$ for $\beta\in(\beta_*,1)$, where $\beta_*$ is large enough.

Let $u\in\heis$. Applying Hölder's inequality \ref{holder} to the right hand side term,
\begin{align*}
|Q_{\beta}(u)|
	&=\left|\int_{v\in \heis}|Q_{\beta}|^2Q_{\beta}(v)m_{\beta}(v^{-1}u)\d\lambda_3(v)\right|\\
	&\leq \||Q_{\beta}|^2Q_{\beta} m_{\beta}(\cdot^{-1}u)\|_{L^1(\heis)}\\
	&\leq \||Q_{\beta}|^2Q_{\beta} \|_{L^{2,1}(\heis)}\|m_{\beta}\|_{L^{2,\infty}(\heis)}\\
	&\leq \|Q_{\beta} \|_{L^{6,3}(\heis)}^3\|m_{\beta}\|_{L^{2,\infty}(\heis)}.
\end{align*}
The conclusion follows from Corollary \ref{cor:bound_Lpq}.

For every $R>0$, we set $B_R=\{(x,y,s)\in\heis;\; \rho(x,y,s)\leq R\}$ and
$$
M(R)=\sup_{(x,y,s)\in B_R^c}|Q_{\beta}(x,y,s)|.
$$

Let $R>0$, $u\in B_R^c$. We cut
\begin{multline*}
|(|Q_{\beta}|^2Q_{\beta})\star m(u)|
	\leq 
	\Big|\int_{v\in B_{R/2}}|Q_{\beta}|^2Q_{\beta}(v)m_{\beta}(v^{-1}u)\d\lambda_3(v)\Big|\\
	+\Big|\int_{v\in B_{R/2}^c}|Q_{\beta}|^2Q_{\beta}(v)m_{\beta}(v^{-1}u)\d\lambda_3(v)\Big|.
\end{multline*}

On the one hand, if $v\in B_{R/2}$ then $uv^{-1}\in B_{R/2}^c$, so
$$
\left|\int_{v\in B_{R/2}}|Q_{\beta}|^2Q_{\beta}(v)m_{\beta}(v^{-1}u)\d\lambda_3(v)\right|
	\leq \frac{|c_\beta|}{R^2}\|Q_\beta\|_{L^{3}(\heis)}^3.
$$
Thanks to Theorem \ref{thm:bound_Lp}, one knows that up to increasing $\beta_*$, there exists some constant $C$ such that $|c_\beta|\|Q_\beta\|_{L^{3}(\heis)}^3\leq C$ for all $\beta\in(\beta_*,1)$.

On the other hand, applying Hölder's inequality \ref{holder} to the right hand side term,
\begin{align*}
\left|\int_{v\in B_{R/2}^c}|Q_{\beta}|^2Q_{\beta}(v)m_{\beta}(v^{-1}u)\d\lambda_3(v)\right|
	&\leq \||Q_{\beta}|^2 m_{\beta}(\cdot^{-1}u)\|_{L^1(B_{R/2}^c)}M\left(\frac{R}{2}\right)\\
	&\leq \||Q_{\beta}|^2 \|_{L^{2,1}(B_{R/2}^c)}\|m_{\beta}(\cdot^{-1}u)\|_{L^{2,\infty}(B_{R/2}^c)}M\left(\frac{R}{2}\right)\\
	&\leq \|Q_{\beta} \|_{L^{4,4}(B_{R/2}^c)}\|Q_{\beta} \|_{L^{4,4/3}(B_{R/2}^c)}\|m_{\beta}\|_{L^{2,\infty}(\heis)}M\left(\frac{R}{2}\right).
\end{align*}

Thanks to the convergence of $(Q_\beta)_\beta$ to $Q_+$ in $\dot{H}^1(\heis)$ as $\beta$ tends to $1$ and the Folland-Stein embedding $\dot{H}^1(\heis)\hookrightarrow L^4(\heis)$, the sequence $(Q_\beta)_\beta$ converges to $Q_+$ in $L^4(\heis)$ and therefore is tight in $L^4(\heis)$. Moreover, the norms $\|Q_{\beta}\|_{L^{4,4/3}(\heis)}$, for $\beta$ close to $1$, are bounded. Therefore, up to increasing $\beta_*$ again, one can choose $R_0>0$ such that
$$
\sup_{\beta\in(\beta_*,1)}\left(\|Q_{\beta} \|_{L^{4,4/3}(B_{R_0/2}^c)}\|m_{\beta}\|_{L^{2,\infty}(\heis)}\right)\times \|Q_{\beta} \|_{L^{4,4}(B_{R_0/2}^c)}\leq \frac{1}{8}.
$$

Then, for every $R\geq R_0$,
$$
\left|\int_{v\in B_{R/2}^c}|Q_{\beta}|^2Q_{\beta}(v)m_{\beta}(v^{-1}u)\d\lambda_3(v)\right|
	\leq \frac{1}{8}M\left(\frac{R}{2}\right).
$$

Combining the two estimates and applying them to $R=2^n$, $n\geq n_0$ so that $2^{n_0}\geq R_0$, we get
$$
M(2^n)\leq \frac{C}{4^n}+\frac{1}{8}M(2^{n-1}).
$$
Iterating, one knows that for all $n\geq n_0$,
\begin{align*}
M(2^n)
	&\leq C\sum_{k=0}^{n-n_0}\frac{1}{4^{n-k}}\frac{1}{8^k}+\frac{1}{8^{n-n_0+1}}M(2^{n_0-1})\\
	&\leq C4^{-n}\sum_{k=0}^{n-n_0}4^{-k}+8^{n_0+1}M(2^{n_0-1})8^{-n}\\
	&\leq (2C+8^{n_0+1}M(2^{n_0-1}))4^{-n}.
\end{align*}
Since $\rho(u)\sim 2^n$ for $2^n\leq \rho(u)\leq 2^{n+1}$, this completes the proof of the result.
\end{proof}

\begin{cor}\label{cor:bound_Hk}
For some  $\beta_*\in(0,1)$, for all $k\geq 1$, there exists $C_k>0$ such that for all $\beta\in(\beta_*,1)$,
$$
\|Q_\beta\|_{\dot{H}^k(\heis)}\leq C_k.
$$
\end{cor}

\begin{proof}
It is enough prove the first part of the claim for $k\in \N$. We proceed by induction on $k$. We already know that it is true for $k=1$ because
$$
\|Q_\beta\|_{\dot{H}^1(\heis)}^2\leq \frac{(-(\hlapl+\beta D_s)Q_\beta,Q_\beta)_{\dot{H}^{-1}(\heis)\times\dot{H}^1(\heis)}}{1-\beta}=\frac{I_\beta}{(1-\beta)^2},
$$
and $(\frac{I_\beta}{(1-\beta)^2})_\beta$ is bounded (cf. \ref{subsection:limit}).

The following additional assumption will be useful in the induction step. Up to increasing $\beta_*$, we can assume that the $Q_\beta$ are bounded in $L^6(\heis)$ and in $L^{\infty}(\heis)$ for $\beta\in(\beta_*,1)$.

Suppose now that the $Q_\beta$ are bounded in $\dot{H}^k(\heis)$ for 	an integer $k\geq 1$. Then by Leibniz' rule,  since $\hlapl =\frac{1}{4}(X^2+Y^2)$ for radial functions, with  $X=\partial_x+2y\partial_s$ and $Y=\partial_y-2x\partial_s$, there exist some coefficients $c_{\lambda}$ such that
\begin{align*}
-\hlapl^{k-1}(-\frac{\hlapl+\beta D_s}{1-\beta}Q_\beta)
	&=-\hlapl^{k-1}(|Q_\beta|^2Q_\beta)\\
	&=\sum_{|\lambda_1|+|\lambda_2|+|\lambda_3|= 2k-2}c_{\lambda}\partial^{\lambda_1}(Q_\beta) \partial^{\lambda_2}(Q_\beta) \partial^{\lambda_3}(\overline{Q_\beta}).
\end{align*}
The notation is similar as in $\R^N$, $\lambda_j$ being a finite sequence of letters $X$ and $Y$ of length $|\lambda_j|$, $\partial^{X}:=X$, $\partial^{Y}:=Y$. The following inequality can be easily proven via the Fourier transform~:
\begin{align*}
(-\hlapl^{k+1}Q_\beta,Q_\beta)_{\dot{H}^{-1}(\heis)\times\dot{H}^1(\heis)}
	&=(-\hlapl^{k}Q_\beta,-\hlapl Q_\beta)_{\dot{H}^{1}(\heis)\times\dot{H}^{-1}(\heis)}\\
	&\leq(-\hlapl^{k-1}(-\frac{\hlapl+\beta D_s}{1-\beta}Q_\beta),-\frac{\hlapl+\beta D_s}{1-\beta}Q_\beta)_{\dot{H}^{1}(\heis)\times\dot{H}^{-1}(\heis)}\\
	&\leq(-\hlapl^{k-1}(-\frac{\hlapl+\beta D_s}{1-\beta}Q_\beta),
	|Q_\beta|^2Q_\beta)_{\dot{H}^{1}(\heis)\times\dot{H}^{-1}(\heis)}.
\end{align*}
We replace the term on the left by the above sum. By integration by parts and Leibniz' rule again, we can manage so that the following indexes of derivation $\mu_i$ all have length less or equal than $(k-1)$~:
\begin{align*}
(-\hlapl^{k+1}Q_\beta,Q_\beta)_{\dot{H}^{-1}(\heis)\times\dot{H}^1(\heis)}
	&=\sum_{\substack{|\mu_1|+\dots+|\mu_6|= 2k-2,
	\\|\mu_1|,\dots,|\mu_6|\leq k-1}}
	c'_{\mu}\int_{\heis}\partial^{\mu_1}(Q_\beta)\dots \partial^{\mu_4}(Q_\beta) \partial^{\mu_5}(\overline{Q_\beta})\partial^{\mu_6}(\overline{Q_\beta}).
\end{align*}
We now apply Hölder's inequality with exponents $p_1,\dots,p_6\in(2,\infty)$ satisfying $\frac{1}{p_1}+\dots+\frac{1}{p_6}=1$, to be chosen later. Then, denoting $m_j=|\mu_j|$,
$$
\left|\int_{\heis}\partial^{\mu_1}(Q_\beta) \dots \partial^{\mu_4}(Q_\beta) \partial^{\mu_5}(\overline{Q_\beta})\partial^{\mu_6}(\overline{Q_\beta})\right|
	\leq \|Q_\beta\|_{\dot{W}^{m_1,p_1}(\heis)}\dots \|Q_\beta\|_{\dot{W}^{m_6,p_6}(\heis)}.
$$
Let us choose the $p_i$ appropriately. The aim is to use complex interpolation, and in particular the following relation between homogeneous Sobolev spaces (see e.g. \cite{Bergh1976}, Theorem 6.4.5, assertion (7))
$$
(L^q(\heis),\dot{H}^k(\heis))_{\theta}=\dot{W}^{m,p}(\heis)
$$
where $p,q\in(2,\infty)$,
$
m=(1-\theta)0+\theta k
$
and
$$
\frac{1}{p}=\frac{1-\theta}{q}+\frac{\theta}{2}.
$$
For example, we choose $\theta_i=\frac{m_i}{k}$ and $p_i$ such that
$$
\frac{1}{p_i}=\frac{1}{6k}+\frac{m_i}{2k}.
$$
Then
$$
0<\frac{1}{p_i}\leq\frac{1}{6k}+\frac{k-1}{2k}=\frac{1+3k-3}{6k}<\half
$$
so $p_i\in(2,\infty)$, and
$$
\frac{1}{p_1}+\dots+\frac{1}{p_6}=\frac{1}{k}+\frac{2k-2}{2k}=1.
$$
Moreover, this choice leads to the exponents
$$
q_i=\frac{6k}{1-m_i/k}.
$$
Since $0\leq m_i\leq k-1$,
$$
2<6k\leq q_i\leq 6k^2<\infty,
$$
we can therefore apply the interpolation result.

Since there is a finite number of terms in the sum, the boundedness of $Q_\beta$ in $L^6(\heis)$, in $L^{\infty}(\heis)$ and in $\dot{H}^k(\heis)$ for $\beta>\beta_*$ ensures that there exists $C_{k+1}>0$ such that for $\beta>\beta_*$,
$$
\|(-\hlapl)^{\frac{k+1}{2}}Q_\beta\|_{L^2(\heis)}\leq C_{k+1},
$$
so the $Q_\beta$ are bounded in $\dot{H}^{k+1}(\heis)$.
\end{proof}

\subsection{Invertibility of \texorpdfstring{$\L_{Q_\beta}$}{L_beta}}

For $\beta\in(-1,1)$ the linearized operator around $Q_\beta$ for the Schrödinger equation is
$$
\L_{Q_\beta} h=-\frac{\hlapl +\beta D_s}{1-\beta}h-2|Q_\beta|^2h-Q_\beta^2\overline{h},\quad h\in\dot{H}^1(\heis).
$$
We prove the invertibility of this operator on a space of finite co-dimension.

\begin{prop}\label{prop:estimates_Lbeta_invertible}
There exist a neighbourhood $\V$ of $Q_+$, $\beta_*\in(0,1)$ and some constant $c>0$ such that for all $\beta\in (\beta_*,1)$, for all $Q_\beta\in\Qbeta\cap\V$, and for all $h\in\dot{H}^1(\heis)$,
$$
\|\L_{Q_\beta} h\|_{\dot{H}^{-1}(\heis)}+|(h,\partial_sQ_+)_{\dot{H}^1(\heis)}|
	+|(h,iQ_+)_{\dot{H}^1(\heis)}|
	+|(h,Q_++2i\partial_sQ_+)_{\dot{H}^1(\heis)}|
	\geq c \|h\|_{\dot{H}^1(\heis)}.
$$
\end{prop}

\begin{proof}
Let $\beta\in(0,1)$ and $Q_\beta\in \Qbeta$. Let $h\in\dot{H}^1(\heis)$. We decompose $h=h^++h_{\perp}$ where $h^+\in \dot{H}^1(\heis)\cap V_0^+$ and $h_\perp=h-h^+\in \dot{H}^1(\heis)\cap \bigoplus_{(n,\pm)\neq (0,+)} V_n^{\pm}$.

We split $\L_{Q_\beta} h$ as
$$
\L_{Q_\beta} h=\L h^+-r_+-r_-+\L^-_{Q_\beta}h,
$$
where
$$
\L h^+=-\hlapl h^+-2\Pi_0^+(|Q_+|^2h^+)-\Pi_0^+(Q_+^2\overline{h^+}),
$$
$$
r_+=2\Pi_0^+((|Q_\beta|^2-|Q_+|^2)h^+)+\Pi_0^+((Q_\beta^2-Q_+^2)\overline{h^+}),
$$
$$
r_-=2\Pi_0^+(|Q_\beta|^2h_\perp)+\Pi_0^+(Q_\beta^2\overline{h_\perp}),
$$
and
$$
\L^-_{Q_\beta}h=-\frac{\hlapl +\beta D_s}{1-\beta}h_\perp
	-2(\id-\Pi_0^+)(|Q_\beta|^2h)
	-(\id-\Pi_0^+)(Q_\beta^2\overline{h}).
$$
We treat each term separately.

$\bullet$ Concerning $\L h^+$, thanks to Corollary \ref{cor:estimates_invertibility_L},
\begin{multline*}
\|\L h^+\|_{\dot{H}^{-1}(\heis)}
	+|(h^+,\partial_sQ_+)_{\dot{H}^1(\heis)}|
	+|(h^+,iQ_+)_{\dot{H}^1(\heis)}|
	+|(h^+,Q_++2i\partial_sQ_+)_{\dot{H}^1(\heis)}|\\
	\geq c\|h^+\|_{\dot{H}^1(\heis)}.
\end{multline*}
Since $\partial_sQ_+$, $iQ_+$ and $(Q_++2i\partial_sQ_+)$ are in $V_0^+$, we know that $(h^+,\partial_sQ_+)_{\dot{H}^1(\heis)}=(h,\partial_sQ_+)_{\dot{H}^1(\heis)}$, $(h^+,iQ_+)_{\dot{H}^1(\heis)}=(h,iQ_+)_{\dot{H}^1(\heis)}$ and $(h^+,Q_++2i\partial_sQ_+)_{\dot{H}^1(\heis)}=(h,Q_++2i\partial_sQ_+)_{\dot{H}^1(\heis)}$.

$\bullet$ Consider now $r_+$ and $r_-$. Let $K$ be the constant in the Folland-Stein embedding $\dot{H}^1(\heis)\hookrightarrow L^4(\heis)$
$$
\|g\|_{L^4(\heis)}\leq K\|g\|_{\dot{H}^1(\heis)}, \quad g\in \dot{H}^1(\heis).
$$

Since the sequence $(\|Q_\beta\|_{L^4(\heis)})_\beta$ is bounded by some constant $C_1$,
\begin{align*}
\|r_-\|_{L^{\frac{4}{3}}(\heis)}
	&\leq 3C_1^2\|h_{\perp}\|_{L^4(\heis)}\\
	&\leq 3KC_1^2\|h_{\perp}\|_{\dot{H}^1(\heis)}.
\end{align*}
and
\begin{align*}
\|r_+\|_{L^{\frac{4}{3}}(\heis)}
	&\leq 3\|Q_{\beta}-Q_+\|_{L^{4}(\heis)}(\|Q_\beta\|_{L^{4}(\heis)}+\|Q_+\|_{L^{4}(\heis)})\|h^+\|_{L^{4}(\heis)}\\
	&\leq 6C_1\|Q_{\beta}-Q_+\|_{L^{4}(\heis)}\|h^+\|_{L^{4}(\heis)}\\
	&\leq 6KC_1\|Q_{\beta}-Q_+\|_{L^{4}(\heis)}\|h^+\|_{\dot{H}^1(\heis)}.
\end{align*}
Let $\varepsilon>0$ to be determined later. There exists $\beta_*(\varepsilon)$ such that for $\beta>\beta_*(\varepsilon)$,
$$
\|Q_{\beta}-Q_+\|_{L^{4}(\heis)}\leq\varepsilon.
$$

We conclude by the dual embedding $L^{\frac{4}{3}}(\heis)\hookrightarrow \dot{H}^{-1}(\heis)$ that there exists a constant $C_2$ (independent of $\varepsilon)$) such that for all $\beta\in(\beta_*(\varepsilon),1)$,
\begin{align*}
\|r_+\|_{\dot{H}^{-1}(\heis)}+\|r_-\|_{\dot{H}^{-1}(\heis)}
	\leq C_2\|h_{\perp}\|_{\dot{H}^1(\heis)}+C_2\varepsilon\|h^+\|_{\dot{H}^1(\heis)}.
\end{align*}

$\bullet$ Finally, we focus on
$$
\L^-_{Q_\beta} h
	=-\frac{\hlapl +\beta D_s}{1-\beta}h_\perp
	-2(\id-\Pi_0^+)(|Q_\beta|^2h)
	-(\id-\Pi_0^+)(Q_\beta^2\overline{h}).
$$
In order to bound the $\dot{H}^{-1}$ norm of this term, we will use the fact that
\begin{align*}
\half\|\L^-_{Q_\beta} h\|_{\dot{H}^{-1}(\heis)}^2+\half\|h_\perp\|_{\dot{H}^{1}(\heis)}^2
	&\geq\|\L^-_{Q_\beta} h\|_{\dot{H}^{-1}(\heis)}\|h_\perp\|_{\dot{H}^{1}(\heis)}\\
	&\geq (\L^-_{Q_\beta} h,h_\perp)_{\dot{H}^{-1}(\heis)\times\dot{H}^{1}(\heis)}.
\end{align*}
On the one hand, by inequality \eqref{eq:normesV0perp},
$$
(-\frac{\hlapl +\beta D_s}{1-\beta}h_\perp,h_\perp)_{\dot{H}^{-1}(\heis)\times\dot{H}^{1}(\heis)}
	\geq \half\frac{1}{1-\beta}\|h_\perp\|_{\dot{H}^{1}(\heis)}^2.
$$
On the other hand,
\begin{align*}
\left|\left(2(\id-\Pi_0^+)(|Q_\beta|^2h)+(\id-\Pi_0^+)(Q_\beta^2\overline{h}),h_\perp\right)_{\dot{H}^{-1}(\heis)\times\dot{H}^{1}(\heis)}\right|
	&\leq 3C_1^2\|h\|_{L^4(\heis)}\|h_\perp\|_{L^4(\heis)}\\
	&\leq \varepsilon^2 \|h\|_{L^4(\heis)}^2+\frac{3C_1^2}{4\varepsilon^2}\|h_\perp\|_{L^4(\heis)}^2.
\end{align*}
To summarize,
$$
\half\|\L^-_{Q_\beta} h\|_{\dot{H}^{-1}(\heis)}^2+\half\|h_\perp\|_{\dot{H}^{1}(\heis)}^2
	\geq \half\frac{1}{1-\beta}\|h_\perp\|_{\dot{H}^{1}(\heis)}^2-\varepsilon^2 \|h\|_{L^4(\heis)}^2-\frac{3C_1^2}{4\varepsilon^2}\|h_\perp\|_{L^4(\heis)}^2,
$$
and by removing the squares appropriately,
\begin{align*}
\|\L^-_{Q_\beta} h\|_{\dot{H}^{-1}(\heis)}
	&\geq \sqrt{\frac{\beta}{1-\beta}}\|h_\perp\|_{\dot{H}^{1}(\heis)}-\sqrt{2}\varepsilon \|h\|_{L^4(\heis)}-\sqrt{\frac{3C_1^2}{2\varepsilon^2}}\|h_\perp\|_{L^4(\heis)}\\
	&\geq \sqrt{\frac{\beta}{1-\beta}}\|h_\perp\|_{\dot{H}^{1}(\heis)}-\sqrt{2}K\varepsilon \|h\|_{\dot{H}^1(\heis)}-\sqrt{\frac{3C_1^2}{2\varepsilon^2}}K\|h_\perp\|_{\dot{H}^1(\heis)}.
\end{align*}

$\bullet$ We conclude by combining all the estimates. Because of the orthogonality of the decomposition along the spaces $\dot{H}^{-1}(\heis)\cap V_n^\pm$ in $\dot{H}^{-1}(\heis)$,
$$
\|\L_{Q_\beta}h\|_{\dot{H}^{-1}(\heis)}^2
	=\|\L h^++r_++r_-\|_{\dot{H}^{-1}(\heis)}^2
	+\|\L_{Q_\beta}^- h\|_{\dot{H}^{-1}(\heis)}^2,
$$
so we can add up the estimates to get
\begin{align*}
\sqrt{2}\|\L_{Q_\beta} h\|_{\dot{H}^{-1}(\heis)}
		+|(h,\partial_sQ_+)_{\dot{H}^1(\heis)}|
		+|(h,iQ_+)_{\dot{H}^1(\heis)}|
		+|(h,Q_++2i\partial_sQ_+)_{\dot{H}^1(\heis)}|\\
	\geq c\|h^+\|_{\dot{H}^1(\heis)}-C_2\|h_{\perp}\|_{\dot{H}^1(\heis)}
		-C_2\varepsilon\|h^+\|_{\dot{H}^1(\heis)} 
	 	+\sqrt{\frac{\beta}{1-\beta}}\|h_\perp\|_{\dot{H}^{1}(\heis)}\\
	 -\sqrt{2}K\varepsilon \|h\|_{\dot{H}^1(\heis)}
	 	-\sqrt{\frac{3C_1^2}{2\varepsilon^2}}K\|h_\perp\|_{\dot{H}^1(\heis)}.
\end{align*}
The terms compensate as follows. Concerning $\|h^+\|_{\dot{H}^1(\heis)}$, fix $\varepsilon>0$ small enough in the sense that
$$
(C_2+\sqrt{2}K)\varepsilon <\frac{c}{2}.
$$
Then for all $\beta>\beta_*(\varepsilon)$,
\begin{align*}
\sqrt{2}\|\L_{Q_\beta} h\|_{\dot{H}^{-1}(\heis)}
	&+|(h,\partial_sQ_+)_{\dot{H}^1(\heis)}|
	+|(h,iQ_+)_{\dot{H}^1(\heis)}|
	+|(h,Q_++2i\partial_sQ_+)_{\dot{H}^1(\heis)}|\\
	&\geq \frac{c}{2}\|h^+\|_{\dot{H}^1(\heis)}
	 +\left(\sqrt{\frac{\beta}{1-\beta}}-\left(C_2+\sqrt{2}K\varepsilon+\sqrt{\frac{3C_1^2}{2\varepsilon^2}}K\right)\right)\|h_\perp\|_{\dot{H}^1(\heis)}.
\end{align*}
Let now $\beta_*\in(0,1)$ such that for all $\beta\in(\beta_*,1)$,
$$
\sqrt{\frac{\beta}{1-\beta}}
	\geq C_2+\sqrt{2}K\varepsilon+\sqrt{\frac{3C_1^2}{2\varepsilon^2}}+\frac{c}{2}.
$$
Then for all $\beta\in(\beta_*,1)$, 
\begin{align*}
\sqrt{2}\|\L_{Q_\beta} h\|_{\dot{H}^{-1}(\heis)}
	+|(h,\partial_sQ_+)_{\dot{H}^1(\heis)}|
	+|(h,iQ_+)_{\dot{H}^1(\heis)}|
	+|(h,Q_+ &+2i\partial_sQ_+)_{\dot{H}^1(\heis)}|\\
	&\geq \frac{c}{2}(\|h^+\|_{\dot{H}^1(\heis)}+\|h_\perp\|_{\dot{H}^1(\heis)})\\
	&\geq \frac{c}{2}\|h\|_{\dot{H}^1(\heis)}.
\end{align*}
\end{proof}

\subsection{Uniqueness of the traveling waves for \texorpdfstring{$\beta$}{beta} close to \texorpdfstring{$1^-$}{1}}

\begin{thm}\label{thm:description_Qbeta}
There exist $\beta_*\in(0,1)$ and a neighbourhood $\mathcal{V}$ of $Q_+$ in $\dot{H}^1(\heis)$ such that for all $\beta\in(\beta_*,1)$, there is a unique $Q_\beta\in \Qbeta\cap\mathcal{V}\cap (\partial_sQ_+,iQ_+,Q_++2i\partial_sQ_+)^{\perp,\dot{H}^1(\heis)}$.
Moreover,
\begin{enumerate}
\item for all $\beta\in(\beta_*,1)$,
$$
\Qbeta=\left\{T_{s_0,\theta,\alpha}Q_\beta:(x,y,s)\mapsto \e^{i\theta}\alpha Q_\beta(\alpha x,\alpha y,\alpha^2 (s+s_0));\quad (s_0,\theta,\alpha)\in \R\times\T\times\R_+^* \right\};
$$
\item for all $\gamma\in(0,\frac{1}{4})$ and all $k\in[1,+\infty)$, $\|Q_\beta-Q_+\|_{\dot{H}^k(\heis)}=\gdO((1-\beta)^\gamma)$; 
\item the map $\beta\in(\beta_*,1)\mapsto Q_\beta\in\dot{H}^1(\heis)$ is smooth, tends to $Q_+$ as $\beta$ tends to $1$, and its derivative $\dot{Q_\beta}$ is uniquely determined by
\begin{equation}\label{eq:Qpoint}
\begin{cases}
\L_{Q_\beta}(\dot{ Q_\beta})=-\frac{\hlapl +D_s}{(1-\beta)^2}Q_\beta
\\
\dot{Q_\beta}\in\dot{H}^{1}(\heis)\cap (\partial_sQ_+,iQ_+,Q_++2i\partial_sQ_+)^{\perp,\dot{H}^1(\heis)}
\end{cases}.
\end{equation}
\end{enumerate}
\end{thm}

\begin{proof}
$\bullet $ Fix any neighbourhood $\mathcal{V}$ of $Q_+$. We first prove the existence of a profile $Q_\beta\in \Qbeta\cap\mathcal{V}\cap (\partial_sQ_+,iQ_+,Q_++2i\partial_sQ_+)^{\perp,\dot{H}^1(\heis)}$ for $\beta$ close enough to $1$. For $\beta\in(0,1)$, we choose $Q_\beta\in\Qbeta$ arbitrarily. By combining Corollary \ref{cor:estimates_stability} with the fact that $\delta(Q_\beta)=\gdO((1-\beta)^{\half})$ from Lemma \ref{lem:delta_Qbeta}, we know that
$$
\inf_{(s_0,\theta,\alpha)\in \R\times\T\times\R_+^*}\|T_{s_0,\theta,\alpha}Q_\beta-Q_+\|_{\dot{H}^1(\heis)}
	=\gdO((1-\beta)^{\frac{1}{4}}).
$$
The same argument as in the proof of Corollary \ref{cor:estimates_stability}, based on the implicit function theorem, enables us to state that for $\beta$ close enough to $1$, one can choose $(s_\beta,\theta_\beta,\alpha_\beta)\in \R\times\T\times\R_+^*$ such that $\tilde{Q_\beta}:=T_{s_\beta,\theta_\beta,\alpha_\beta}Q_\beta\in\mathcal{V}$ and
$$
(\tilde{Q_\beta},\partial_sQ_+)_{\dot{H}^1(\heis)}
	=(\tilde{Q_\beta},iQ_+)_{\dot{H}^1(\heis)}
	=(\tilde{Q_\beta},Q_++2i\partial_sQ_+)_{\dot{H}^1(\heis)}
	=0.
$$
This gives the existence part of the result.

$\bullet$ We now prove uniqueness for some small neighbourhood $\mathcal{V}$ of $Q_+$. We first set $\mathcal{V}$ as the neighbourhood of $Q_+$ from Proposition \ref{prop:estimates_Lbeta_invertible}. Let $\beta\in(\beta_*,1)$, and fix two profiles $Q_{\beta}$ and $\tilde{Q_\beta}$ in $\Qbeta\cap \mathcal{V}\cap (\partial_sQ_+,iQ_+,Q_++2i\partial_sQ_+)^{\perp,\dot{H}^1(\heis)}$.
We define
$$
h:=Q_\beta-\tilde{Q_\beta}\in\dot{H}^1(\heis)\cap (\partial_sQ_+,iQ_+,Q_++2i\partial_sQ_+)^{\perp,\dot{H}^1(\heis)}.
$$
By subtracting the equations solved by $Q_\beta$ and $\tilde{Q_\beta}$, $h$ satisfies
$$
-\frac{\hlapl+\beta D_s}{1-\beta}h=2\Pi_0^+(|Q_\beta|^2h)+\Pi_0^+(Q_\beta^2\overline{h})+\gdO(\|h\|_{\dot{H}^1(\heis)}^2),
$$
so that
$$
\L_{Q_\beta}h=\gdO(\|h\|_{\dot{H}^1(\heis)}^2).
$$
Since $Q_\beta$ belongs to  the neighbourhood $\mathcal{V}$ from Proposition \ref{prop:estimates_Lbeta_invertible}, this means that for some constants $c>0$ and $C>0$,
$$
C\|h\|_{\dot{H}^1(\heis)}^2
	\geq \|\L_{Q_\beta} h\|_{\dot{H}^{-1}(\heis)}
	\geq c \|h\|_{\dot{H}^1(\heis)}.
$$
Up to reducing the neighbourhood $\mathcal{V}$, one can chose it small enough such that $h$ has to be the zero function.

$\bullet$ The description of the set $\Qbeta$ is then a direct consequence. Indeed, if $\beta\in(\beta_*,1)$, fix $U_\beta\in\Qbeta$. We know from the first point that $\beta_*$ is sufficiently close to $1$ to ensure the existence of $(s_\beta,\theta_\beta,\alpha_\beta)\in\R\times\T\times\R_+^*$ such that $T_{s_\beta,\theta_\beta,\alpha_\beta}U_\beta\in\mathcal{V}\cap (\partial_sQ_+,iQ_+,Q_++2i\partial_sQ_+)^{\perp,\dot{H}^1(\heis)}$. By the uniqueness point, $ T_{s_\beta,\theta_\beta,\alpha_\beta}U_\beta=Q_\beta$.

$\bullet$ We now show the convergence of $(Q_\beta)_\beta$ to $Q_+$ in $\dot{H}^k(\heis)$ for all $k\geq 1$. Applying Corollary \ref{cor:estimates_delta} to $(Q_\beta-Q_+)$, we know that for $\beta$ close to $1$,
$$
\delta(Q_\beta)\geq c\|Q_\beta-Q_+\|_{\dot{H}^1(\heis)}^2.
$$
But $\delta(Q_\beta)=\gdO((1-\beta)^{\half})$ from Proposition \ref{lem:delta_Qbeta}, therefore
$
\|Q_\beta-Q_+\|_{\dot{H}^1(\heis)}=\gdO((1-\beta)^{\frac{1}{4}}).
$

One can now deduce that for all $0<\gamma<\frac{1}{4}$, as $\beta$ goes to $1$, 
$$
\|Q_\beta-Q_+\|_{\dot{H}^k(\heis)}=\gdO((1-\beta)^{\gamma}).
$$
Indeed, the interpolation formula \cite{Bergh1976}
$$
(\dot{H}^m(\heis),\dot{H}^1(\heis))_{4\gamma}=\dot{H}^k(\heis)
$$
with $m\in\R$ chosen so that $k=(1-4\gamma)m+4\gamma$, leads to
$$
\|Q_\beta-Q_+\|_{\dot{H}^k(\heis)}
	\leq \|Q_\beta-Q_+\|_{\dot{H}^m(\heis)}^{1-4\gamma}
	\|Q_\beta-Q_+\|_{\dot{H}^1(\heis)}^{4\gamma},
$$
and it only remains to use the fact that $(Q_\beta-Q_+)_\beta$ is bounded in $\dot{H}^m(\heis)$ for $\beta$ close to $1$ (Corollary \ref{cor:bound_Hk}) and that $\|Q_\beta-Q_+\|_{\dot{H}^1(\heis)}^{4\gamma}=\gdO((1-\beta)^{\gamma})$ as $\beta$ goes to $1$.

$\bullet$ We now prove the last point of the theorem about the smoothness of the map $\beta\mapsto Q_\beta$. We first show that equation \eqref{eq:Qpoint} uniquely determines a function $\dot{Q_\beta}$ lying on the appropriate space
$$
W_1:=\dot{H}^{1}(\heis)\cap (\partial_sQ_+,iQ_+,Q_++2i\partial_sQ_+)^{\perp,\dot{H}^1(\heis)}.
$$
Define
$$
W_{-1}:=\dot{H}^{-1}(\heis)\cap (\partial_sQ_+,iQ_+,Q_++2i\partial_sQ_+)^{\perp,L^2(\heis)},
$$
and set
$$
F:(\beta,U)\in(\beta_*,1)\times W_1 \mapsto -\frac{\hlapl+\beta D_s}{1-\beta}U-|U|^2U\in\dot{H}^{-1}(\heis).
$$
Notice that $\partial_\beta F$ takes values in the space $W_{-1}$.
Indeed, the derivative $\partial_\beta F(\beta,U)$ is equal to
$$
\partial_\beta F(\beta,U)
	=-\frac{\hlapl+D_s}{(1-\beta)^2}U.
$$
In particular, since $Q_+,iQ_+,\partial_sQ_+$ and $i\partial_sQ_+$ belong to $\dot{H}^1(\heis)\cap V_0^+$, and since $-(\hlapl+D_s)$ vanishes on this space,
\begin{align*}
(\partial_\beta F(\beta,U),\partial_sQ_+)_{\dot{H}^{-1}(\heis)\times\dot{H}^1(\heis)}
	&=(\partial_\beta F(\beta,U),iQ_+)_{\dot{H}^{-1}(\heis)\times\dot{H}^1(\heis)}\\
	&=(\partial_\beta F(\beta,U),Q_++2i\partial_sQ_+)_{\dot{H}^{-1}(\heis)\times\dot{H}^1(\heis)}\\
	&=0,
\end{align*}
or equivalently $\partial_\beta F(\beta,U)\in W_{-1}$.

Consider $\L_{Q_\beta}$ as a self-adjoint operator on $L^2(\heis)$.
Then thanks to Proposition \ref{prop:estimates_Lbeta_invertible}, we get that $\textnormal{Ker}(\L_{Q_\beta})\subset \textnormal{Vect}_{\R}(\partial_sQ_+,iQ_+,Q_+2i\partial_s Q_+)$. Therefore,
$$
\textnormal{Im}(\L_{Q_\beta})
	=\textnormal{Ker}(\L_{Q_\beta})^{\perp,L^2(\heis)}
	=\dot{H}^{-1}(\heis)\cap\textnormal{Vect}_{\R}(\partial_sQ_+,iQ_+,Q_++2i\partial_sQ_+)^{\perp,L^2(\heis)},
$$
so $\textnormal{Im}(\L_{Q_\beta})=W_{-1}$.
This implies that $\L_{Q_\beta}$ is an isomorphism from $W_1$ to $W_{-1}$, with continuous inverse :
$$
\|\L_{Q_\beta}h\|_{\dot{H}^{-1}(\heis)}\geq c\|h\|_{\dot{H}^{1}(\heis)},\quad h\in W_1.
$$
In particular, $\partial_\beta F(\beta,Q_\beta)\in W_{-1}=\textnormal{Im}(\L_{Q_\beta})$, and by invertibility of $\L_{Q_\beta}$ from $W_1$ to $W_{-1}$, $\dot{Q_\beta}:= (\L_{Q_\beta})^{-1}(\partial_\beta F(\beta,Q_\beta))$ is uniquely determined and satisfies \eqref{eq:Qpoint}.

We now show that $\dot{Q_\beta}$ is a derivative of the map $\beta\in(\beta_*,1)\mapsto Q_\beta\in\dot{H}^1(\heis)$. Fix $\beta\in(\beta_*,1)$. For $\varepsilon>0$ small enough, $f_\varepsilon:=\frac{Q_{\beta+\varepsilon}-Q_\beta}{\varepsilon}-\dot{Q_\beta}$ is well defined. Moreover, since $(\beta+\varepsilon,Q_{\beta+\varepsilon})$ and $(\beta,Q_\beta)$ are both solution to the equation $F(\alpha,U)=0$, then
\begin{align*}
0
	&=F(\beta+\varepsilon,Q_{\beta+\varepsilon})-F(\beta,Q_\beta)\\
	&=F(\beta+\varepsilon,Q_{\beta+\varepsilon})-F(\beta,Q_{\beta+\varepsilon})
	+F(\beta,Q_{\beta+\varepsilon})-F(\beta,Q_\beta)\\
	&=\varepsilon\partial_\beta F(\beta+\varepsilon,Q_\beta)+\L_{Q_\beta}(Q_{\beta+\varepsilon}-Q_\beta)+\gdO(\varepsilon^2+\|Q_{\beta+\varepsilon}-Q_\beta\|_{\dot{H}^1(\heis)}^2).
\end{align*}
Actually, since $F$ is smooth in the $\beta$ variable,
\begin{align*}
0
	&=\varepsilon\partial_\beta F(\beta,Q_\beta)
	+\L_{Q_\beta}(Q_{\beta+\varepsilon}-Q_\beta)
	+\gdO(\varepsilon^2+\|Q_{\beta+\varepsilon}-Q_\beta\|_{\dot{H}^1(\heis)}^2).
\end{align*}
Replacing $\partial_\beta F(\beta,Q_\beta)$ by $\L_{Q_\beta}(\dot{Q_\beta})$, we get
$$
\L_{Q_\beta}(f_\varepsilon)
	=\gdO(\varepsilon+\frac{\|Q_{\beta+\varepsilon}-Q_\beta\|_{\dot{H}^1(\heis)}^2}{\varepsilon}).
$$
Since $f_\varepsilon\in W_1$, we know that $\|\L_{Q_\beta}(f_\varepsilon)\|_{\dot{H}^{-1}(\heis)}\geq c \|f_\varepsilon\|_{\dot{H}^1(\heis)}$. This implies that for some constant $C>0$,
$$
C(\varepsilon+\frac{\|Q_{\beta+\varepsilon}-Q_\beta\|_{\dot{H}^1(\heis)}^2}{\varepsilon})
	\geq c \|f_\varepsilon\|_{\dot{H}^1(\heis)}.
$$
But
$$
\|Q_{\beta+\varepsilon}-Q_\beta\|_{\dot{H}^1(\heis)}^2
	=\varepsilon^2\|f_\varepsilon+\dot{Q_\beta}\|_{\dot{H}^1(\heis)}^2
$$
so
$$
C\varepsilon(1+\|f_\varepsilon+\dot{Q_\beta}\|_{\dot{H}^1(\heis)}^2)
	\geq c\|f_\varepsilon\|_{\dot{H}^1(\heis)}.
$$
Letting $\varepsilon\to 0$, we get that $ \|f_\varepsilon\|_{\dot{H}^1(\heis)}\to 0$, so the map $\beta\mapsto Q_\beta$ is indeed $\classeC^1$ with derivative $\dot{Q_\beta}$. The smoothness follows from an implicit function theorem. Set
$$
\Phi:(\beta,U,V)\in(\beta_*,1)\times W_1\times W_{1}\mapsto \L_{Q_\beta} V-\partial_\beta F(\beta,U) \in W_{-1}.
$$
If $\beta\mapsto Q_\beta$ has regularity $\classeC^n$ for $\beta\in(\beta_*,1)$, then the function $\Phi$ is also $\classeC^n$. For fixed $\beta\in(\beta_*,1)$, $\Phi(\beta,Q_\beta,\dot{Q_\beta})=0$, and $\partial_VF(\beta,Q_\beta,\cdot)=\L_{Q_\beta}$, which is an isomorphism from $W_{1}$ to $W_{-1}$. Applying the implicit function theorem, there exists a $\classeC^n$ map $V$ defined on a neighbourhood of $(\beta,Q_\beta)$ in $(\beta_*,1)\times W_1$ and valued in $W_1$ such that $V(\beta,Q_\beta)=\dot{Q_\beta}$ and that on this neighbourhood,
$$
F(\beta,U,V(\beta,U))=0.
$$
In particular for $\beta'$ close to $\beta$, $F(\beta',Q_\beta',V(\beta',Q_{\beta'}))=0$ and since $\dot{Q_{\beta'}}$ is uniquely determined by \eqref{eq:Qpoint}, $\dot{Q_{\beta'}}=V(\beta',Q_{\beta'})$. The function $V$ being $\classeC^n$, supposing that $\beta\mapsto Q_\beta$ is $\classeC^n$ for some integer $n$, then $\beta\mapsto\dot{Q_\beta}$ is $\classeC^{n}$, and therefore $\beta\mapsto Q_\beta$ is $\classeC^{n+1}$.
\end{proof}


\section{Appendix~: proof of Lemma \ref{lem:calculation_proj}}\label{appendix:calculation_proj}

We establish an explicit formula for the orthogonal projections $P_0F_1$, $P_0F_2$ and $P_0F_3$ which are under integral form. Then, we estimate numerically $\langle P_0F_j,F_j\rangle_{L^2(\C_+)}$, $j=1,2,3$, in order to get Lemma \ref{lem:calculation_proj}.

$\bullet$ We  know that
$$
-\pi P_0(F_1)(s+it)
	=\int_{v\in \R_+}\int_{u\in\R}\frac{1}{(s-u+i(t+v))^2}\frac{1}{(u+i(v+1))}\frac{1}{\sqrt{u^2+(v+1)^2}}\d u\d v.
$$
Let us apply the change of variables $u=(v+1)\sh(y)$, $\d u=(v+1)\ch(y)\d y=\sqrt{u^2+(v+1)^2}\d y$. Then
$$
-\pi P_0(F_1)(s+it)
	=\int_{v\in \R_+}\int_{y\in\R}\frac{1}{(s-(v+1)\sh(y)+i(t+v))^2}\frac{1}{(\sh(y)+i)(v+1)}\d y\d v.
$$
We now apply the change of variables $x=\exp(y)$, $\d x=\exp(y)\d y$~:
\begin{align*}
-\pi P_0&(F_1)(s+it)\\
	&=\int_{v\in \R_+}\int_{y\in\R}\frac{8\e^{3y}}{\big(2(s+i(t+v))\e^y-(v+1)\e^{2y}+(v+1)\big)^2}\frac{1}{(\e^{2y}-1+2i\e^y)(v+1)}\d y\d v\\
	&=\int_{v\in \R_+}\int_{x\in\R_+}\frac{8x^2}{\big(2(s+i(t+v))x-(v+1)x^2+(v+1)\big)^2}\frac{1}{(x^2-1+2ix)(v+1)}\d x\d v.
\end{align*}
Thanks to Fubini's theorem, one can exchange the integral signs so that
\begin{align*}
-\pi P_0&(F_1)(s+it)\\
	&=\int_{x\in\R_+}\frac{8x^2}{(x+i)^2}\int_{v\in \R_+}\frac{1}{\big(2(s+it)x-x^2+1+v(-x^2+2ix+1)\big)^2}\frac{1}{(v+1)}\d x\d v\\
	&=\int_{x\in\R_+}\frac{8x^2}{(x+i)^2(x-i)^4}\int_{v\in \R_+}\frac{1}{\big(\frac{x^2-2(s+it)x-1}{x^2-2ix-1}+v\big)^2}\frac{1}{(v+1)}\d v\d x.
\end{align*}

The residue formula implies the following result. For any rational function $R$ such that $\int_{\R_+}R(v)\d v$ is convergent, then
$$
\int_{\R_+}R(v)\d v=-\sum_{w\in\C}\Res_w(R(w)\log_0(w)),
$$
where $\log_0$ is the positive determination of the logarithm.
Here, we consider the rational function $R(v)=\frac{1}{\big(\frac{x^2-2zx-1}{x^2-2ix-1}+v\big)^2}\frac{1}{(v+1)}$, $z=s+it$. We fix $\lambda=\frac{x^2-2zx-1}{x^2-2ix-1}$.

Assume that $z\neq i$ so that $\lambda\neq 1$. The residues at the simple pole $-1$ and the double pole $-\lambda$ are equal to
\begin{align*}
\Res_{-1}(R(w)\log_0(w))
	=\left(\frac{1}{(\lambda+w)^2}\log_0(w)\right)\Big|_{w=-1}
	=\frac{1}{(\lambda-1)^2}i\pi
\end{align*}
and
\begin{align*}
\Res_{-\lambda}(R(w)\log_0(w))
	&=\frac{\d}{\d w}\left(\frac{1}{(w+1)}\log_0(w)\right)\Big|_{w=-\lambda}\\
	&=\left(\frac{1}{w(w+1)}-\frac{\log_0(w)}{(w+1)^2}\right)\Big|_{w=-\lambda}\\
	&=\frac{1}{\lambda(\lambda-1)}-\frac{\log_0(-\lambda)}{(\lambda-1)^2}.
\end{align*}
Remark that
$$
\lambda=1-2(z-i)\frac{x}{(x-i)^2},
\quad
\frac{1}{\lambda-1}=-\frac{1}{2}\frac{(x-i)^2}{x}\frac{1}{z-i},
\quad
\text{and}
\quad
\frac{1}{(\lambda-1)^2}=\frac{1}{4}\frac{(x-i)^4}{x^2}\frac{1}{(z-i)^2}.
$$
Therefore,
\begin{align*}
\Res_{-1}(R(w)\log_0(w))
	&=i\pi\frac{1}{4}\frac{(x-i)^4}{x^2}\frac{1}{(z-i)^2}
\end{align*}
and
\begin{align*}
\Res_{-\lambda}(R(w)\log_0(w))
	=-\frac{(x-i)^2}{x^2-2zx-1}\frac{(x-i)^2}{2x(z-i)}
	-\log_0\left(-1+\frac{2(z-i)x}{(x-i)^2}\right)\frac{(x-i)^4}{4x^2(z-i)^2}.
\end{align*}
Consequently,
\begin{align*}
\frac{8x^2}{(x+i)^2(x-i)^2}\int_{\R_+}R(v)\d v
	=&\frac{-2i\pi}{(x+i)^2(z-i)^2}
	+\frac{4x}{(x^2-2zx-1)(x+i)^2(z-i)}\\
	&+2\log_0\left(-1+\frac{2(z-i)x}{(x-i)^2}\right)\frac{1}{(x+i)^2}\frac{1}{(z-i)^2}.
\end{align*}

We can integrate every term of the right hand side.
First,
$$
\int_{x\in\R_+}\frac{-2i\pi}{(x+i)^2(z-i)^2}\d x
	=\frac{-2\pi}{(z-i)^2}.
$$
Then, an integration by parts leads to
\begin{align*}
\int_{x\in\R_+}
	\log_0	\Big(-1+\frac{2(z-i)x}{(x-i)^2}\Big)\frac{1}{(x+i)^2} \d x
	&=\pi+2(z-i)\int_{\R_+}\frac{1}{(x-i)(x^2-2zx-1)}\d x.
\end{align*}
We conclude that
\begin{dmath*}
-\pi P_0(F_1)(z)
	=\frac{-2\pi}{(z-i)^2}
	+\frac{4}{z-i}\int_{x\in\R_+}\frac{1}{x^2-2zx-1}\frac{x}{(x+i)^2}\d x
	+\frac{2}{(z-i)^2}\left(\pi+2(z-i)\int_{x\in\R_+}\frac{1}{(x-i)(x^2-2zx-1)}\d x\right)
%
	=\frac{4}{z-i}\int_{x\in\R_+}\frac{1}{x^2-2zx-1}\frac{2x^2+ix-1}{(x+i)^2(x-i)}\d x.
\end{dmath*}

We apply the residue formula to get an exact expression for $-\pi P_0(F_1)$. We consider the rational function $R(x)=\frac{1}{x^2-2zx-1}\frac{2x^2+ix-1}{(x+i)^2(x-i)}$. Fix
$$
x_{\pm}:=z\pm\sqrt{z^2+1}.
$$
Since $z\neq i$, the rational function $R$ admits three simple poles $x_+$, $x_-$ and $i$ and one double pole $-i$. We calculate the residue
\begin{align*}
\Res_{x_+}(R(w)\log_0(w))
	&=\frac{2x_+^2+ix_+-1}{(x_+-x_-)(x_++i)^2(x_+-i)}\log_0(x_+).
\end{align*}
The identities $x_+^2=2zx_++1$, $(x_++i)^2=2(z+i)x_+$, $x_+x_-=-1$ and $(x_+-i)(x_--i)=-2i(z-i)$ enable to simplify
\begin{align*}
\Res_{x_+}(R(w)\log_0(w))
	&=i\frac{(z+i)x_--2iz}{4(z^2+1)^{\frac{3}{2}}}\log_0(x_+).
\end{align*}
The same arguments lead to
\begin{align*}
\Res_{x_-}(R(w)\log_0(w))
	&=\frac{2x_-^2+ix_--1}{(x_--x_+)(x_-+i)^2(x_--i)}\log_0(x_-)\\
	&=-i\frac{(z+i)x_+-2iz}{4(z^2+1)^{\frac{3}{2}}}\log_0(x_-).
\end{align*}
Moreover, the residue at the pole $i$ is
$$
\Res_{i}(R(w)\log_0(w))
	=\frac{1}{-1-2zi-1}\frac{-4}{-4}\frac{i\pi}{2}
	=-\frac{\pi}{4(z-i)}.
$$
Finally, the residue a the double pole $-i$ is
\begin{align*}
\Res_{-i}(R(w)\log_0(w))
	=&\Big[\frac{1}{x(x^2-2zx-1)}\frac{2x^2+ix-1}{(x-i)}+\frac{4x+i}{(x^2-2zx-1)(x-i)}\log_0(x)\\
	&-\frac{2x^2+ix-1}{(x^2-2zx-1)(x-i)}\Big(\frac{1}{x-x_+}+\frac{1}{x-x_-}+\frac{1}{x-i}\Big)\log_0(x)\Big]_{x=-i},
\end{align*}
which simplifies as
$$
\Res_{-i}(R(w)\log_0(w))
	=-\frac{i}{2(z+i)}.
$$
%

We conclude that
\begin{align*}
\int_{x\in\R_+}\frac{1}{x^2-2zx-1}\frac{2x^2+ix-1}{(x+i)^2(x-i)}\d x
	=&-i\frac{(z+i)x_--2iz}{4(z^2+1)^{\frac{3}{2}}}\log_0(x_+)
	+i\frac{(z+i)x_+-2iz}{4(z^2+1)^{\frac{3}{2}}}\log_0(x_-)\\
	&+\frac{\pi}{4(z-i)}
	+\frac{i}{2(z+i)},
\end{align*}
therefore, as soon as $z\neq i$,
\begin{align*}
-\pi P_0(F_1)(z)
	=-i\frac{(z+i)x_--2iz}{(z-i)(z^2+1)^{\frac{3}{2}}}\log_0(x_+)
	+i\frac{(z+i)x_+-2iz}{(z-i)(z^2+1)^{\frac{3}{2}}}\log_0(x_-)
	+\frac{\pi}{(z-i)^2}
	+\frac{2i}{z^2+1},
\end{align*}
with
$$
x_{\pm}=z\pm\sqrt{z^2+1}.
$$
Note that $\log_0(x_{\pm})$ is well defined because if $z\pm\sqrt{z^2+1}$ is real, then $z$ should be real, which we exclude by assumption ($z\in\C_+$).

$\bullet$ We apply the same strategy for $(F_1+F_2)(z)=\frac{2i}{(z+i)^2}\frac{1}{|z+i|}$. We have
$$
-\frac{\pi}{2i} P_0(F_1+F_2)(s+it)
	=\int_{v\in \R_+}\int_{u\in\R}\frac{1}{(s-u+i(t+v))^2}\frac{1}{(u+i(v+1))^2}\frac{1}{\sqrt{u^2+(v+1)^2}}\d u\d v.
$$
With the change of variables $u=(v+1)\sh(y)$, $\d u=(v+1)\ch(y)\d y=\sqrt{u^2+(v+1)^2}\d y$, we get
$$
\frac{i\pi}{2}P_0(F_1+F_2(z))
	=\int_{v\in \R_+}\int_{y\in\R}\frac{1}{(s-(v+1)\sh(y)+i(t+v))^2}\frac{1}{(\sh(y)+i)^2(v+1)^2}\d y\d v.
$$
Now apply the change of variables $x=\exp(y)$, $\d x=\exp(y)\d y$~:
\begin{align*}
\frac{i\pi}{2}P_0 & (F_1+F_2(z))\\
	&=\int_{v\in \R_+}\int_{y\in\R}\frac{16\e^{4y}}{\big(2(s+i(t+v))\e^y-(v+1)\e^{2y}+(v+1)\big)^2}\frac{1}{(\e^{2y}-1+2i\e^y)^2(v+1)^2}\d y\d v\\
	&=\int_{v\in \R_+}\int_{x\in\R_+}\frac{16x^3}{\big(2(s+i(t+v))x-(v+1)x^2+(v+1)\big)^2}\frac{1}{(x^2-1+2ix)^2(v+1)^2}\d x\d v.
\end{align*}
Thanks to Fubini's theorem, one can exchange the integral signs so that
\begin{align*}
\frac{i\pi}{2}P_0 & (F_1+F_2(z))\\
	&=\int_{x\in\R_+}\frac{16x^3}{(x+i)^4}\int_{v\in \R_+}\frac{1}{\big(2(s+it)x-x^2+1+v(-x^2+2ix+1)\big)^2}\frac{1}{(v+1)^2}\d x\d v\\
	&=\int_{x\in\R_+}\frac{16x^3}{(x+i)^4(x-i)^4}\int_{v\in \R_+}\frac{1}{\big(\frac{x^2-2(s+it)x-1}{x^2-2ix-1}+v\big)^2}\frac{1}{(v+1)^2}\d v\d x.
\end{align*}

We apply the consequence of the residue formula to $R(v)=\frac{1}{\big(\frac{x^2-2zx-1}{x^2-2ix-1}+v\big)^2}\frac{1}{(v+1)^2}$, $z=s+it$. We fix $\lambda=\frac{x^2-2zx-1}{x^2-2ix-1}$ as in the first point.

Assume that $z\neq i$, therefore $\lambda\neq 1$. The residue at the double pole $-1$ is equal to
\begin{align*}
\Res_{-1}(R(w)\log_0(w))
	&=\frac{\d}{\d w}\left(\frac{1}{(\lambda+w)^2}\log_0(w)\right)\Big|_{w=-1}\\
	&=\left(\frac{1}{w(w+\lambda)^2}-2\frac{\log_0(w)}{(w+\lambda)^3}\right)\Big|_{w=-1}\\
	&=\frac{-1}{(-1+\lambda)^2}-2\frac{1}{(\lambda-1)^3}i\pi\\
	&=-\frac{(x-i)^4}{4x^2(z-i)^2}+\frac{i\pi}{4}\frac{(x-i)^6}{x^3(z-i)^3}.
\end{align*}
The residue at the double pole $-\lambda$ is
\begin{align*}
\Res_{-\lambda}(R(w)\log_0(w))
	&=\frac{\d}{\d w}\left(\frac{1}{(w+1)^2}\log_0(w)\right)\Big|_{w=-\lambda}\\
	&=\left(\frac{1}{w(w+1)^2}-2\frac{\log_0(w)}{(w+1)^3}\right)\Big|_{w=-\lambda}\\
	&=\frac{-1}{\lambda(\lambda-1)^2}-2\frac{-\log_0(-\lambda)}{(\lambda-1)^3}\\
	&=-\frac{(x-i)^6}{x^2-2xz-1}\frac{1}{4x^2(z-i)^2}-\frac{(x-i)^6}{4x^3(z-i)^3}\log_0\left(-1+\frac{2(z-i)x}{(x-i)^2}\right).
\end{align*}
Therefore,
\begin{align*}
\frac{16x^3}{(x+i)^4(x-i)^4} \int_{\R_+}R(v)\d v
	=&\frac{4x}{(x+i)^4(z-i)^2}-\frac{4i\pi(x-i)^2}{(x+i)^4(z-i)^3}
	+\frac{4(x-i)^2x}{(x+i)^4(x^2-2xz-1)(z-i)^2}\\
	&+\frac{4(x-i)^2}{(x+i)^4(z-i)^3}\log_0\left(-1+\frac{2(z-i)x}{(x-i)^2}\right).
\end{align*}

We now integrate again in $x$ to get that for all $z\neq i$,
\begin{align*}
\frac{i\pi}{2}P_0(F_1+F_2(z))
	=\frac{-2 (z-2 i)}{3 (z-i) (z+i)^2}
	-\frac{(1+2 i z)\left( \log_0 (z+\sqrt{z^2+1}) -\log_0 (z-\sqrt{z^2+1})\right)}{3 (z-i) (z+i)^2 \sqrt{z^2+1}}.
\end{align*}

$\bullet$ We do the last computation for $(F_1+F_3)(z)=\frac{-2i}{(z+i)(\overline{z}-i)}\frac{1}{|z+i|}=\frac{-2i}{|z+i|^3}$. We have
$$
-\frac{\pi}{-2i} P_0(F_1+F_3)(s+it)
	=\int_{v\in \R_+}\int_{u\in\R}\frac{1}{(s-u+i(t+v))^2}\frac{1}{(u^2+(v+1)^2)^{3/2}}\d u\d v.
$$
Apply the change of variables $u=(v+1)\sh(y)$, $\d u=(v+1)\ch(y)\d y=\sqrt{u^2+(v+1)^2}\d y$, then
$$
-\frac{i\pi}{2} P_0(F_1+F_3)(s+it)
	=\int_{v\in \R_+}\int_{y\in\R}\frac{1}{(s-(v+1)\sh(y)+i(t+v))^2}\frac{1}{\ch(y)^2(v+1)^2}\d y\d v.
$$
We now put $x=\exp(y)$, $\d x=\exp(y)\d y$~:
\begin{align*}
-\frac{i\pi}{2}  P_0 & (F_1+F_3)(s+it)\\
	&=\int_{v\in \R_+}\int_{y\in\R}\frac{16\e^{4y}}{\big(2(s+i(t+v))\e^y-(v+1)\e^{2y}+(v+1)\big)^2}\frac{1}{(\e^{2y}+1)^2(v+1)^2}\d y\d v\\
	&=\int_{v\in \R_+}\int_{x\in\R_+}\frac{16x^3}{\big(2(s+i(t+v))x-(v+1)x^2+(v+1)\big)^2}\frac{1}{(x^2+1)^2(v+1)^2}\d x\d v.
\end{align*}
Thanks to Fubini's theorem, one can exchange the integral signs so that
\begin{align*}
-\frac{i\pi}{2} P_0 & (F_1+F_3)(s+it)\\
	&=\int_{x\in\R_+}\frac{16x^3}{(x+i)^2(x-i)^2}\int_{v\in \R_+}\frac{1}{\big(2(s+it)x-x^2+1+v(-x^2+2ix+1)\big)^2}\frac{1}{(v+1)^2}\d x\d v\\
	&=\int_{x\in\R_+}\frac{16x^3}{(x+i)^2(x-i)^6}\int_{v\in \R_+}\frac{1}{\big(\frac{x^2-2(s+it)x-1}{x^2-2ix-1}+v\big)^2}\frac{1}{(v+1)^2}\d v\d x.
\end{align*}

We have already done the computation of the integral in the $v$ variable in the latter point. We proved that putting $R(v)=\frac{1}{\big(\frac{x^2-2(s+it)x-1}{x^2-2ix-1}+v\big)^2}\frac{1}{(v+1)^2}$,
\begin{align*}
\frac{16x^3}{(x+i)^4(x-i)^4} \int_{\R_+}R(v)\d v
	=&\frac{4x}{(x+i)^4(z-i)^2}-\frac{4i\pi(x-i)^2}{(x+i)^4(z-i)^3}
	+\frac{4(x-i)^2x}{(x+i)^4(x^2-2xz-1)(z-i)^2}\\
	&+\frac{4(x-i)^2}{(x+i)^4(z-i)^3}\log_0\left(-1+\frac{2(z-i)x}{(x-i)^2}\right).
\end{align*}
Therefore,
\begin{align*}
\frac{16x^3}{(x+i)^2(x-i)^6}\int_{v\in \R_+}R(v)\d v
	= & \frac{4x}{(x+i)^2(x-i)^2(z-i)^2}
	-\frac{4i\pi}{(x+i)^2(z-i)^3}\\
	&+\frac{4x}{(x+i)^2(x^2-2xz-1)(z-i)^2}\\
	&+\frac{4}{(x+i)^2(z-i)^3}\log_0\left(-1+\frac{2(z-i)x}{(x-i)^2}\right).
\end{align*}

We now integrate again in $x$ to get that for all $z\neq i$,
\begin{align*}
-\frac{i\pi}{2}P_0(F_1+F_3)(z)
	=\frac{2 (z+2 i)}{(z-i)^2 (z+i)}
	+\frac{(1-2 i z) \left(\log_0 (z+\sqrt{z^2+1})-\log_0 (z-\sqrt{z^2+1})\right)}{(z-i)^2 (z+i) \sqrt{z^2+1}}
\end{align*}

$\bullet$ Now we can compute numerically $\langle P_0F_j,F_j\rangle_{L^2(\C_+)}$, $j=1,2,3$, the error estimate for every term can be chosen almost arbitrarily now that we know $P_0F_j$.

We set $\varepsilon=10^{-10}$ and we deduce
$$
|\langle \pi P_0F_1,F_1\rangle_{L^2(\C_+)}-2|\leq \varepsilon,
$$
$$
|\langle \pi P_0F_2,F_2\rangle_{L^2(\C_+)}-\frac{10}{9}|\leq \varepsilon
$$
and
$$
|\langle \pi P_0F_3,F_3\rangle_{L^2(\C_+)}-0.1303955989|\leq \varepsilon.
$$

\newpage

\bibliography{mybib}{}
\bibliographystyle{abbrv}

\Addresses
\end{document}